\pdfoutput=1

\documentclass{amsart}
\usepackage[utf8]{inputenc}
\usepackage{amsmath, amssymb, amsthm, graphics, comment, bm, mathtools, mathrsfs}
\usepackage[normalem]{ulem}
\usepackage{diagmac2, textcomp, color, tikz}
\usetikzlibrary{shapes.geometric, arrows}

\usepackage[alphabetic]{amsrefs}
\usepackage[all,cmtip,color,matrix,arrow]{xy}
\usepackage{yfonts}
\usepackage{stmaryrd}
\usepackage{enumerate}
\usepackage{hyperref}
\hypersetup{
colorlinks=true,
linkcolor=[RGB]{0,0,204},
filecolor=magenta,      
urlcolor=cyan,
citecolor=[RGB]{0,204,0},
}
\usepackage[capitalize]{cleveref}
\usepackage[mathscr]{euscript}
\usepackage{quiver}
\usepackage{geometry}
\usepackage{mathscinet}
\usepackage[activate={true, nocompatibility}, final, tracking=true, kerning=true, spacing=true, factor=1100, stretch=10, shrink=10]{microtype}
\microtypecontext{spacing=nonfrench}
\usepackage{float}

\calclayout

\newcommand{\bA}{\mathbb{A}}
\newcommand{\bC}{\mathbb{C}}

\newcommand{\bG}{\mathbb{G}}

\newcommand{\bP}{\mathbb{P}}
\newcommand{\bQ}{\mathbb{Q}}

\newcommand{\bZ}{\mathbb{Z}}
\newcommand{\bV}{\mathbb{V}}

\newcommand{\sC}{\mathscr{C}}

\newcommand{\oH}{\operatorname{H}}

\newcommand{\mtc}[1]{\mathcal{}{#1}}
\newcommand{\mb}[1]{\mathbb{#1}}

\newcommand{\sP}{\mathscr{P}}
\newcommand{\sS}{\mathscr{S}}

\newcommand{\sX}{\mathscr{X}}

\newcommand{\cA}{\mathcal{A}}
\newcommand{\cB}{\mathcal{B}}
\newcommand{\cC}{\mathcal{C}}
\newcommand{\cD}{\mathcal{D}}
\newcommand{\cE}{\mathcal{E}}

\newcommand{\cI}{\mathcal{I}}
\newcommand{\cK}{\mathcal{K}}
\newcommand{\cL}{\mathcal{L}}
\newcommand{\cM}{\mathcal{M}}

\newcommand{\cO}{\mathcal{O}}
\newcommand{\cP}{\mathcal{P}}

\newcommand{\cS}{\mathcal{S}}
\newcommand{\cU}{\mathcal{U}}
\newcommand{\cW}{\mathcal{W}}
\newcommand{\cX}{\mathcal{X}}
\newcommand{\cY}{\mathcal{Y}}

\newcommand{\sY}{\mathscr{Y}}

\newcommand{\ord}{\mathrm{ord}}

\newcommand{\bmu}{\bm{\mu}}
\newcommand{\Coker}{\operatorname{Coker}}

\newcommand{\Ext}{\operatorname{Ext}}
\newcommand{\ext}{\operatorname{ext}}
\DeclareMathOperator{\height}{ht}

\newcommand{\lct}{\operatorname{lct}}

\DeclareMathOperator{\jdeg}{jdeg}

\newcommand{\spec}{\operatorname{Spec}}

\newcommand{\vol}{\operatorname{vol}}
\newcommand{\Sym}{\operatorname{Sym}}
\newcommand{\Hom}{\operatorname{Hom}}

\newcommand{\sm}{\operatorname{sm}}

\newcommand{\Pic}{\operatorname{Pic}}

\newcommand{\Aut}{\operatorname{Aut}}
\newcommand{\PGL}{\operatorname{PGL}}

\newcommand{\Id}{\operatorname{Id}}

\newcommand{\od}{\operatorname{d}}
\newcommand{\KSBA}{\operatorname{KSBA}}
\newcommand{\KSB}{\operatorname{KSB}}
\newcommand{\Proj}{\operatorname{Proj}}

\DeclareMathOperator{\wt}{wt}

\DeclareMathOperator{\lc}{lc}
\DeclareMathOperator{\nod}{nod}
\DeclareMathOperator{\reg}{reg}
\DeclareMathOperator{\cusp}{cusp}
\DeclareMathOperator{\univ}{univ}

\newcommand{\leqor}{\underset{{\scriptscriptstyle (}-{\scriptscriptstyle )}}{<}}

\newcommand{\red}{{\mathrm{red}}}

\DeclareDocumentCommand
\Ell{O{\ensuremath{c(\epsilon)}}O{n}}{\cE\ell\ell_{#2,#1}}
\newcommand{\MKSB}{\cM^{\mathrm{KSB}}}
\newcommand{\MKSBA}{\cM^{\mathrm{KSBA}}}
\DeclareDocumentCommand\Psell{O{n}}{\cP_{#1}}
\DeclareDocumentCommand\Wei{O{n}}{\cW_{#1}}
\DeclareDocumentCommand\WeiStrat {O{k} O{n}} {\cS_{#2,#1}}

\newcommand{\tWei}{\widetilde{\cW}_n}
\newcommand{\tWeik}{\widetilde{\cW}_{n-m,m}}

\newcommand{\norm}{{\nu}}

\newtheorem{theorem}{Theorem}[section]
\newtheorem{Teo}[theorem]{Theorem}

\newtheorem{Lemma}[theorem]{Lemma}
\newtheorem{Oss}[theorem]{Observation}
\newtheorem*{Oss'}{Observation}
\newtheorem{Cor}[theorem]{Corollary}
\newtheorem{Prop}[theorem]{Proposition}

\theoremstyle{definition}
\newtheorem{Algorithm}[theorem]{Algorithm}
\newtheorem{Construction}[theorem]{Construction}
\newtheorem{EG}[theorem]{Example}
\newtheorem{Def}[theorem]{Definition}
\newtheorem{Notation}[theorem]{Notation}
\newtheorem{Remark}[theorem]{Remark}
\newtheorem{context}[theorem]{Context}

\begin{document}
\title{Moduli of elliptic surfaces of Kodaira dimension one \newline fibered over rational curves}

\author{Dori Bejleri, Josiah Foster, Andres Fernandez Herrero,\\
Giovanni Inchiostro, Svetlana Makarova, and Junyan Zhao}

\begin{abstract}
    In this article, we study an infinite sequence of irreducible components of Koll\'{a}r--Shepherd-Barron (KSB-) moduli spaces of surfaces of arbitrarily large volumes, and describe the combinatorics of the intersections of the irreducible components of the surfaces parametrized by  the boundary. Moreover, we describe the stable reduction steps in finding the KSB-limits in an explicit combinatorial way. Our main tool are the techniques of wall-crossing for Kollár--Shepherd-Barron--Alexeev (KSBA-) moduli and twisted stable maps. 
\end{abstract}

\maketitle
\tableofcontents


\section{Introduction}

The moduli spaces of smooth surfaces of general type admit Koll\'ar--Shepherd-Barron (abbv. KSB) compactifications, which allow smooth objects to degenerate to certain singular surfaces satisfying \emph{KSB-stability} (\cite{KSB}). 
The KSB-stable surfaces of a fixed volume $v$ are parametrized by a proper Deligne-Mumford stack $\MKSB_v$, whose coarse space is known as the KSB moduli space. More generally, for surface pairs of log general type $(X, cD)$ with $c \in \mathbb{Q} \cap (0,1]$, there is a Koll\'ar--Shepherd-Barron--Alexeev (abbv. KSBA) moduli space $\MKSBA_{c,v}$ parametrizing \emph{KSBA-stable pairs}. 
Classical questions about the classification of surfaces can be translated, via the moduli theory, to questions about the geometry of the corresponding moduli spaces. It is therefore an interesting and important endeavor to gain an explicit understanding of irreducible components of this moduli problem.

Besides the classical case of the moduli of curves (see \cite{DM69}), there are currently few instances in the literature where the geometry of an irreducible component of the KSB moduli spaces of varieties is completely understood. Even fewer examples are known to admit an explicit combinatorial description of the singular varieties parametrized by the boundary of their respective components of the moduli space (as opposed to the KSBA moduli of pairs where substantially more examples have been computed).
The first examples in dimension higher than one, where the boundary was fully understood, include quotients or special covers of a product of two curves (see \cites{vO06,Liu12,Rol10}), followed by the Campedelli surfaces and Burniat surfaces; see \cite{AP23}. These are surfaces which admit a cover to some toric del Pezzo surface $X_0$ branched along a line arrangement $\sum L_i$, so one can reduce the study to the description of moduli spaces of pairs $(X_0,c\sum L_i)$, where $c>0$ is a (rational) coefficient which makes the covering map crepant. The boundary of components of the moduli space has also been partially computed in some other examples (e.g. \cites{FPR, FPRR, GPSZ}).

In this paper, we provide a novel example. We study an infinite sequence of irreducible components $\overline{\cP}_n$ of the KSB-moduli spaces of stable surfaces of volume $\frac{(n-2)^2}{n}$, depending on a parameter $n\in\mb{Z}_{\geq3}$. Every point in $\overline{\cP}_n$ represents a singular(!) 
\textit{pseudo-elliptic surface} over a rational curve (not necessarily irreducible), and a general such surface contains a unique quotient singularity of type $\frac{1}{n}(1,1)$. Moreover, we develop an explicit combinatorial description of the irreducible components of the surfaces parametrized by the boundary of $\overline{\cP}_n$, which parametrizes surfaces with worse singularities than an isolated $\frac{1}{n}(1,1)$, and yields a boundary stratification of $\overline{\cP}_n$ indexed by certain decorated graphs.

\subsubsection*{Moduli of Weierstrass fibrations} 
The objects we study in this paper are elliptic surfaces of Kodaira dimension $1$, fibered over a rational curve with a section $S$ such that $S^2=-n$. 
They are projective surfaces $X$ which admit a proper equidimensional morphism $f \colon X\rightarrow C$ to a rational curve $C$ (not necessarily irreducible) such that each fiber is a connected curve of arithmetic genus $1$.
A general such surface $(f \colon X\rightarrow C,S)$ satisfies that $X$ is smooth, $C\simeq \mb{P}^1$, and the fibers of $f$ are either smooth or nodal. The number $n \geq 0$ is a fundamental invariant of the elliptic surface called the \emph{height}. 
The moduli of elliptic surfaces was first constructed using GIT by Miranda \cite{Miranda}. It parametrizes Weierstrass fibrations (see Definition \ref{def:weierstrass}), which are elliptic surfaces with integral fibers such that the total space has log canonical singularities. We denote by $\cW^{\lc}_n$ (resp. $\cW^{\min}_n$) the moduli stacks parametrizing such Weierstrass fibrations with at worst log canonical singularities (resp. Weierstrass fibrations with at worst canonical singularities). \\

The main observation is the following.
Given a general elliptic surface of Kodaira dimension $1$ over $\mb{P}^1$ $(f \colon X\rightarrow \mb{P}^1,S)$ with a section $S$ such that $S^2=-n$, there are two natural constructions to perform: 

\begin{itemize}
    \item (\textit{KSB(A) moduli}). Considering $(X,cS)$ as a pair of log general type gives us a point of a KSBA moduli space $\MKSBA_{c,v}$ for certain $c$. On the other hand, contracting the negative section $S$ results in a surface $Y$ with klt singularities, and the canonical divisor $K_Y$ is $\mb{Q}$-Cartier and ample of volume $v=\frac{(n-2)^2}{n}$. Such a surface $Y$, called a \emph{pseudo-elliptic surface}, is represented by a point in the KSB-moduli stack $\cM^{\KSB}_{v}$ (see Definition \ref{def:ksb}).
    \item (\textit{Twisted stable maps}). Since the fibers of $f$ are connected curves of arithmetic genus $1$ with at worst nodal singularities, and the section $S$ intersects each fiber at a smooth point of the fiber, then by the universality of $\overline{\cM}_{1,1}$, one obtains a natural morphism $C\rightarrow \overline{\cM}_{1,1}$.
    Such a morphism is parametrized by the moduli stack of twisted stable maps $\cK_n:=\overline{\cK}_{0,0}(\overline{\cM}_{1,1},n)$ (see Definition \ref{def_tsm_stability}), which was first introduced in \cite{AV_compactifying}.
    
\end{itemize}

We establish a link between the KSB moduli stacks and the moduli of twisted stable maps, using the KSBA moduli stacks $\MKSBA_{c,v}$ of log pairs $(X,cS)$ (see Definition \ref{definition: KSBA-stability}) as a bridge, where $0<c<\frac{n-2}{n}$ is a rational coefficient. Now we state our first main theorem.

\begin{theorem}[= \Cref{proposition: Phi defines iso of W with KSBA} + \Cref{prop: existence of the morphism} + \Cref{thm_proof_psi_n_isom} + \Cref{thm: irreducibility of K_n} + \Cref{proposition: maps between compactifications}]
\label{thm:main} \quad \newline
    Let $n\geq 3$ be an integer with $n\neq 4$, and let $0<\epsilon\ll1$ be a rational number. Set $v:=\frac{(n-2)^2}{n}$, $c(\epsilon):=\frac{n-2}{n}-\epsilon$, and $v(\epsilon):=v - n\epsilon$. Then the following statements hold.
\begin{enumerate}
    \item
    \label{item: main: open immersion of Wn}
    For any $0 < \epsilon < \frac{n-2}{n}$, there is a natural morphism
    $$ \Phi_{n,\epsilon}\ \colon\ \cW_n^{\lc}
  \  \longrightarrow \ \MKSBA_{c(\epsilon),v(\epsilon)},  
    \quad
    (f \colon X\to C, S)\ \mapsto\  \big(X, c(\epsilon)S\big),$$
    which is an open immersion. 
    \item
    \label{item: main: open immersion of En}    
    Let $\cE_n$ (resp. $\cE_n^{\lc}$) be the image of $\cW^{\min}_n$ (resp. $\cW_n^{\lc}$) under $\Phi_{n,\epsilon}$. Then there is a natural morphism
    $$\Psi_n \ \colon\ \cE_n^{\lc}\ \longrightarrow \ \MKSB_{v},\quad \big(X, c(\epsilon)S\big)\ \mapsto \ Y ,$$ which is an open immersion.
    Here, the surface $Y$ is obtained from $X$ by contracting the section $S$.
    \item
    \label{item: main: birationality of Kn and En}
    The stack $\cK_n$ of twisted stable maps is irreducible and proper.
    Its normalization $\cK_n^{\nu}$ admits a natural birational morphism
    $ \bar{\Phi}_{n,\epsilon} \ \colon\ \cK_n^{\nu}\ \longrightarrow \ \overline{\cE}_{n,\epsilon}$, where $\overline{\cE}_{n,\epsilon}$ denotes the normaliztion of the closure of $\mathcal{E}_n \subset \MKSBA_{c(\epsilon), v(\epsilon)}$.
\end{enumerate}
In particular, the following are all birational to each other, for each $n\geq 3$, $n\neq 4$:
the moduli stack $\cK_n$, the moduli stack $\cW_n^{\lc}$, and
the irreducible components of $\MKSBA_{c(\epsilon),v(\epsilon)}$ and $\MKSB_v$ which generically parametrize (pseudo)elliptic surfaces of height $n$ over a rational curve.
\end{theorem}

Let $\overline{\cP}_n$ be the normalization of the closure of the image $\cP_n$ of $\Psi_n$. Then each $\overline{\cP}_n$ and $\overline{\mathcal{E}}_{n, \epsilon}$ are irreducible components of their respective KSB- and KSBA-moduli (see \Cref{coroll: overline P is irreducible} and \Cref{cor_the_moduli_is_smooth}). The compactifications $\overline{\mathcal{E}}_{n, \epsilon}$ are related by wall-crossing morphisms as $\epsilon$ varies, and we can view $\overline{\cP}_n$ as the moduli space at the wall $\epsilon = 0$ \cite{AB, inchiostro2020moduli, ascher2021wall, meng2023mmp}. 

We now summarize all the moduli stacks that we have introduced in one diagram.
\[\begin{tikzcd}
    &&& {\cW^{\min}_n} \\
    {\cK^\nu_n} &&& {\cW^{\lc}_n} \\
    {\cK_n} && {\overline \cE_{n,\epsilon}} && {\overline \cP_n} \\
    && \MKSBA_{c(\epsilon),v(\epsilon)} && \MKSB_v
    \arrow[hook, from=1-4, to=2-4]
    \arrow[from=2-1, to=3-1]
    \arrow["{\bar{\Phi}_n}", "{\text{birat.}}"', from=2-1, to=3-3]
    \arrow["{\Phi_{n,\epsilon}}"', from=2-4, to=3-3]
    \arrow["{\Psi_n}", from=2-4, to=3-5]
    \arrow[dashed, from=3-1, to=3-3]
    \arrow["\bar{\Psi}_{n, \epsilon}", dashed, from=3-3, to=3-5]
    \arrow[hook, from=3-3, to=4-3]
    \arrow[hook, from=3-5, to=4-5]
\end{tikzcd}\]

To study the surfaces parametrized by the boundary of $\overline{\cP}_n$, we take a morphism from $\spec R$, where $R$ is a DVR, to the stack of twisted stable maps $\cK_n$, and run the (relative) minimal model program (abbv. MMP) for the family pulled back from the universal family over $\overline{\cM}_{1,1}$.
We can associate a graph $\Gamma$, called a \emph{sliced tree}, to each elliptic surface $(f \colon X\rightarrow C,S)$ derived from a twisted stable map (i.e. \emph{tsm-stable elliptic surface}) in such a way that each vertex of $\Gamma$ represents an irreducible component of $X$ and each edge between two vertices means that the two corresponding components intersect (see Definition \ref{def_assignment}).
Furthermore, such a sliced tree includes the data of certain numerical decorations for its vertices (which encode the degree of the $j$-map of each component) and decorations for its edges (which encode the gluing data for any two given components).
We show that the MMP steps (i.e. stable reduction) can be described as a pruning process of the associated graphs (\Cref{prop_pruning_tree}).
The main operations of the pruning process include cutting down edges and adding \emph{klt-markings} and \emph{lc-markings}. This yields a class of decorated graphs we call \emph{stable pruned trees} (\Cref{def_pruned_tree}) which encode the combinatorics of a KSBA-stable limit of a family of elliptic surfaces.  

{\begin{theorem}[=\Cref{algorith: stable reduction} + \Cref{prop_pruning_tree} + \Cref{thm:stratificationE_n}]\label{main:combinatorics} The stable reduction in the moduli spaces $\overline{\cE}_{n,t}$ is given by \cref{algorith: stable reduction} for $0<\frac{n-2}{n}-t\ll 1$, which computes the combinatorics of the irreducible components of the KSBA-stable limit.
Moreover, stable reduction corresponds to a pruning process of the graph associated to the tsm-stable elliptic surfaces. In particular, there is a stratification of $\overline{\mathcal{E}}_{n,t}$ indexed by stable pruned trees whose strata parametrizes KSBA-stable elliptic surfaces which are glued together according to the combinatorics of the graph. 
\end{theorem}}

Thus the boundary of the compactification $\overline{\mathcal{E}}_{n,t}$ for $0 \ll t < \frac{n-2}{n}$ has a particularly nice combinatorial description. However, this structure can be passed to $\overline{\cP}_n$ and $\overline{\cE}_{n,t}$ for other $t$ via the following wall-crossing theorem. 

{\begin{theorem}[=\Cref{cor_amm_moduli_spaces_are_isom} + \Cref{thm:boundary_P_n}]\label{main:wall}
The moduli spaces $\overline{\cE}_{n,t}$ are isomorphic for any $0<t<\frac{n-2}{n}$ and the birational map $\bar{\Psi}_{n,t} \colon \overline{\cE}_{n,t}\to \overline{\cP}_n$ is an isomorphism for all $0 < t < \frac{n-2}{n}$. Moreover, $\overline{\cP}_n$ inherits a boundary stratification indexed by stable pruned trees of height $n$.
\end{theorem}}
    
An immediate corollary of the above two theorems is the following.

\begin{Cor}
    Let $n\geq 3$ be an integer with $n\neq 4$, and $v:=\frac{(n-2)^2}{n}$. Then $\overline{\cP}_n$ is an irreducible component of $\MKSB_v$, is proper and of dimension $10n-2$.
    The interior of $\overline{\cP}_n$ parametrizes pseudo-elliptic surfaces with an isolated $\frac{1}{n}(1,1)$-singularity, and the boundary of $\overline{\cP}_n$ parametrizes pseudo-elliptic surfaces whose associated graph is a pruned tree of height $n$.
\end{Cor}

Finally, we state a more explicit result in the special case when $n = 3$, i.e. $S^2 = -3$: then the cohomology groups encoding deformation theory can be computed explicitly (Section \ref{section: n=3}), and the combinatorics is relatively simple.

\begin{theorem}[=\Cref{cor_KSB_limits_have_at_most_6_vertices} + \Cref{prop_K3_smooth_and_of_dim_28} + \Cref{thm:embedding_3}]\label{main:n=3}
Let $n=3$ and $X$ be a surface in $\overline{\cP}_3$.
\begin{enumerate}
    \item The pruned tree $\Pi$ associated to $X$ is a chain with at most six vertices.
    \item If $X$ is in the interior $\cP_3$, then one has $h^1(X,T_X)=28$ and $h^2(X,T_X)=0$. In particular the interior of $\cP_3$ is an irreducible smooth stack of dimension $28$.
\end{enumerate}

\end{theorem}

\begin{Remark}
    We remark that for $n = 4$, we still obtain an irreducible component $\overline{\cP}_4$ of the KSB-moduli space and the combinatorics of the set of boundary objects is still captured by pruned trees. However, the deformation theory fails and so the map $\mathcal{W}_n^{\min} \to \overline{\cP}_4$ may not be birational. In particular, it is an open question whether $\overline{\cP}_4$ has generically non-reduced structure. See Remark \ref{rem:n=4}. 
\end{Remark}


\subsubsection*{Prior and related works}

Beyond the relevant work that was mentioned before, 
this paper uses extensively the results of \cites{lanave2002explicit, ascher2017log, AB, inchiostro2020moduli} on moduli of elliptic surfaces. In particular, the existence of the morphism $\cK^{\nu}_n\to \overline{\cE}_{n\epsilon}$ follows from general wall-crossing phenomena studied in \cites{AB, inchiostro2020moduli, ascher2021wall,meng2023mmp}, whereas the understanding of the limits in $\overline{\cP}_n$ follows from \cites{lanave2002explicit, AB, inchiostro2020moduli}.
More precisely, the morphism $\bar{\Phi}_n$ was studied in \cite{AB}, as a special case of general wall-crossing phenomena of \cites{ascher2021wall, meng2023mmp}. The explicit combinatorial gadget of \Cref{subsection_comb}
is instead an improvement of the refined numerical data defined in \cite{inchiostro2020moduli}, for elliptic surfaces of Kodaira dimension one fibered over $\bP^1$. During the revision of this paper, another work \cite{ISZ25} appeared, studying the KSBA moduli of surfaces fibered in log Calabi–Yau curves. As a side remark, we note that the families under consideration are precisely those attaining the minimal volumes for a fixed geometric genus $p_g$; see \cites{Liu25}.

\subsection*{Outline of the paper}

In Sections \ref{section: preliminaries} and \ref{section: elliptic surfaces}, we compile some preliminaries on elliptic surfaces, KSB(A)-stability and KSB(A)-moduli spaces.
In Sections \ref{sec:moduli_ell} and \ref{section: main theorem}, we study different moduli spaces of elliptic surfaces with sections, and prove Theorem \ref{thm:main}(1-2).
In Section \ref{section: n=3}, we give a more direct proof in the case when $n=3$ by computing explicitly the deformation theory of such surfaces.
In \cref{section:tsmcompactification}, we introduce twisted stable maps, prove Theorem \ref{thm:main}(3), and develop a combinatorial algorithm for computing the KSB-stable limits in order to prove \Cref{main:combinatorics}. Finally, in \Cref{sec:wall-crossing-iso}, we prove \Cref{main:wall} and put everything together in the case $n=3$ to prove \Cref{main:n=3}(1). 

\subsection*{Acknowledgement}

This project was initiated during the AGNES Summer School on Higher Dimensional Moduli at Brown University in 2022. We are grateful to the organizers as well as all the speakers. During the preparation of this article, AFH and GI were partially supported by an AMS-Simons travel grant. DB is partially supported by NSF grant DMS-2401483.
SM was a visiting researcher at Universit\"at Duisburg-Essen in 2023---2024 and would like to thank it and Marc Levine for the hospitality. We thank Rita Pardini for helpful comments on an earlier version of this manuscript.

\section{Preliminaries on KSBA-stable pairs}
\label{section: preliminaries}

In this article, we work over an algebraically closed field $k$ of characteristic $0$. The reader may assume that $k=\bC$.
\subsection{Background definitions}

We recall a few definitions that will be useful for the rest of the paper.

\begin{Notation}
    Consider a flat and finite type morphism $g \colon X\to Y$ of relative dimension $n$. If the fibers of $g$ are Gorenstein, then the relative dualizing complex is isomorphic to a complex of the form $\omega_g[n]$ where $\omega_g$ is a line bundle.
    If the fibers of $g \colon X \to Y$ are normal, then we have an explicit description: $$\omega_g = \left(\Lambda^n\Omega_{X/Y}^1\right)^{**},$$ where $\Omega_{X/Y}^1$ is the sheaf of relative differentials. 
    Indeed, both are reflexive sheaves: $\omega_g$ is a line bundle so it is reflexive, while $(\Lambda^n\Omega_{X/Y}^1)^{**}$ is reflexive by virtue of being a dual. 
    By flat base change for dualizing complexes, they agree on the smooth locus of $g$,
    which has codimension at least 2 along each fiber, so they are isomorphic by Hartogs's lemma. 
\end{Notation}
We now introduce a few standard definitions from birational geometry \cite{KM98}*{Notation 0.4}. We refer the reader to \cite{KM98} for a more extensive treatment of what follows. 
\begin{Notation}
    Given a normal variety $X$ with two Weil divisors $D_1,D_2$, we write $D_1\sim D_2$ to denote that $D_1$ and $D_2$ are linearly equivalent.
    \end{Notation}

\begin{Def}
    Let $X$ be a normal variety. A \emph{$\bQ$-divisor} $D=\sum a_i D_i$ on $X$ is a formal linear combination of 
    integral Weil divisors $D_i\subseteq X$ with rational coefficients $a_i\in \bQ$.
\end{Def}

\begin{Def}
    Let $X$ be a normal variety, $D$ a Weil divisor on $X$, and $\pi \colon X'\to X$ a birational morphism. We use the same notation as \cite{KM98}*{Notation 0.4 (11)} and write $\pi_*^{-1}(D)$ for
    the \emph{proper transform} of $D$. When $\pi^{-1}$ is defined on a dense open $D_0 \subset D$,
    Then $\pi_*^{-1}(D)$ is the closure of $\pi^{-1}(D_0)$.
    The definition is extended to $\bQ$-divisors by linearity. 
\end{Def}

\begin{Def}[{\cite{KM98}*{Def. 2.11}}]
\label{definition: canonical singularities}
    A normal variety $X$ is said to have \emph{canonical singularities} if
    \begin{enumerate}
        \item there exists $n > 0$ such that $nK_X$ is a Cartier divisor, in other words, $K_X$ is $\bQ$-Cartier, and
        \item for some resolution of singularities $\pi \colon X' \to X$, we have
        $\pi_*(\cO_{X'}(nK_{X'})) \cong \cO_X(nK_X)$.
    \end{enumerate}
\end{Def}

\begin{Def}[{\cite{kollar_modbook}*{Def. 11.5}}]
\label{def_klt}
    Let $X$ be a normal projective variety, and $D=\sum a_i D_i$ be an effective $\mb{Q}$-divisor. 
    Let $i \colon X^{\sm} \hookrightarrow X$ be the inclusion of the smooth locus of $X$. Then $(X,D)$ is called a \emph{log pair} if there is an integer $n>0$ such that the sheaf 
        $$\displaystyle i_* \left( \omega_{X^{\sm}}^{\otimes n}(nD|_{X^{\sm}}) \right)$$ 
        is a 
        line bundle on $X$.
        We will denote the corresponding Cartier divisor by $n(K_X+D)$ and call $K_X+D$ \emph{$\bQ$-Cartier}.
        We say that a log pair $(X,D)$ is \emph{Kawamata log terminal} or \emph{klt} (resp. \emph{log-canonical} or \emph{lc}) if the following conditions are satisfied:
    \begin{enumerate} 
        \item the coefficients $a_i$ satisfy $0\leq a_i<1$ (resp. $0\leq a_i\leq 1$), and
        \item given a log-resolution $\pi \colon X'\to X$, denote by $E_j$ the integral components of the exceptional divisor. Assume that $\pi$ is such that the union of the supports of $\pi_*^{-1}D$ and $\sum E_j$
        is simple normal crossing. Then we require that $b_j>-n$ (resp. $b_j\ge -n$) for every $j$, where the $b_j\in \bZ$ are defined by
        
        \begin{equation}
        \label{equation: coeffs for log discrepancy}
            n(K_{X'} + \pi_*^{-1} D) \sim \pi^*(n(K_X+D)) + \sum b_j E_j. 
        \end{equation}
    \end{enumerate}
\end{Def}


\begin{Def}
    Let $X$ be a projective variety and $D$ be an effective $\mb{Q}$-divisor. 
    Then $(X,D)$ is called \emph{a semi-log-canonical pair} (or \emph{an slc pair}) if
    $X$ is equidimensional and $S_2$, 
    all points of codimension 1 are either smooth or nodal,
    the irreducible components of $D$ intersect the smooth locus of $X$,
    and the following hold.
    \begin{enumerate}
        \item Denote by $U\subseteq X$ the locus where $X$ is Gorenstein and $D$ is Cartier. We then require that, for some $n>0$, the sheaf $i_* \left( \omega_U^{\otimes n}\otimes \cO_U(nD|_U) \right)$
        is a line bundle on $X$. We denote the corresponding Cartier divisor by $n(K_X+D)$ and say that $K_X+D$ is $\bQ$-Cartier.
        \item If $\nu \colon X^\norm \to X$ is the normalization of $X$,  $\Delta \subseteq X^\norm$ is the preimage of the nodal locus, and $D^\norm$ is the proper transform of $D$ in $X^\norm$, then the pair $(X^\norm, D^\norm +\Delta)$ is log canonical.
    \end{enumerate}
\end{Def}
\begin{Def}\label{def_lct}
    Let $(X,D)$ be a pair such that $K_X$ and $D$ are both $\bQ$-Cartier and $X$ is log-canonical.
    We define the \emph{log-canonical threshold} of $(X,D)$, denoted by $\lct(X; D)$, to be $$\sup\{\alpha \in \mathbb{R}:(X,\alpha D) \textup{ is log-canonical}\}.$$
\end{Def}

\subsection{Canonical models}

\begin{Def}
    Let $(X,D)$ be an lc pair and assume that $X$ is a surface. 
    A \textit{canonical model} of $(X,D)$ is a birational map $\pi \colon X\to X'$ to a variety $X'$ such that, if we write $D':=\pi_*D$, then:
    \begin{enumerate}
        \item $(X',D')$ is an lc pair,
        \item there is $n>0$ such that $n(K_{X'} + D')$ is ample, and
        \item given $m$ such that both $m(K_X + D)$ and $m(K_{X'} + D')$ are Cartier, if we write 
        \[
            m(K_X + D) - \pi^*(m(K_{X'} + D')) \sim \sum a_i E_i,
        \]
        where $E_i$ are the components of the exceptional divisor of $\pi$, then $a_i \ge 0$.
    \end{enumerate}
\end{Def}

\begin{Remark}
    The definition of canonical model extends to higher dimensional varieties. However, one can no longer assume that $\pi$ is a morphism, rather it has to be a birational rational map whose inverse does not contract divisors.
    In this manuscript, we will only need the surface version, so we refer the reader to \cite{KM98}*{\S3.8} for more details on the higher dimensional case.
\end{Remark}

\begin{Teo}[{cf. \cite{KM98}*{Thm. 3.52}}]
\label{theorem: construction of canonical models}
    Let $(X,D)$ be an lc pair with $X$ proper. Fix $m$ such that $m(K_X + D)$ is Cartier.
    Then a canonical model $(X', D')$, if it exists, is unique up to isomorphism and is given by
    \[
    X'\ =\ \Proj \left(\bigoplus_{r \geq 0} H^0 
    \bigg( X, \cO_X\big(rm(K_X+D)\big) \bigg)
    \right).\]
\end{Teo}

\begin{Def}
    More generally, if $(X,D)$ is a normal but not necessarily log canonical pair, we can define the \emph{canonical model} of $(X,D)$ as follows.
    Let $\mu \colon X' \to X$ be a log resolution with reduced exceptional divisor $E$.
    By assumption, $\mathrm{Supp}(\mu_*^{-1}D) \cup E$ is normal crossings, so by the above theorem, if the pair $(X, \mu_*^{-1}D + E)$ admits a canonical model, it is unique. We define this to be the canonical model of $(X,D)$. One can check by the above description as a Proj that this is independent of the choice of $\mu$ if it exists. 
\end{Def}

\begin{Def}
\label{definition: KSBA-stability}
    A pair $(X, D)$ is called \textit{KSBA-stable} if 
    \begin{enumerate}
        \item it is slc and $X$ is connected;
        \item $K_X+D$ is an ample $\bQ$-Cartier $\bQ$-divisor.
    \end{enumerate}
    The \emph{volume} of a KSBA-stable pair $(X,D)$ is $\vol(X,D) = (K_X + D)^{\dim X}$.
\end{Def}

\begin{Teo}[{cf. \cite{kollar_modbook}*{Theorem 8.15}}]
\label{theorem: KSBA is a good moduli theory}
    Fix $v\in \mb{Q}_{>0}$ and $c\in [0,1]_{\mb{Q}}$. Then there is a proper Deligne-Mumford stack $\MKSBA_{c,v}$, whose closed points parametrize KSBA-stable surface pairs $(X,cD)$ with volume $(K_X+cD)^2=v$, where $D$ is an effective $\mathbb{Z}$-divisor.
\end{Teo}

The complete definition of a family of KSBA-stable pairs is more subtle and we omit it. For this paper, the following remark suffices, and it follows from Koll\'ar's definition \cite{kollar_modbook}*{8.13} and \cite{kollar_modbook}*{Theorem 5.4}, noting that relative ampleness can be defined fiberwise for finite type morphisms \cite{EGAIV}*{9.6.4}.

\begin{Remark}
\label{remark: easy_family}
        Assume that $\cX\to B$ is a flat, pure dimensional, Gorenstein and projective morphism with $B$ reduced, and $\cD\subseteq \cX$ is a Cartier divisor, flat over $B$, and such that for every $b\in B$ the fiber $(\cX_b,c\cD_b)$ is a stable pair with $(K_{\cX_b}+c\cD_b)^{\dim \cX_b}=v$ for some rational coefficient $c$. 
        Then $\cX\to B$ is a family of KSBA-stable pairs: namely, it corresponds to a morphism $B\to \MKSBA_{c,v}$.
\end{Remark} 

\begin{Def}\label{def:ksb} 
A connected projective variety $X$ is \emph{KSB-stable} if $(X, 0)$ is KSBA-stable, i.e., $X$ is slc and $K_X$ is ample. Its {volume} is defined as $\vol(X): = (K_X)^{\dim X}$.
\end{Def}

\Cref{theorem: KSBA is a good moduli theory} is still applicable here as the special case when $D = 0$, and we get a moduli space which we denote $\MKSB_v$. Through out this paper, when we refer to $\MKSB_v$, we take $v=\frac{(n-2)^2}{n}$ to be the volume of a pseudo-elliptic surface of with a $\frac{1}{n}(1,1)$-singularity, where $n\geq 3$.

\subsection{KSB-stable families}\label{subsection_KSB_stab}

In this subsection, we introduce the notion of KSB-stable family. 

\begin{Notation}
    Let $X \to B$ be a flat family with $S_2$ fibers. For any reflexive sheaf $\cL$ on $X$ which is a line bundle over a dense open subset, we set $\cL^{[m]} := \left(\cL^{\otimes m}\right)^{**}$ to be the reflexive hull of the $m^{th}$ tensor power.
\end{Notation}

\begin{Def}
    Let $\pi \colon X\to B$ be a flat and proper morphism. We say that $\pi$ is \textit{KSB-stable} if for every $b\in B$, the fiber $X_b$ is KSB-stable, and $\pi$ satisfies \emph{Koll\'ar's condition}:
    %
    For every $B'\to B$ and for every $n\in \bZ$, if we denote by $p_1 \colon X\times_B B'\to X$ the first projection, then the natural map
    $$p_1^*(\omega_{X/B}^{[n]}) \longrightarrow \omega_{X\times_B B'/B'}^{[n]}$$
    is an isomorphism.
\end{Def}

It was proven in \cite{Kol09} that Koll\'ar's condition is an algebraic condition, i.e, the corresponding subfunctor is represented by a monomorphism of schemes. When the base $B$ is reduced, then Koll\'ar's condition can be checked numerically as follows.
\begin{Teo}[{cf. \cite{kollar_modbook}*{Thm. 5.1}}] 
\label{thm: kollar condition numerical} 
Let $\pi \colon X \to B$ be a flat, pure dimensional, proper morphism over a reduced base scheme $B$. Suppose that for all $b \in B$, the fiber $X_b$ is KSB-stable. If the volume $\left(K_{X_b}\right)^{\dim X_b}$ is the same constant value for all $b \in B$, then $\pi$ is KSB-stable. 
\end{Teo}

We now introduce an auxiliary tool, developed in \cites{hacking,AH}, which we will use to study the infinitesimal structure of $\MKSB$.

\begin{Def}\label{def_covering_stack}
    The \emph{covering stack} of a KSB-stable family $X\to B$ 
    is the quotient stack
\[
\cX
:= \left[ \spec_X \bigg( \bigoplus_{m \in \bZ} \omega_{X/B}^{[m]} \bigg)\bigg\slash \bG_m \right] 
.\]
Note that the natural map $\pi \colon \cX\to X$ is an isomorphism on the open set $U\subset X$ where $\omega_{X/B}$ is a line bundle. 
\end{Def}

The main relevance of \Cref{def_covering_stack} lies in the following result.

\begin{Teo}[{cf. \cite{AH}*{Thm. 5.3.6}}]
\label{thm_to_check_defs_we_look_at_covering_stack} 
Let $X\to B$ be a KSB-stable family. Then its covering stack $\cX\to B$ is a flat and proper family of Deligne-Mumford stacks, and $\cX_b$ is the covering stack of $X_b$ for every $b\in B$.
Conversely, if $\cX\to B$ is a flat and proper morphism, whose fibers $\cX_b$ are covering stacks of KSB-stable varieties, then there is a KSB-stable family $X\to B$ whose covering stack is $\cX\to B$.    
\end{Teo}
In particular, if $X$ is KSB-stable, to study the local structure of $\MKSB_v$ around the point $p$ corresponding to $X$, one can either study flat deformations of $X$ which satisfy Koll\'ar's condition, or, equivalently, study flat deformations of the covering stack of $\cX$.

\section{Elliptic surfaces}
\label{section: elliptic surfaces}

In this part, we will recall a few facts about elliptic fibrations from \cite{Miranda} that are needed in the rest of the paper. Throughout this section, let $C$ be a smooth curve. 

\begin{Def}
    An \emph{elliptic surface} over $C$ is a pair $(g \colon X \to C, S)$ where 
    \begin{enumerate}[(1)]
        \item $g \colon X \to C$ is a flat proper morphism with connected fibers of arithmetic genus $1$ such that the generic fiber of $g$ is smooth, and
        \item $S \subset X$ is a section.
    \end{enumerate}
    The elliptic surface is \emph{standard} if $S \subset X$ does not pass through a singular point of any fiber.
    A standard elliptic surface $(f \colon X \to C, S)$ is \emph{minimal} if $X$ is smooth and there are no $(-1)$-curves contracted by $g$. 
\end{Def}

Given a standard elliptic surface $(f \colon Y \to C, S_Y)$, one can always contract every fiber component not meeting $S_Y$ to obtain a standard elliptic surface $(g : X \to C, S)$ with integral fibers, called the \emph{Weierstrass model}.

\begin{Def}[{\cite{Miranda}*{II.3.2}}] \label{defn: weierstrass fibration}
    A \emph{Weierstrass fibration} is a standard elliptic surface whose geometric fibers are integral. 
\end{Def}

The condition of integrality on the fiber implies that $X$ admits a \emph{global Weierstrass equation}. That is, $X$ is isomorphic to 
\begin{equation}
\label{equation: global Weierstrass equation}
    \big\{y^2 z - x^3 - A x z^2 - B z^3 = 0\big\} \ \subseteq\ \mathbb{P}\big( \cL^{-2} \oplus \cL^{-3} \oplus\cO_C\big),
\end{equation}
where $\cL$ is an effective line bundle on $C$ and $A \in H^0(C, \cL^4)$, $B \in H^0(C, \cL^6)$ (see \cite{Miranda}*{II.5.1, II.5.2}). In particular, the fibers of $g$ are either a smooth elliptic curve, a rational curve with a node, or a rational curve with a cusp. The line bundle $\cL$ is called the \emph{fundamental line bundle} of the fibration, and determines the canonical bundle of $X$ as follows. 

\begin{Teo}[Kodaira's canonical bundle formula \cite{BPV84}*{V.12.1}, \cite{Miranda}*{Prop. III.1.1}] \label{thm: canonical bundle formula}
Let $(g \colon X \to C, S)$ be a Weierstrass fibration, and let $\cL$ be the fundamental line bundle. Then we have $$\omega_X \cong g^*(\omega_C \otimes \cL).$$
\end{Teo}

\subsection{Minimal Weierstrass fibrations}

We will be particularly interested in the case of canonical surface singularities (also called \textit{du Val} or \emph{ADE} singularities). See \cite{KM98}*{\S4.2}, especially \cite{KM98}*{Thrm. 4.20}, for a description.

\begin{Def}
    Let $(X \to C, S)$ be a Weierstrass fibration.
    We say it is a \emph{minimal Weierstrass fibration} if $X$ 
    has canonical singularities.
    We say it is an \emph{lc Weierstrass fibration} if $X$ 
    has log-canonical singularities, i.e. $(X,0)$ is lc.
\end{Def}

The naming comes from the well-known fact that $(X \to C, S)$ is a  minimal Weierstrass fibration if and only if the minimal resolution of $X$ is a minimal elliptic surface \cite{Miranda}*{Def. III.3.1, Prop. III.3.2}. In particular, the minimal Weierstrass model is unique among birational models of the elliptic fibration since the minimal model of surfaces is unique. 

\begin{Remark}[cf. {\cite{Miranda}*{III.3.2}}]\label{remark_when_there_are_lc_sing_based_on_W_data}
    In terms of the Weierstrass data $(\cL, A, B)$, 
    the Weierstrass fibration
    is minimal (resp. lc) if and only if 
    for all $p \in C$:
    $$\min\big\{3\mathrm{ord}_p(A),\  2\mathrm{ord}_p(B)\big\} \ \leqor  \  12 . $$
\end{Remark}

\begin{Construction}
\label{construction:Lfiber}
    We now describe a birational transformation that, given a minimal Weierstrass fibration $(g \colon X\to C,S)$, produces a strictly lc Weierstrass fibration $(g' \colon X' \to C, S')$, i.e. one which is not minimal.
    First, consider a fiber $F$ of $g$ over $p\in C$, and let $X'$ be the blow-up of $X$ at the intersection point $F\cap S$.
    Let $F'$, $S'$ be the proper transforms of $F$, $S$, respectively, and let $E$ denote the exceptional divisor, then
    the fiber of $X' \to C$ over $p$ is $F' \cup E$.
    Assume that $(X,S+F)$ is an lc pair, that is, $F$ was chosen either smooth or nodal.
    One can check that $K_{X'}+S'+F'$ is nef over $C$ by checking its intersections with $F'$ and $E'$ are nonnegative ($0$ and $1$, respectively). 
    The canonical model of $(X',S'+F')$ over 
    $C$ contracts only $F'$, and the resulting surface pair $(X^c,S^c)$ is strictly log-canonical, as $F'$ is an exceptional divisor with discrepancy $-1$.
    In particular, this process generates strictly lc singularities.

    Recall that $\cL$ is the dual of the normal bundle of the section, so this process corresponds to replacing $\cL$ with $\cL\otimes\cO_C(p)$.
    Similarly, we are replacing $(A,B)$ with $(z^4A,z^6B)$, where $z$ is a nonzero section of $\cO_C(p)$ vanishing at $p$. 
\end{Construction}

\begin{Remark}
    The singular fibers of a minimal Weierstrass fibration are classified in terms of the dual graph of a minimal resolution, which can be determined from the vanishing order of the Weierstrass data via Tate's algorithm.
    The classification of singular fibers is due to Kodaira and Ner\'on, and we use Kodaira's notation here.
    We refer the reader to \cite{Miranda}*{Section I.4} and \cite{MR2732092}*{Page 66} for more details. 
\end{Remark}

\begin{Remark}
    Let $g\colon X\rightarrow\bP^1$ be a Weierstrass fibration with section $S\subseteq X$. It follows from \cite{Miranda}*{II.5.6} that $S^2\le 0$.
\end{Remark}

When $S$ is rational and $S^2 = -n$, it follows from the adjunction formula that $\omega_X|_S\simeq\cO_{\bP^1}(n-2)$. Moreover, $\omega_g$ is trivial on the fibers, so by the seesaw theorem \cite[page 54, Cor. 6]{mumford1974abelian} we can write $\omega_g = g^* \cO_{\bP^1}(k)$ for some $k$, and we have
$$\omega_X\ =\ g^{*}\omega_{\bP^1}\otimes \omega_g\ =\ 
g^{*}\cO_{\bP^1}(-2)\otimes g^{*}\cO_{\bP^1}(k).$$
We can combine these two observations to deduce that $\omega_X\simeq g^{*}\mathcal{O}_{\bP^1}(n-2)$ and $\omega_g\simeq g^{*}\mathcal{O}_{\bP^1}(n)$. The canonical bundle formula (\Cref{thm: canonical bundle formula}) implies that $\cL \cong \cO_{\mathbb{P}^1}(n)$ and $\omega_g \cong g^*\cL$. 

\begin{Def}
    Given an irreducible elliptic surface $(X,S)\to \bP^1$ with fibers that are irreducible genus one curves, we define the \emph{height} of $(X,S)\to \bP^1$ to be $-(S)^2$.
\end{Def}

\begin{Remark}
    The height $n = 0$ if and only if $X = E \times C$ for an elliptic curve $E$ \cite{Miranda}*{Lemma III.1.4}. 
\end{Remark}

\subsection{Twisted fibers}\label{subsection_twisted_fiber}
Not all elliptic surfaces are standard.
For example, one can start with a Weierstrass fibration $X\to C$
as in \cref{construction:Lfiber}, but then fix
a cuspidal fiber $F$ instead of a smooth or nodal one, and section $S$.
One can then consider the minimal log-resolution $(X', S'+F' + E)\to (X,S+F)$, where $S'$, $F'$ are proper transforms of $S$, $F'$, respectively, and $E$ is the reduced exceptional divisor.
Now we may assume the pair $(X',S'+F'+E)$ is SNC, so one can construct its relative canonical model $(X^c, S^c+F^c)$ over $C$ as in \cite{ascher2017log}.
It is proven in \textit{loc.~cit.} that the new pair $(X^c,S^c)$ is not a standard elliptic surface, namely, $X^c$ is singular at $F^c\cap S^c$. However, the singular fibers arising in this process are controlled.

\begin{Def}[{\cite{ascher2017log}*{Def. 4.9}}]
Let $g \colon (X,S)\to C$ be a relative canonical model of an elliptic surface with section.
A \emph{twisted fiber} of $g$ is an irreducible but non-reduced fiber.    
\end{Def}

Twisted fibers also appear as follows. Consider a Weierstrass fibration $(X,S)\to C$.
Let $U\subseteq C$ be the locus where the fibers of $g$ are either smooth or nodal.
Assume that $U\neq C$. There is a map $\phi  \colon U \rightarrow \overline{\mathcal{M}}_{1,1}$.
Up to replacing $C$ by some \textit{root stack} $\gamma \colon \cC\to C$ of $C$, with $U\subseteq C$ the open locus over which $\gamma$ is an isomorphism, we can extend $\phi$ to $\Phi \colon \cC\to \overline{\mathcal{M}}_{1,1}$ (cf. \cites{tsm, bejleri2022height, bresciani2023arithmetic}).
Consider the pull-back of the universal family via $\Phi$, denoted by $(\cX,\cS)\to \cC$, and let $g' \colon (X',S')\to C$ be the corresponding map on coarse moduli spaces.
From \cite{ascher2017log} all the fibers of $g'$ over $C\setminus U$ are twisted, so this procedure replaces the fibers of $g$ away from $U$ with twisted fibers.

In fact, these two constructions of twisted fibers are equivalent. 

\begin{Prop}[{\cite{tsm}*{Prop. 4.12}}]
    Let $(g \colon X \to C, S)$ be a Weierstrass fibration with a cuspidal fiber $F$.
    Then the relative canonical model $(X^c, S^c+F^c)$ of $(X, S + F)$ over $C$ is the twisted model $(X', S')$ obtained from the root stack construction above. 
\end{Prop}

\subsection{Stability for Weierstrass fibrations}

For the remainder of this section, we assume that $g \colon X \to \bP^1$ is an lc Weierstrass fibration with section $S$, fiber class $f$, and height $-S^2 = n$.

\begin{Lemma}
\label{lemma_when_KX+cS_is_ample_and_volume}
    Let $(X \to \bP^1, S)$ be an lc Weierstrass fibration as above with height $n$. Then the following holds for the pair $(X,cS)$.
    \begin{enumerate}
        \item When $n \geq 3$, the pair is KSBA-stable if and only if $0<c<\frac{n-2}{n}$.
        \item Its volume is $\vol(X,cS) = 2c(n-2) - nc^2$.
    \end{enumerate}
\end{Lemma}

\begin{proof}
    Assume first that $0<c<\frac{n-2}{n}$. Since $X$ has lc singularities, $S$ is smooth and contained in the smooth locus of $X$, and $c<1$, we conclude that $(X,cS)$ is automatically lc, hence slc.
    As its volume is $$(K_X+cS)^2\ =\ ((n-2)f+cS)^2\ =\ 2c(n-2)-c^2n,$$ and
    the inequalities on $c$ imply that $(K_X+cS)^2 > 0$.
    Now to check that $K_X+cS$ is ample, it suffices to show that $K_X+cS$ intersects every irreducible curve $D$ positively.
    Since $g \colon X\rightarrow \mathbb{P}^1$ has irreducible fibers and $(K_X+cS). f=c>0$, we may assume that $D$ is not a fiber. 
    Then in particular we have $(D.f)\geq 1$. Thus, it follows that $$(K_X+cS).D \ \geq\ n-2+c\cdot (S.D)\ \geq\ n-2-cn\ >\ 0,$$ for $0<c<\frac{n-2}{n}$. Here we have used that $(S.D) \geq 0$ if $D \neq S$ and $(S.D) = -n$ if $S = D$. The converse statement in part (1) is clear from the formula for $(K_X+cS)^2$.
\end{proof}

\begin{Oss}
\label{observation: canonical model for epsilon=0}
    Consider a pair $(X,cS)$ as in {\Cref{lemma_when_KX+cS_is_ample_and_volume}}. When $c=\frac{n-2}{n}$, we have that $(K_X+cS).S=0$. 
    One can check that the canonical model of $\left( X, \frac{n-2}{n}S \right)$ is the surface $Y$ obtained by contracting the section $S$ to a point. If we denote by $\pi \colon X\to Y$
    such a contraction, then $$\pi^*K_Y \ =\ K_X+\frac{n-2}{n}S,$$ and thus the volume of $Y$ is $$\vol(K_Y)\ =\ 2\cdot\frac{(n-2)}{n}\cdot(n-2) - n\cdot \frac{(n-2)^2}{n^2}\ =\ \frac{(n-2)^2}{n}.$$
\end{Oss}

\begin{Remark}
    In \cites{lanave2002explicit, AB}, the surfaces obtained from a Weierstrass fibration by contracting the section are also called \textit{pseudo-elliptic surfaces}.
\end{Remark}

\section{Moduli spaces of elliptic surfaces}\label{sec:moduli_ell}

In this section, we introduce the three moduli spaces of elliptic surfaces we will be dealing with, and we will present some of the relations between them.





%
\subsection{Weierstrass fibrations}

\label{subsec: Weierstrass fibrations}
We begin by recalling a description for the moduli stack of Weierstrass fibrations.

\begin{Def}\label{def:weierstrass}
Fix $n \geq 0$. We define the stack $\Wei^{\min}$ as the pseudofunctor
\begin{equation*}
\Wei^{\min}(B) \ :=\  \left\{\xymatrix{S \subset X \ar[rr]^f \ar[rd] & & C \ar[ld] \\ & B &} \left| \hspace{-3ex}
\parbox{26em}{
    \begin{enumerate}
    \item $f \colon X \to C$ is flat projective of relative dimension $1$;
    \item $C \to B$ is flat projective with geometric fibers $\bP^1$;
    \item $S \subset X$ is a section of $f$;
    \item for each $b \in B$, the fiber $f_b \colon (X_b, S_b)\to C_b$ is a minimal Weierstrass fibration of height $-(S_b)^2 = n$.
    \end{enumerate}}		
\right. \right\}.
\end{equation*}
Similarly, we let $\Wei^{\lc}$ be the pseudofunctor of families as above, where we allow $(f_b \colon X_b \to C_b, S_b)$ to be an lc Weierstrass fibration for all $b \in B$. 
\end{Def}

In the definition above, $C \to B$ is a family of one-dimensional Brauer-Severi varieties and, in particular, corresponds to a $\PGL_2$-torsor over $B$.

\begin{Teo} \textup{(cf. \cite{PS}*{Main Theorem 1.2} \&  \cite{canning2022integral}*{Thm. 2.8})}
\label{theorem: construction of Wn} 
\leavevmode
Keep the notations as above.
\begin{enumerate}
    \item The pseudofunctors $\Wei^{\min}$ and $\Wei^{\lc}$ are representable by smooth and irreducible Deligne-Mumford stacks of finite type. Moreover, $\Wei^{\min}$ is an open substack of $\Wei^{\lc}$
    \item There are nonempty open loci $\Wei^{\reg}, \Wei^{\nod} \subset \Wei^{\min}$ parametrizing those Weierstrass fibrations such that $X$ is smooth or that have at worst nodal fibers respectively. 
\end{enumerate}
\end{Teo}

\begin{proof}
    The construction of $\Wei^{\min}$ as a coarse moduli space is classical (cf. \cites{mir, seiler}) and the modern stacky perspective is given in \cite{PS}*{\S4} and \cite{canning2022integral}*{\S2}. We sketch here the irreducibility, smoothness and quasicompactness of the stacks. We refer the interested reader to the references (e.g. \cite{PS}*{Thm. 1.2(b)}) for the fact that the stacks are Deligne-Mumford.
    
    Consider the pseudofunctor $\tWei^{\lc}$ whose $B$-points consist of tuples $$\big(f \colon X \to C,\  S,\  \alpha \colon C \ \cong\  \mathbb{P}^1_B\big),$$
    where $(f \colon X \to C, S)$ is as above and $\alpha$ is an isomorphism of $C$ with $\mathbb{P}^1_B$.
    We call $\alpha$ a \emph{framing}. 
    There is a $\PGL_2$-action on $\tWei^{\lc}$, which precomposes the framing with an automorphism of $\bP^1_B$, such that $$\Wei^{\lc}\ \cong \ \big[\tWei^{\lc}/\PGL_2\big].$$
    Thus, it suffices to show that $\tWei^{\lc}$ is representable by a smooth Deligne-Mumford stack and that the corresponding subfunctors
    $$\tWei^{\reg}\ \subseteq
    \ \tWei^{\nod} \  \subseteq\ \tWei^{\min}\ \subseteq\ \tWei^{\lc}$$ are all open. 
    Let $V_m := H^0(\bP^1, \cO_{\bP^1}(m))$ be the space of degree $m$ homogeneous polynomials and define
    $$
    \bV_n \ :=\  V_{4n} \oplus V_{6n}. 
    $$
    Two Weierstrass equations give isomorphic surfaces over $\bP^1$ if and only if they differ by the $\bG_m$-action with weights $(4,6)$ on $\bV_n$. Thus there is a natural morphism $$\varphi \colon \tWei^{\lc}\ \longrightarrow\ [\bV_n/\bG_m]$$
    which sends a framed lc Weierstrass fibration $(f \colon X \to \bP^1, S)$ of height $n$ 
    to the coefficients of its Weierstrass equation $(A,B) \in [\bV_n/\bG_m]$ as in \cref{equation: global Weierstrass equation}. See \cite{PS}*{\S4} for a description of this construction in families.
    The image of $\varphi$ is the open substack $\cU^{\lc}$ defined by the condition
    \begin{equation}\label{eqn_ineq}
    \min\big\{3\cdot \ord_x(A), \ 2\cdot\ord_x(B)\big\}\  \le\ 12 
    \end{equation}
    for any $x \in \bP^1$. 
    Similarly, one has that 
    \begin{enumerate}
        \item the image of $\tWei^{\min}$ is the open substack $\cU \subset \cU^{\lc}$ defined by a strict inequality in (\ref{eqn_ineq}),
        \item the image of $\tWei^{\nod}$ is the open locus where the vanishing of $A$ and the vanishing of $B$ are disjoint,
        \item and the image of $\tWei^{\reg}$ is open because smoothness is open in flat families (cf. \cite{Har13}*{Chap. III.10}).
    \end{enumerate} 

    The construction, which sends a family of Weierstrass data $(A,B)$ to the surface
    $$
    X \ = \ \bV\big(Y^2 Z - (X^3 + AXZ^2 + BZ^3)\big) \ \subseteq\ \bP\big(\cO_{\bP^1} \oplus \cL^{-2} \oplus \cL^{-3}\big)\ \longrightarrow\ \bP^1
    $$
    with the section $S = \bV(X,Z)$, defines an inverse $\cU^{\lc} \to \tWei^{\lc}$, which maps $\cU$ isomorphically onto $\tWei^{\min}$. 
    
    Finally, it follows from this construction that $\Wei^{\min}$ and $\Wei^{\lc}$ are smooth, since they are given as quotients of open substacks of $\bV_n$. 
\qedhere    
\end{proof}

\begin{Cor}
    Let $(f \colon X \to C, S) \to B$ be a family of lc Weierstrass fibrations. Then there exists a closed subscheme $$\Delta(f)^{\cusp}\ \subseteq\ C$$ such that for each $b \in B$, $x \in \Delta(f)^{\cusp}_b$ if and only if the fiber $f_b^{-1}(x)$ is cuspidal. Moreover, the formation of $\Delta(f)^{\cusp}$ commutes with base change. 
\end{Cor}

\begin{proof} 
Let $\mathscr{C}^{\univ} \to \Wei^{\lc}$ denote the universal family of base curves for the universal Weierstrass fibration. It suffices to define the closed substack $\Delta^{\cusp} \subseteq \mathscr{C}^{\univ} \to \Wei^{\lc}$ and define $\Delta(f)^{\cusp}$ via pulling back. Moreover, we can do this smooth locally over $\Wei^{\lc}$, so it suffices to work with framed Weierstrass fibrations. Given a Weierstrass fibration $(f \colon X \to \bP^1, S)$ over $\spec k$, the cuspidal fibers are given by the vanishing 
$$
\big\{A = B = 0\big\}
$$
where $A$ and $B$ are the Weierstrass data. Let $\cA$ and $\cB$ be the universal Weierstrass data on $\bP^1_{\tWei^{\lc}}$. Then the closed substack 
$$
\widetilde{\Delta}^{\cusp}\ =\ \big\{\cA = \cB = 0\big\}\ \subseteq\ \bP^1_{\tWei^{\lc}}
$$
descends to a closed substack $\Delta^{\cusp}\subseteq \mathscr{C}^{\univ}$, which does the job. 
\end{proof} 

Each point in $\partial \Wei^{\lc} := \Wei^{\lc} \setminus \Wei^{\min}$ corresponds to a Weierstrass fibration $(f \colon X \to C, S)$ with some finite number $m>0$ of cuspidal fibers, where $X$ has elliptic singularities.
The associated minimal Weierstrass model $$\big(f^{\min} \colon X^{\min} \longrightarrow  C,\  S^{\min}\big)$$ has height $n - m$, and $(f \colon X \to C, S)$ can be recovered from $X^{\min}$ by choosing $m$ smooth or nodal fibers $F_1, \ldots F_m$, blowing up $F_i \cap S^{\min}$ and contracting the strict transform of $F_i$ to an elliptic singularity as in Construction \ref{construction:Lfiber}. 
This yields a stratification of $\Wei^{\lc}$ which we extend to a compactification in Section \ref{section:tsmcompactification}. 

\begin{Def}
    Define the pseudofunctor $\cW_{n-m,m}$ by sending a base scheme $T$ to  \begin{equation*}
\cW_{n-m,m}(T) = \left\{\xymatrix{S \subset X \ar[rr]^f \ar[rd] & & C \supset D \ar[ld] \\ & T &} \left| \hspace{-3ex}
\parbox{26em}{
    \begin{enumerate}
    \item $(f \colon X \to C, S) \to T$ is a family of minimal Weierstrass fibrations of height $n - m$;
    \item $D \subset C \to T$ is a relative effective Cartier divisor of degree $m$ which is \'etale over $T$;
    \item $D \cap \Delta(f)^{\cusp} = \varnothing$. 
    \end{enumerate}}		
\right. \right\}.
\end{equation*}
Note in particular that $\cW_{n,0} = \Wei^{\min}$.     
\end{Def} The following result will not be used in the rest of the paper. However, it is of independent interest as it describe the structure of $\cW_n^{\lc}$.

\begin{Teo}
\label{theorem: stratification of Wei lc} 
    The stacks $\cW_{n-m,m}$ are smooth, separated Deligne-Mumford stacks of finite type. 
    There exists a locally closed stratification 
    $$
    \bigsqcup_{m = 0}^n \cW_{n-m,m} \ \longrightarrow\ \Wei^{\lc}
    $$
    such that the image of $\cW_{n-m,m}$ is the locus of Weierstrass fibrations of height $n$ with exactly $m$ strictly log canonical elliptic singularities.
\end{Teo}

\begin{proof}
    Analogous to the construction of $\Wei^{\min}$, let $\tWeik$ be the $\PGL_2$ cover of $\cW_{n-m,m}$ given by adding the data of a framing $\alpha \colon C \cong \bP^1_T$ to the definition of $\cW_{n-m,m}$. 
    Then $\tWeik$ can be identified with the substack of  $\widetilde{\cW}_{n-m} \times \Sym^m \bP^1$ sending any base scheme $T$ to the groupoid of pairs $(f \colon X \to \bP^1_T, S)$ in $\widetilde{\cW}^{\min}_{m,n}(T)$ along with a degree $m$ relative Cartier divisor $D \in \Sym^m \bP^1(T)$ that is \'etale over $T$ and satisfies the condition that $\Delta(f)^{\cusp} \cap D = \varnothing$. Note that the \'etaleness of $D \to T$ is an open condition. On the other hand, since $$\Delta(f)^{\cusp}\ \subseteq \ \bP^1_T\  \ \  \textup{and}\ \  D\ \subseteq \ \bP^1_T$$ are both closed and $\bP^1_T \to T$ is proper, then the condition that $\Delta(f)^{\cusp} \cap D = \varnothing$ is open on $T$. Therefore, $\tWeik$ is an open substack of $\widetilde{\cW}_{n-m} \times \Sym^m\bP^1$. By \Cref{theorem: construction of Wn} for $\cW^{\min}_{n-m}$, this concludes the proof of the first claim. 

    Now we define natural morphisms $\tWeik \ \longrightarrow\  \tWei^{\lc}$ as follows. For any $T$-point of $\tWeik$, let
    $$
    A_0 \in \oH^0\big(\bP^1_T, L_0^{\otimes 4} \boxtimes \cO_{\bP^1}(4n-4m)\big) \quad \textup{ and }\quad B_0 \in \oH^0\big(\bP^1_T, L_0^{\otimes 6} \boxtimes \cO_{\bP^1}(6n-6m)\big)
    $$
    be the minimal Weierstrass data associated to the height $n-m$ fibration $(f \colon X \to \bP^1_T, S)$, where $L_0$ is a line bundle on $T$, and let $z \in \oH^0(\bP^1_T, \cO_{\bP^1_T}(D))$ be a defining polynomial of $D$. Note that $$\cO_{\bP^1_T}(D) \ \cong \ M\boxtimes \cO_{\bP^1}(m)$$ for some line bundle $M \in \Pic(T)$ via the isomorphism $\Pic(\bP^1_T) \cong  \Pic(T) \times \Pic(\bP^1)$. Then we can define new Weierstrass data
    $$
    A = A_0 z^4 \in \oH^0\big(\bP^1_T, L^{\otimes 4}\boxtimes \cO_{\bP^1}(4n)\big) \quad \textup{ and }\quad  B = B_0 z^6 \in \oH^0(\bP^1_T, L^{\otimes 6}\boxtimes \cO_{\bP^1}(6n))
    $$
    where $L = L_0 \otimes M$. By assumption $z$ has distinct roots and $z \neq 0$ whenever $A_0 = B_0 = 0$. Thus $(A,B)$ is minimal away from $D$, and strictly log canonical along $D$. Therefore, the Weierstrass pair $(A,B)$ defines a family of framed lc Weierstrass fibrations over $T$ with elliptic singularities over $D$ by Construction \ref{construction:Lfiber}. 

    This construction is functorial and $\PGL_2$-equivariant and hence defines natural morphisms
    $$
    \varphi_m : \cW_{n-m,m} \ \longrightarrow\ \Wei^{\lc}
    $$
    for each $m = 0, \ldots n$,  whose image is the locus of lc Weierstrass fibrations with exactly $m$ elliptic singularities. We claim that
    $$
    \bigsqcup_{m = 0}^n \cW_{n-m,m}\  \longrightarrow\ \Wei^{\lc}
    $$
    is a locally closed stratification.
    
    Afte passing to $\tWei^{\lc}$, we are reduced to checking that 
    $$
    \bigsqcup_{m=0}^n \varphi_m\  :\  \bigsqcup_{m = 0}^n \tWeik \ \longrightarrow\  \tWei^{\lc}
    $$
    is a locally closed stratification. It is surjective on closed points, because any lc Weierstrass data over $\bP^1_k$ can be factored as $A = A_0 z^4$ and $B = B_0 z^6$, where $(A_0, B_0)$ is minimal. Notice also that, at the level of geometric points, we have that $\varphi_m$ is stabilizer preserving (the stabilizer consist of $\mu_2$ generically, $\mu_4$ if $B = 0$ and $\mu_6$ if $A = 0$ for both the source and target stacks), so each $\varphi_m$ is representable. Furthermore, since the factorization $A = A_0 z^4$ and $B = B_0 z^6$, where $(A_0, B_0)$ is minimal, is unique (up to scaling) at the level of field-valued points, $\varphi_m$ is injective on geometric points. We claim that the morphism of smooth stacks $\varphi_m$ is unramified. 
    To see this, it suffices to check that $\varphi_m$ induces injections for tangent spaces for each geometric point. 

    Set $\bV_d := \oH^0(\cO_{\mathbb{P}^1}(d))$ for any given $d$. We note that $\tWeik$ may be viewed as an open substack of $$\mathcal{N}_1\ =\  \left[ (\bV_{4n-4m} \times \bV_{6n-6m})/\mathbb{G}_m \right] \times \left[ \bV_m/\mathbb{G}_m \right],$$ where the first copy of $\mathbb{G}_m$ acts with weights $4$ and $6$ on the vector spaces $\bV_{4n-4m}$ and $\bV_{6n-6m}$, and the second copy of $\mathbb{G}_m$ acts with weight $1$ on the vector space $\bV_m$. 
    A $k$-point of $\mathcal{N}_1$ consists of a triple of elements $(\alpha, \beta, z)$ in $\oH^0(\mathcal{O}_{\bP^1}(4n-4m)) \times \oH^0(\mathcal{O}_{\bP^1}(6n-6m)) \times \oH^0(\mathcal{O}_{\bP^1}(m))$, and the tangent space of the stack at that point is the cokernel of the linear morphism of $k$-vector spaces $$\psi_1\ \colon\ k^{\oplus 2} \longrightarrow  \oH^0(\mathcal{O}_{\bP^1}(4n-4m)) \oplus \oH^0(\mathcal{O}_{\bP^1}(6n-6m)) \oplus \oH^0(\mathcal{O}_{\bP^1}(m))$$ given by $(x,y) \mapsto (4x\alpha, 6x\beta, yz)$. 
    On the other hand, the stack $\widetilde{W}_{n}^{lc}$ can be viewed as an open substack of $$\mathcal{N}_2 \ =\  \left[ \left( \bV_{4n} \times \bV_{6n} \right) / \mathbb{G}_m \right],$$ where $\mathbb{G}_m$ acts with weights $4$ and $6$ on the vector spaces $\bV_{4n}$ and $\bV_{6n}$. 
    A $k$-point of $\mathcal{N}_2$ is given by a pair $(\delta, \gamma) \in (\oH^0(\mathcal{O}_{\bP^1}(4n)) \times \oH^0(\mathcal{O}_{\bP^1}(6n))$, and the corresponding tangent space is the cokernel of 
    \[\psi_2 \colon k \to \oH^0(\mathcal{O}_{\bP^1}(4n)) \oplus \oH^0(\mathcal{O}_{\bP^1}(6n))\]
    given by $x \mapsto (4x\delta, 6x\gamma)$. 
    We note that the morphism $\varphi_m$ extends to a morphism $\widetilde{\varphi}_m$ that sends a $k$-point $(\alpha, \beta, z)$ to $(\alpha z^4, \beta z^6)$. 
    A direct computation shows that the induced morphism of tangent complexes at the $k$-point $(\alpha, \beta, z)$ is given by
    \[ 
    \begin{tikzcd}
      k^{\oplus 2} \ar[d, "\mu"] \ar[r, "\psi_1"] & \oH^0(\mathcal{O}_{\bP^1}(4n-4m)) \oplus \oH^0(\mathcal{O}_{\bP^1}(6n-6m)) \oplus \oH^0(\mathcal{O}_{\bP^1}(m))\ar[d, "\nu"] \\ k \ar[r, "\psi_2"] &\oH^0(\mathcal{O}_{\bP^1}(4n)) \oplus \oH^0(\mathcal{O}_{\bP^1}(6n))
    \end{tikzcd}
    \]
    where $\mu$ is the surjective morphism $\mu(x,y) = x+y$, and we have $\nu(a,b,c) = (z^4\cdot a + 4\alpha z^3 \cdot c, z^6\cdot b + 6\beta z^5 \cdot c)$. Notice that, if the $k$-point $(\alpha, \beta, z)$ is in the open substack $\tWeik$, then we must have $(\alpha, \beta) \not \equiv (0,0), z \not \equiv 0$, and furthermore the three sections $\alpha, \beta, z$ don't have a joint simultaneous zero in $\mathbb{P}^1$ (because the divisor cut out by $z=0$ does not intersect the locus of cusps, which is the locus cut out by $\alpha,\beta =0$). It follows then from the description above that the kernel of $\nu$ is given by the subspace of $(a,b,c)$ satisfying the linear equations $z \cdot a + 4 \alpha \cdot c= 0$ and $z \cdot b + 6 \beta \cdot c =0$. \\
    
    \noindent \textbf{Claim:} If $(a,b,c)$ is in the kernel of $\nu$, then it is of the form $(4x \alpha, 6 x \beta, -xz)$ for some constant $x \in k$. \\
    
    Let us prove the claim. We may change coordinates in $\mathbb{P}^1$ so that none of $\alpha, \beta, z$ vanish at $\infty$, and then we view $\alpha, \beta, z$ as elements of $k[t]$. 
    We may assume furthermore  for the sake of the following argument that $\alpha, \beta, z$ are monic. 
    Set $g_1 := \gcd(\alpha, z)$ and $g_2 = \gcd(\beta, z)$. 
    The solutions to the equation  $z \cdot a + 4 \alpha \cdot c= 0$ are of the form $(a, c) = ( x p \cdot 4\alpha/g_1, -xp \cdot z/g_1)$, where $x \in k$ and $p$ is a monic polynomial. 
    Similarly, any solution of $z \cdot b + 6 \beta \cdot c =0$ is of the form $(b,c) = (y q (6\beta/g_2), -y q (z/g_2))$, where $y \in k$ and $q$ is a monic polynomial. 
    If we want $(a,b,c)$ to be a solution to both equations, then we have $-xp(z/g_1)= c =-y q (z/g_2)$, which, in view of the polynomials being monic, implies $x=y$. 
    Therefore, we have $-xp(z/g_1)= c =-x q (z/g_2)$, which implies that $p g_2 = q g_1$. 
    Now, the polynomials $g_1$ and $g_2$ cannot have a common root, because that would yield a simultaneous root of $\alpha, \beta$ and $z$, contradicting our assumptions. 
    Therefore $g_1$ and $g_2$ are coprime, and the equation $p g_2 = q g_1$ forces $p=g_1$ and $q=g_2$. 
    We conclude that $(a,b,c) = (4x \alpha, 6 x \beta, -xz)$, as claimed. \\
    
    Now, observe that any element of the kernel $(4x \alpha, 6 x \beta, -xz)$ is of the form $\psi_1(x,-x)$. 
    From this we conclude that the induced morphism on tangent spaces $\Coker(\psi_1) \to \Coker(\psi_2)$ is injective, and therefore $\varphi_m$ is unramified. By \cite{stacks_project}*{\href{https://stacks.math.columbia.edu/tag/05VH}{Tag 05VH}}, we conclude that $\varphi_m$ is a monomorphism. To complete the proof, we will apply the valuative criterion for locally closed embeddings (cf. \cite{Kol09}*{Prop. 42}). \\
    

  Let $T = \spec R$ be the spectrum of a DVR and let $T \to \tWei^{\lc}$ a map whose image is contained in $\varphi_m(\tWeik)$. This is equivalent to a family of lc Weierstrass data $(A,B)$ over $T$ such that for each $t \in T$, the pair $(A_t, B_t)$ defines a fibration with exactly $m$ elliptic singularities. Let $\eta \in T$ denote the generic point. There is a canonical subscheme $D_A$ of the generic fiber $\mathbb{P}^1_{\eta}$ whose support consists of the locus of points where the section $A_{\eta}$ has a zero of multiplicity $\geq 4$ (Zariski locally around every vanishing point of $A_{\eta}$, it is given by the vanishing of the $3^{rd}$ differential of the section). Similarly, there is a canonical subscheme $D_B$ where $B_{\eta}$ has a zero of multiplicity $\geq 6$. Consider the reduced subscheme $D_{\eta} := (D_A\cap D_B)^{red}$ of the intersection $D_A \cap D_B$, which is a Cartier divisor on $\mathbb{P}^1_{\eta}$. We denote by $D \subset \mathbb{P}^1_T$ the flat closure of the reduced subscheme $D_{\eta}$. By construction, there exist factorizations $A_{\eta} = (A_0)_{\eta} + 4D_{\eta}$ and $B_{\eta} = (B_0)_{\eta} + 6D_{\eta}$ for some Cartier divisors $(A_0)_{\eta}$ and $(B_0)_{\eta}$ on $\mathbb{P}^1_{\eta}$. By taking flat closures in $\mathbb{P}^1_T$, we get factorizations of relative Cartier divisors $A = A_0 + 4D$ and $B = B_0 + 6 D$. To conclude the proof of the valuative criterion, it suffices to show that the tuple $(A_0, B_0, D)$ yields a $T$-point of $\tWeik$. This amounts to showing that $D$ is etale over $T$, and that the fibers of Weierstrass fibration defined by $(A_0,B_0)$ are minimal. By the assumption that the original Weierstrass fibration $T \to \tWei^{\lc}$ was log canonical, it follows that for each geometric $T$-fiber there is no point where the section $A$ has a zero of order $\geq 5$ and simultaneously the section $B$ has a zero of order $\geq 7$. Therefore, from the factorizations $A = A_0 + 4D$ and $B = B_0 + 6 D$ it follows that the geometric $T$-fibers of $D$ are forced to be reduced, and hence $D$ is \'etale over $T$. On the other hand, by construction we have arranged so that at the generic point the pair $((A_0)_{\eta}, (B_0)_{\eta})$ yields a minimal Weierstrass fibration. By our assumption that the image of $\eta$ is contained in $\varphi_m(\tWeik)$, it follows that the order of the Cartier divisor $D_{\eta}$ is $m$. If we denote by $s \in T$ the special point, then the order of the special fiber $D_{s}$ is also $m$. Since the image of $s$ is also contained in $\varphi_k(\tWeik)$, the factorizations $A_s = (A_0)_s + 4D_s$ and $B = (B_0)_s + 6 D_s$ with $D_s$ of degree $m$ force $((A_0)_s, (B_0)_s)$ to define a minimal Weiertrass fibration (otherwise we would be able to factor out a further divisor $D'_s \supset D_s$, which would mean that $(A_s,B_s)$ is actually in the image of $\varphi_{m'}$ for some $m'>m$; this would be disjoint from $\varphi_m(\tWeik)$). We conclude that the tuple $(A_0, B_0, D)$ yields the desired $T$-point of $\tWeik$.
\end{proof}

\subsection{Locus of elliptic surfaces in the KSBA moduli} 

Let us start by fixing some notations.

\begin{Notation}
    For any rational number $0 < \epsilon < \frac{n-2}{2}$, we write $c(\epsilon):= \frac{n-2}{n}-\epsilon$ and $v(\epsilon) = \frac{(n-2)^2 - (n\epsilon)^2}{n}$.
    Then $v(\epsilon)$ is the volume of the pair $\big( X,c(\epsilon)S \big)$ as in \cref{lemma_when_KX+cS_is_ample_and_volume}. 
\end{Notation}
    
By \cref{lemma_when_KX+cS_is_ample_and_volume} and \cref{remark: easy_family}, 
for any $0 < \epsilon < \frac{n-2}{n}$, there is a morphism 
\begin{equation}
\label{equation: morphism Wei to KSBA}
    \Phi_{n,\epsilon}\ \colon\ \cW_n^{\lc}
  \  \longrightarrow \ \MKSBA_{c(\epsilon),v(\epsilon)},  
    \quad
    (X\to C, S)\ \mapsto\  \big(X, c(\epsilon)S\big),
\end{equation}
which forgets the fibration structure.
We note that by \cite[Theorem 1.1(b)]{ascher2021wall} and \cite{meng2023mmp}, there exist birational maps between normalizations $(\MKSBA_{c(\epsilon_2),v(\epsilon_2)})^\nu \to (\MKSBA_{c(\epsilon_1),v(\epsilon_1)})^\nu$ for $0 < \epsilon_2 < \epsilon_1 < \frac{n-2}{n}$.
These moduli stacks may parametrize different surfaces, for example if $\left(0, \frac{n-2}{n}\right)$ is not a single chamber.
However, even when these moduli stacks parametrize different surfaces, we don't claim that $\MKSBA_{c(\epsilon),v(\epsilon)}$ are not isomorphic for different choices of $\epsilon$.
Nonetheless, we will show now that $\Phi_{n,\epsilon}$ is an open immersion for any choice of $\epsilon \in \left(0, \frac{n-2}{n}\right)$, and we will study more in detail the normalization of the closure of $\Phi_{n,\epsilon}$ in Corollary \ref{cor_amm_moduli_spaces_are_isom}.
The parameter $\epsilon$ determines the target of the morphism $\Phi_{n,\epsilon}$.

\begin{Def} \label{defn: overline E_n epsilon}
    We denote by $\overline{\cE}_{n,\epsilon}$ the normalization of the scheme-theoretic image (cf. \cite{stacks_project}*{\href{https://stacks.math.columbia.edu/tag/0CMH}{Tag 0CMH}}) of $\Phi_{n,\epsilon}$.
\end{Def}

We now prove part (\ref{item: main: open immersion of Wn}) of \cref{thm:main}.

\begin{Prop}
\label{proposition: Phi defines iso of W with KSBA}
For any $n > 2$ and $0 < \epsilon < \frac{n-2}{n}$, the morphism 
$\Phi_{n,\epsilon} \colon \Wei^{\lc} \to \MKSBA_{c(\epsilon),v(\epsilon)}$
is an open immersion.
\end{Prop}

\begin{proof} 
    It suffices to prove that $\Phi_{n,\epsilon}$ is smooth, induces a bijection on automorphisms of geometric points, and is injective on isomorphism classes of geometric points, since this would jointly imply that it is an \'etale monomorphism, hence an open immersion. 

\textit{$\Phi_{n,\epsilon}$ is injective.}
This follows from \cite{inchiostro2020moduli}*{Lem. 3.6}. This lemma is stated only for minimal Weierstrass fibrations, however the argument works for our case of lc Weierstrass fibrations of Kodaira dimension $1$.

\textit{$\Phi_{n,\epsilon}$ is representable.} It suffices to check on geometric points. We need to show that if $\sigma$ is an automorphism of the Weierstrass fibration $(f \colon X\to \bP^1, S)$ which induces the identity on $(X,S)$, then it is the identity. Indeed, from the proof of injectivity above, the map $f$ is uniquely determined from $(X,S)$, and since $f$ is surjective, the only (set-theoretic) function $\sigma$ which makes the following diagram commutative is the identity, as desired
$$\xymatrix{X\ar[d]^f \ar[r]^\Id & X \ar[d]_f \\ \bP^1\ar[r]^\sigma &\bP^1.}$$

\textit{$\Phi_{n,\epsilon}$ is surjective on automorphisms of geometric points.} We need to check that any automorphism of $(X,S)$ extends to an automorphism of $(f:X\to \bP^1, S)$, namely, to one of the fibration. But this follows since $\omega_X=f^*\cO_\bP^1(n-2)$ so the global sections of $\omega_X$ induce the morphism $f$, and any automorphism of $X$ extends (by functoriality) to an automorphism of $\omega_X$ and therefore to its global sections.

\textit{$\Phi_{n,\epsilon}$ is smooth.} Smoothness is smooth local on the source so it suffices to check that the composition $\tWei^{\lc} \to \MKSBA$ is smooth. We use the infinitesimal lifting property (cf. \cite{stacks_project}*{\href{https://stacks.math.columbia.edu/tag/0DP0}{Tag 0DP0}}, \cite{stacks_project}*{\href{https://stacks.math.columbia.edu/tag/02HT}{Tag 02HT}}), which applies to our case since the source and the target are stacks locally of finite type over a field.
Let $A' \to A$ be a surjection of local Artinian rings  with residue field $k$ such that the kernel $I$ satisfies $I^2 = 0$.
Given a diagram of solid arrows as below, we will show that we can find a lifting, marked as a dotted arrow.
$$\xymatrix{\spec(A)\ar[d] \ar[r] & \tWei^{\lc}\ar[d] \\ \spec(A')\ar[r] \ar@{.>}[ru]& \MKSBA_{c(\epsilon),v(\epsilon)}}$$ 
In other words, we need to show that for any family $$\big(\cX',c(\epsilon)\cS'\big)\ \longrightarrow \ \spec A'$$ in $\MKSBA_{c(\epsilon),v(\epsilon)}$ whose restriction to $\spec(A)$ is a Weierstrass fibration, there is a morphism $\cX'\to \bP^1_{A'}$ such that $\big(\cX'\to \bP^1_A, \cS'\big)$ is a Weierstrass fibration. Namely, the deformation problem is the following:
$$\xymatrix{(\cX,\cS)\ar[d] \ar[r] & (\cX',\cS')\ar@{.>}[d] \\ \bP^1_A \ar@{.>}[r]\ar[d]& \bP^1_{A'}\ar[d]\\\spec(A) \ar[r] & \spec(A')}$$

From \cite{BHPS}*{Prop. 3.10}, it suffices to prove that $\Hom\big(\Omega^1_{\bP^1},R^1f_*\cO_X\big) = 0$. We have $R^1f_*\cO_X = \cO_{\bP^1}(-n)$ by Lemma \ref{push}, so 
$$\Hom(\Omega^1_{\bP^1},R^1f_*\cO_X)\  \simeq \  \oH^0\big(\mb{P}^1,\cO_{\bP^1}(-n+2)\big)\ =\ 0.$$
\end{proof}

\begin{Lemma}\label{push}
    Let $(f:X\rightarrow \mb{P}^1,S)$ be an lc Weierstrass elliptic fibration with $S^2=-n$. Then $$R^1f_{*}\cO _X\ = \ \cO _{\mb{P}^1}(-n).$$
\end{Lemma}

\begin{proof}
    Note that for all points $b \in \mathbb{P}^1$ we have $H^1(X_b, \cO_{X_b}) \cong k(b)$ for the arithmetic genus $1$ fiber, and hence it follows that $R^1f_*\cO_X$ is a line bundle by cohomology and base-change. This follows because $h^1(\bP^1, R^1f_*(\cO_X))= n-1$, as seen from the Leray spectral sequence for $f$ and the computations $h^2(X,\cO_X)=h^0(X,\omega_X)=n-1$, and $h^0(\mb{P}^1,R^2f_{*}\cO _X)\ =\ h^2(\mb{P}^1,f_{*}\cO _X)\ =\ 0.$
\end{proof}

    \begin{Notation}
        We denote by $\cE^{\lc}_n$ (resp. $\cE_n$) the image of $\Wei^{\lc}$ (resp.  $\Wei^{\min}$) under $\Phi_{n,\epsilon}$. We denote by $\Phi_n \colon \Wei^{\lc} \to \cE_n^{\lc}$ the corresponding isomorphism. 
    \end{Notation}

\begin{Cor}\label{cor_the_moduli_is_smooth}
The moduli space $\cE_n^{\lc}$ is smooth and irreducible. \qed
\end{Cor}

\begin{proof}
    This follows from Theorem \ref{theorem: construction of Wn} and \cref{proposition: Phi defines iso of W with KSBA}.
\end{proof}

\begin{Remark}\label{rem:walls}
    It follows from \cites{ascher2021wall, meng2023mmp} that there are finitely many critical values $$0 \ <\  t_1 \ <\  \cdots \ <\ t_m\ <\ \frac{n-2}{n}$$ 
    such that the compactification $\cE_n \subseteq \overline{\cE}_{n,\epsilon}$ of the image as in \Cref{defn: overline E_n epsilon} is independent of $\epsilon$ for $0 < t_i \le \epsilon < t_{i + 1}$ (resp. $0 < \epsilon < t_1$). We call $t_i$ the walls and denote by $\overline{\cE}_{n,t_i}$ the compactification $\overline{\cE}_{n, \epsilon}$ for $0< t_i \le \epsilon < t_{i +1}$ and by $\overline{\cE}_{n,0}$ the compactification $\overline{\cE}_{n, \epsilon}$ for $0 < \epsilon < t_1$. In this paper we will be most interested in $\overline{\cE}_{n,0}$ which parametrizes stable pairs $(X, \left(\frac{n-2}{n} - \epsilon\right)S)$ for $0 < \epsilon \ll 1$ and $\overline{\cE}_{n, t_m}$ which parametrizes stable pairs $(X, \epsilon S)$ for $0 < \epsilon \ll 1$. For convenience we denote this latter space by $\overline{\cE}_{n}$. We refer the reader to Corollary \ref{cor_amm_moduli_spaces_are_isom} for the relation between $\cE_{n,t}$ with $t$ varying.
\end{Remark}

\subsection{Locus of pseudo-elliptic surfaces in the KSB moduli}

Consider the moduli stack $\cE_n$ constructed in the previous subsection. The closed points  of $\cE_n$ parametrize KSBA-stable pairs $$\left(X,\bigg(\frac{n-2}{n} - \epsilon\bigg)S\right)$$ such that there is a map $X\to \bP^1$ which makes $(X\to \bP^1, S)$ a Weierstrass fibration. Moreover, we proved that these pairs above do not depend on $\epsilon$, as long $0<\epsilon<\frac{n-2}{n}$.
Therefore, it is natural to ask what happens if $\epsilon=0$.
In this case, there is a morphism $$\Psi_n \ \colon \ \cE_n\ \longrightarrow \ \MKSB_{v} $$
for $v=\frac{(n-2)^2}{n}$ which, on the level of points, sends a pair $(X, \left(\frac{n-2}{n} - \epsilon\right)S)$ to the log canonical model of $\big(X,\frac{n-2}{n}S\big)$, which is the surface obtained by contracting the negative section $S$ (cf. \cref{observation: canonical model for epsilon=0}). 
The construction of such a morphism follows from some of the results \cites{ascher2021wall, meng2023mmp}, which we recall here for the convenience of the reader.

The following constructs the natural morphism in part (\ref{item: main: open immersion of En}) of \cref{thm:main}.

\begin{Prop} 
\label{prop: existence of the morphism}
    Let $n\geq 3$ and $v=\frac{(n-2)^2}{n}$. Then there is a natural morphism $\Psi_n \colon \cE_n^{\lc}\to \MKSB_v$.
\end{Prop}

\begin{proof}
    Let $\left( X,\frac{n-2}{n}S \right)$ be a pair parametrized by a closed point of $\cE_n^{\lc}$.
    This is an lc pair whose lc centers are isolated points, namely the cusps of any strictly lc fibers.
    Since $K_X+\frac{n-2}{n}S$ is big and nef, then by Kawamata-Viehweg vanishing for lc pairs \cite{fujinoAG}*{Thm. 1.10} we have that $\oH^i \left( X, m(K_X+\frac{n-2}{n}S) \right)=0$ for any $m>0$ such that $m(K_X+\frac{n-2}{n}S)$ is an integral divisor and for any $i>0$.
    Then cohomology and base change applied to the universal family $\pi \colon (\cX,\cS)\to \cE_n^{\lc}$ implies that for any positive integer $m$ such that $m(K_X+\frac{n-2}{n}S)$ is an integral divisor, the formation of 
    $$\textstyle \cY = \operatorname{Proj}_{\cE_n^{\lc}}\left(\bigoplus_{d \geq 0} \pi_*\cO_{\cX}\left(dm\big(K_{\cX/\cE_n^{\lc}} + \frac{n-2}{n}\cS\big)\right)\right)$$
    commutes with base change, and the push-forwards $\pi_*\cO_{\cX}\left(dm(K_{\cX/\cE_n^{\lc}} + \frac{n-2}{n}\cS)\right)$ are vector bundles.
    It gives rise to a projective morphism $\cY\to \cE_n^{\lc}$ which fiberwise is the canonical model of $\left( X,\frac{n-2}{n}S \right)$, and which is flat since we are taking Proj of a $\cE_n^{\lc}$-flat algebra.
    Moreover, since the base is smooth (\cref{cor_the_moduli_is_smooth}) and the volume of every fiber is $v=\frac{(n-2)^2}{n}$ (\Cref{observation: canonical model for epsilon=0}), then the family $\cY\rightarrow \cE_n^{\lc}$ is a KSB stable family by \Cref{thm: kollar condition numerical}. Therefore, it induces a morphism $\cE_n^{\lc}\to \MKSB_v$ by the universality of $\cM_v^{\KSB}$.
\end{proof}

\begin{Def} \label{defn: P_n}
    We denote by $\Psell$ the image of $\Psi_n$, and by $\overline{\cP}_n$ the normalization of its scheme-theoretic closure.
\end{Def}

\section{Local study of the contraction morphism \texorpdfstring{$\Psi_n$}{Psin}}\label{section: main theorem}

In this section, we will prove that $\Psi_n \colon \cE_n^{\lc}\to \MKSB_v$ is an open immersion when $n=3$ or $n > 4$. To this end, we need a few observations on deformations of local models of the covering stack at the contracted point. We record the necessary lemmas first; they will be used to control the deformations of the canonical covering stack (see \Cref{thm_to_check_defs_we_look_at_covering_stack}) of a certain KSBA-stable surface.

\begin{Lemma}\label{lemma_small_def_of_local_sing_constant}
    Let $n$ be an even positive integer. Consider the action of $G = \bmu_n/\bmu_2$ on $Z = \spec(k[x,y,z]/(xy-z^2))$ defined by $\xi\cdot x = \xi^2x$, $\xi \cdot y = \xi^2y$ and $\xi \cdot z = \xi^2z$.
    \begin{enumerate}
    \item If $n > 4$ then
    $$
    \Ext^1_{\bmu_n/\bmu_2}(\Omega_{Z}^1,\cO_{Z}) = 0
    $$
    and the only infinitesimal deformation of $[Z/G]$ is the trivial one.
    \item If $n = 4$, then 
    $$
    \Ext^1_{\bmu_n/\bmu_2}(\Omega_{Z}^1,\cO_{Z}) = k
    $$
    and $[Z/G]$ admits a unique formal smoothing direction. 
    \end{enumerate}
\end{Lemma}
\begin{proof}Let $R:= k[x,y,z]/(xy-z^2)$ and consider the $G$-equivariant presentation of $\Omega_{Z}^1$ given by
$$R^{\oplus 1} =R e_f\xrightarrow{1\mapsto (y,x,-2z)} R^{\oplus 3} = R\od x\oplus R \od y \oplus R \od z\longrightarrow \Omega_{Z}^1\longrightarrow0.$$
 Observe that this is the sequence
\[
(xy-z)/(xy-z)^2\to \Omega^1_{\bA^3}\otimes k[x,y,z]/(xy-z)\to \Omega^1_{k[x,y,z]/(xy-z)}\to 0
\]
\noindent where we denoted by $e_f$ a generator of $(xy-z)/(xy-z)^2$.
 A generator $\xi\in \bmu_n$ acts with weight 2 on $x,y,z$, so it acts with weight 2 also on $\od x$, $\od y$ and $\od z$. In particular, it acts with weight $4$ on $xy-z^2$, so it acts also with weight 4 on $e_f$. To compute $\Ext^1_{\bmu_n/\bmu_2}(\Omega_{Z}^1,\cO_{Z})$, we dualize the sequence above
$$\psi:R(\od x)^{\vee}\oplus R (\od y)^{\vee} \oplus R (\od z)^{\vee} \longrightarrow R e_f^{\vee},$$
and the cokernel of $\psi$ is $k$. Observe that, as we dualized, the action is with with weight $n-4$, which is the inverse of $4$ in $\bmu_n/\bmu_2$. So there are no invariants if $n\neq 4$. If $n = 4$, then the invariants are one-dimensional and equal to the group $\Ext^1(\Omega_Z^1, \cO_Z)$ whose nonzero element corresponds to the unique smoothing direction of the $A_1$ singularity of $Z$.  
\end{proof}
\begin{Lemma}\label{lemma inertia for A2}
    Let $R$ be an Artinian local ring over $k$, and
    consider the action of $\bmu_n$ on $\bA^2_R = \spec(R[x,y])$ defined by $\xi \cdot x = \xi x$ and $\xi \cdot y = \xi y$. 
    Let $f \colon \cI \to \bA^2_R$ be the pullback of the inertia stack of $[\bA^2_R/\bmu_n]$ along the smooth cover $\bA^2_R \to [\bA^2_R/\bmu_n]$. 
    Then the scheme-theoretic support of the cokernel of $\cO_{\bA^2_R}\to f_*\cO_{\cI}$ is defined by the ideal $(x,y)$.
\end{Lemma}
\begin{proof}
    The group scheme $\cI \to \bA^2_R$ fits into the following fiber product:
    $$\xymatrix{\cI\ar[d]_f\ar[r] & \bA^2_R\times \bmu_n\ar[d] \\ \bA^2_R\ar[r]^-{\operatorname{diag}} & \bA^2_R\times \bA^2_R.}$$
    If we write $\bmu_n = \spec(k[t]/(t^n-1))$, then $\cI = \spec(R[x,y,t]/(t^n-1, tx-x, ty-y))$, and the morphism $f$ induces the inclusion $R[x,y]\to R[x,y,t]/(t^n-1, tx-x, ty-y)$ on global sections. In particular, we have \[f_*\cO_\cI/\cO_{\bA^2_R} \ =\ \bigoplus_{i=0}^{n-1} R \cdot t^i \ \cong\ (R[x,y]/(x,y))^{\oplus n}\]
    as an $R[x,y]$-module, which has scheme-theoretic support defined by the ideal $(x,y)$ as desired.
\end{proof}

\begin{Lemma}\label{lemma_support_inertia} Let $R$ be an Artinian local ring, and
    consider the action of $\bmu_n$ on $Z:=\spec(R[x,y,z]/(xy-z^2))$ defined by $\xi \cdot x = \xi x$, $\xi \cdot y = \xi y$ and $\xi \cdot z = \xi z$. Let $f\colon \cI \to Z$ be the pullback of the inertia stack of $[Z/\bmu_n]$ along $Z \to [Z/\bmu_n]$. Then the support of the cokernel of $\cO_{Z}\to f_*\cO_{\cI}$ is the vanishing of the ideal $(x,y,z)$.
\end{Lemma}
\begin{proof}
    The argument is the same as in Lemma \ref{lemma inertia for A2}: in this case we have that
    $$\cI = \spec(R[x,y,z,t]/(xy-z^2, t^n-1, tx-x, ty-y, tz-z))$$
    and $f$ is induced by the following inclusion on global sections $$R[x,y,z]/(xy-z^2)\to R[x,y,z,t]/(xy-z^2, t^n-1, tx-x, ty-y, tz-z).$$
    We can conclude similarly as before.
\end{proof}

We are now well-equipped to finish the proof of part (\ref{item: main: open immersion of En}) of \cref{thm:main}.

\begin{Teo}\label{thm_proof_psi_n_isom}
    Suppose that $n>2$ and $n \neq 4$ is an integer. Then the morphism 
    $$\Psi_n \ \colon\  \cE^{\lc}_n\ \longrightarrow \ \MKSB_v$$ 
    defined in \textup{\Cref{prop: existence of the morphism}} is an open immersion.
\end{Teo}
\begin{proof}
 In this proof, we will denote by $Y$ the pseudo-elliptic surface obtained by contracting the section $S$ of an elliptic surface $X$ with $S^2=-n$, and by $p$ be the point to which $S$ is contracted. From \cite{KM98}*{Remark 4.9 (2)} the singularity is determined from the dual graph associated to its minimal resolution, and by the local analysis of \Cref{lemma_when_KX+cS_is_ample_and_volume} we have that $Y$ has a $\frac{1}{n}(1,1)$-singularity at $p$.
 
\textit{$\Psi_n$ is injective on isomorphism classes of geometric points}.
It suffices to check that there is an inverse on sets of geometric points. As all the singularities of $Y$ away from $p$ are either Du Val or strictly log-canonical (see Remark \ref{remark_when_there_are_lc_sing_based_on_W_data}), we can identify $X$ with the minimal resolution of $Y$ around its unique klt but not canonical singularity, and $S$ with its exceptional divisor.

\textit{$\Psi_n$ is injective on automorphism groups}.
Let $(f \colon X\rightarrow \mb{P}^1,S)$ be a Weierstrass fibration associated to a geometric point of $\cE_n^{\lc}$ and let $Y$ be its image via $\Psi_n$.
It suffices to observe that any automorphism on $X$ which induces the identity on $Y$ agrees with the identity on the dense open subset $X\setminus S$, so it must be the identity.

\textit{$\Psi_n$ is surjective on automorphism groups}. This is because any automorphism of $Y$ must fix $p$ as it is the unique strictly klt singularity. Thus, any such automorphism lifts to the blowup of $Y$ around $p$, which is exactly $(X,S)$.

\textit{$\Psi_n$ is smooth}: we apply the infinitesimal lifting property of smoothness (cf. \cite{stacks_project}*{\href{https://stacks.math.columbia.edu/tag/0DP0}{Tag 0DP0}}, \cite{stacks_project}*{\href{https://stacks.math.columbia.edu/tag/02HT}{Tag 02HT}}). We need to show that if $A'\to A$ is a quotient of an Artinian local rings with residue field $k$ and with square-zero ideal, and if we have a diagram of solid arrows as below, then one can find the dotted arrow:
$$\xymatrix{\spec(A)\ar[d] \ar[r] & \cE_n^{lc}\ar[d] \\ \spec(A')\ar[r] \ar@{.>}[ru]& \MKSB_v.}$$ The morphism $\spec(A')\to \MKSB_v$ induces a KSB-stable family $Y'\to \spec(A')$, and let $\cY'\to \spec(A')$ its covering stack (cf. \Cref{def_covering_stack}). 
As mentioned at the beginning of the proof, the singularity at $p$ is formally locally isomorphic to
$$\big(0\in \spec k[\![x,y]\!]/\bmu_n\big)$$ with the action $\xi \cdot x = \xi x$ and $\xi \cdot y = \xi y$. 

The canonical covering stack $\cY_k\to Y_k$, on a neighborhood of $p$, is formally locally isomorphic to:
\begin{enumerate}
    \item $[\spec(k[\![x,y]\!])/\bmu_n]$ for $n$ odd with the action $\xi \cdot x = \xi x$ and $\xi \cdot y = \xi y$; and
    \item $[\spec(k[\![x,y,z]\!]/(xy-z^2))/(\bmu_n/\bmu_2)]$ for $n$ even, with the action $\xi \cdot x = \xi^2 x$, $\xi \cdot y = \xi^2 y$ and $\xi \cdot z = \xi^4 z$.
\end{enumerate}
Indeed, formally locally around $p$, the covering stack $\cY_k$ is the relative coarse moduli space of the map $[\spec k[\![x,y]\!]/\bmu_n]\to {\bf{B}} \mathbb{G}_m$
given by the line bundle with section $\operatorname{d} x\wedge\operatorname{d}y$. More explicitly, it is the stacky quotient of $\spec k[\![x,y]\!]$
by the kernel of the representation of $\bmu_n$ on $\operatorname{d}x\wedge\operatorname{d}y$. As $\xi$ acts on $\operatorname{d}x\wedge\operatorname{d}y$
as $$\xi \cdot(\operatorname{d}x\wedge\operatorname{d}y)= \xi^2\operatorname{d}x\wedge\operatorname{d}y,$$ the kernel is trivial if $n$ is odd and $\bmu_2$
 if it is even.

 Since $[\spec(k[\![x,y]\!])/\bmu_n]$ is smooth, by \Cref{lemma_small_def_of_local_sing_constant}(1), the small deformations of the analytic local singularity of $\cY_k$ at the preimage of $p$ are trivial for $n>2$ and $n\neq 4$. In particular, for $n$ odd, there is roof diagram of pointed stacks as follows, with all the arrows \'etale and inducing an isomorphism on automorphisms groups:
 $$ \left([\spec(A'[x,y])/\bmu_n], 0\right)\longleftarrow (U,u) \longrightarrow (\cY',y).$$
 Similarly, for $n$ even, we have a diagram as follows 
  $$ ([\spec(A'[x,y,z]/(xy-z^2))/(\bmu_n/\bmu_2)], 0)\longleftarrow (W,w) \longrightarrow (\cY',y).$$
From Lemma \ref{lemma inertia for A2} and Lemma \ref{lemma_support_inertia}, the support of the inertia is the closed substack which on $U$ is the pull-back of the vanishing of $(x,y)$ and on $W$ is the vanishing of $(x,y,z)$. 
Hence, we can perform the blow-up along the closed substack given by the support of the inertia stack, which \'etale locally corresponds to performing the blow up of $(x,y)$ in  $[\spec(A'[x,y])/\bmu_n]$ and $(x,y,z)$ in  $[\spec(A'[x,y,z]/(xy-z^2))/(\bmu_n/\bmu_2)]$. In particular, from the analogous computation on the local charts, this blow-up, denoted by $\cX'\to \cY'$, is flat and commutes with base change.
Taking the coarse moduli space commutes with base change, so if $\cX'\to X'$ is the coarse moduli space of $\cX'$, then $\cX_A:=\cX'_A\to X_A:=X'_A$ is the coarse moduli space of $\cX_A$.
Then it follows from \cite{kollar2009lectures}*{Section 2.4 page 86} that the surface $X_A$ is a minimal resolution of $Y_A$ around the strictly klt singularity: we have that $X_A$ is the elliptic surface associated to $\spec(A)\to \cE_n^{\lc}$. The desired family giving the morphism $\spec(A')\to \cE_n^{\lc}$ is the pair $(X',S')\to \spec(A')$ where $S'$ is the coarse moduli space of the exceptional divisor of $\cX'\to \cY'$.

In particular, $\MKSB_v$ is smooth along the image of $\Psi_n$, which is open since $\Psi_n$ is smooth. Therefore, it follows that $\Psi_n$ is an \'etale monomorphism, and hence an open immersion.
\end{proof}

\begin{Cor} \label{coroll: overline P is irreducible}
    The scheme-theoretic image $\overline{\cP}_n$ (see \Cref{defn: P_n}) is an irreducible component of $\MKSB_v$ for $n = 3$ and $n > 4$. 
\end{Cor}
\begin{proof}
This is a consequence of \Cref{thm_proof_psi_n_isom}, in view of the irreducibility of $\cE_n^{\lc}$ from \Cref{cor_the_moduli_is_smooth}.
\end{proof}

\begin{Remark}\label{rem:n=4}
    In fact, the conclusion of the previous corollary holds even for $n = 4$. Indeed, the proof of Theorem \ref{thm_proof_psi_n_isom} shows that given any $\mathbb{Q}$-Gorenstein deformation $Y' \to \spec A'$ of $Y$ which induces the trivial deformation of the $\frac{1}{n}(1,1)$ singularity at $p$, we can blow up the singular locus in the family of canonical covering stacks to obtain a deformation of elliptic surfaces whose pseudoelliptic contraction yields $Y' \to \spec A'$. In particular, deformations which are locally trivial around $p$ are in the image of $\Psi_n$ for any $n$. Thus, if $\overline{\cP}_4$ is not an irreducible component of $\MKSB_v$, then there is a $\mathbb{Q}$-Gorenstein deformation of $Y$ over a curve which is not locally trivial around $p$ and thus induces a non-constant deformation of the $\frac{1}{4}(1,1)$. By \ref{lemma_small_def_of_local_sing_constant}(2), this deformation must smooth the $\frac{1}{4}(1,1)$ singularity and thus is a $\mathbb{Q}$-Gorenstein smoothing of $Y$, but it is well known that $Y$ cannot admit a projective $\bQ$-Gorenstein smoothing since such a smoothing would violate the Noether inequality: indeed, one has $$K^2=1,\ \textup{and}\ \ p_g=h^0(X,K_X)=3.$$
    On the other hand, it is an interesting question whether this irreducible component of $\MKSB_v$ has some non-trivial non-reduced structure for $n = 4$ which is larger than the scheme structure on $\overline{\mathcal{P}}_4$. 
\end{Remark}

\section{The case when \texorpdfstring{$n=3$}{n=3}} \label{section: n=3}

In this section, we present a more explicit proof of some of our results in the case when $n=3$. Throughout most of this section, we will adopt the following assumptions.

\begin{context}\label{notation_section_n=3_are_generic}
     Let $(g \colon X\to \bP^1,S)$ be a Weierstrass fibration with $X$ smooth, with $S^2=-3$ and with $36$ singular nodal fibers. Let $\pi \colon X\to Y$ be the contraction of $S$, and $p_i$ be the nodal points on the singular fibers of $g$.
\end{context}

\begin{Lemma}\label{lemma_notation_gives_an_open_condition}
    The locus in $\cW_3^{\min}$ where the conditions of \Cref{notation_section_n=3_are_generic} are satisfied is open.
\end{Lemma}

\begin{proof}
    A Weierstrass fibration $(g : X \to \bP^1, S)$ of height $3$ has $36$ nodal singular fibers if and only if all the fibers are of Kodaira type $\mathrm{I}_1$. The condition of $g$ having only $\mathrm{I}_1$ fibers is equivalent to $g$ having only nodal fibers and $X$ being smooth. Thus, the locus satisfying the required condition is exactly the intersection $\cW_3^{\nod} \cap \cW_3^{\reg} \subset \cW_3^{\min}$ which is open by Theorem \ref{theorem: construction of Wn}(2).
\end{proof}
The following two exact sequences in the setting of \Cref{notation_section_n=3_are_generic} will be useful:
\begin{equation}\label{eq_4}
    0\longrightarrow \Omega^1_{X/\bP^1}\longrightarrow \omega_{g}\longrightarrow\bigoplus_{i=1}^{36}k_{p_i}\longrightarrow0,
\end{equation}
\begin{equation}\label{eq_5}
    0\longrightarrow g^{*}\omega_{\bP^1}\longrightarrow \Omega^1_X\longrightarrow \Omega^1_{X/\bP^1}\longrightarrow0,
\end{equation}
Note that \eqref{eq_4} follows from the description of the dualizing sheaf of nodal curves, whereas \eqref{eq_5} follows from the Zariski exact sequence for cotangent sheaves jointly with the fact that $g:X \to \mathbb{P}^1$ is generically smooth and the source $X$ is integral.

We begin with the following preliminary computations. 
\begin{Lemma}
\label{lemma_cohomol_facts_using_LSS}
    In the situation of \Cref{notation_section_n=3_are_generic}, we have the following:
    \begin{align}
  &\Ext^1(g^*\omega_{\bP^1},\cO_X) = \oH^1(\bP^1,\mathcal{O}_{\bP^1}(2))\oplus \oH^0(\bP^1,\mathcal{O}_{\bP^1}(-1)) =  0 \\
        &\Ext^1(\omega_g,\mathcal{O}_X)=\oH^1(\bP^1,\mathcal{O}_{\bP^1}(-3))\oplus \oH^0(\bP^1,\mathcal{O}_{\bP^1}(-6))\simeq k^{\oplus2} \\       
        &\oH^1(X,g^*\omega_{\bP^1}\otimes g^*\cO_{\bP^1}(1)) =  \oH^1(\bP^1,\mathcal{O}_{\bP^1}(-1))\oplus \oH^0(\bP^1,\mathcal{O}_{\bP^1}(-4)) = 0 
    \end{align}
\end{Lemma}

\begin{proof} These follow from the Leray spectral sequence for $g : X \to \bP^1$, the fact that $\omega_g = g^*\cO_{\bP^1}(3)$, the fact that $R^1g_*\cO_X = \cO_{\bP^1}(-3)$, and the projection formula. We will compute the first one to illustrate this. First note that 
$$
\Ext^1(g^*\omega_{\bP^1}, \cO_X) = H^1(X, g^*\omega^{\vee}_{\bP^1}) = H^1(X, g^*\cO_{\bP^1}(2)). 
$$
Then by the projection formula, $g_*g^*\cO_{\bP^1}(2) = \cO_{\bP^1}(2)$ and $R^1g_*g^*\cO_{\bP^1}(2) = \cO_{\bP^1}(-1)$ and the Leray spectral sequence yields the following exact sequence and the claim. 
$$
\xymatrix{0 \to H^1(\bP^1, \cO_{\bP^1}(2)) \to H^1(X, g^*\cO_{\bP^1}(2)) \to H^0(\bP^1, \cO_{\bP^1}(-1)) \to 0}
$$
\end{proof}

\begin{Lemma}
\label{3} 
    In the situation of \Cref{notation_section_n=3_are_generic}, we have 
    $$h^2(X,T_X(-S))=0\text{ }\text{ and }\text{ }h^1(X,T_X)=30.$$ 
\end{Lemma}

\begin{proof} \quad \newline
\textbf{Proof of $h^2(X,T_X(-S))=0$}.
From Serre duality, it suffices to show that $h^0(X,\Omega_X^1\otimes\omega_X(S))=0$. Twisting (\ref{eq_5}) by $\omega_X(S)$ and taking the induced long exact sequence, we get the following.
$$0\to \oH^0(X,g^*\omega_{\bP^1}\otimes\omega_X(S))\to \oH^0(X, \Omega_X^1\otimes\omega_X(S))\longrightarrow \oH^0(X, \Omega_{X/\bP^1}^1\otimes\omega_X(S)) \longrightarrow \oH^1(X, g^*\omega_{\bP^1}\otimes\omega_X(S))$$

Note that $g_*\cO_{X}(S) = g_*\cO_X = \cO_{\bP^1}$ \cite{canning2022integral}*{Lem. 2.5}, $R^1g_*\cO_X(S)=0$ by cohomology and base change, and $\omega_X=g^*\cO_{\bP^1}(1)$ by the canonical bundle formula (\Cref{thm: canonical bundle formula}). By the projection formula, $g_*(g^*\omega_{\bP^1} \otimes \omega_X(S)) = \cO_{\bP^1}(-1)$ so $$\oH^0(X,g^*\omega_{\bP^1}\otimes\omega_X(S)) = \oH^0(\bP^1,\cO_{\bP^1}(-1))=0.$$
Similarly, from the Leray spectral sequence, we have $\oH^1(X, g^*\omega_{\bP^1}\otimes\omega_X(S)) = \oH^1(\bP^1, \cO_{\bP^1}(-1))=0$. Therefore
$$\oH^0(X,\Omega_X^1\otimes\omega_X(S))=\oH^0(X,\Omega_{X/\mb{P}^1}^1\otimes\omega_X(S)).$$ So it suffices to show that $\oH^0(X,\Omega_{X/\mb{P}^1}^1\otimes\omega_X(S))=0$. Twisting (\ref{eq_4}) by $\omega_X(S)$ and taking the associated long exact sequence, leads to
$$0\to \oH^0(X,\Omega_{X/\mb{P}^1}^1\otimes\omega_X(S))\to \oH^0(X,\omega_{g}\otimes \omega_X(S)) \xrightarrow{\alpha}\bigoplus_{i=1}^{36} k_{p_i}.$$
    But $$\oH^0(X,\omega_{g}\otimes \omega_X(S))= \oH^0(X,g^{*}\cO _{\mb{P}^1}(3)\otimes g^{*}\cO _{\mb{P}^1}(1)(S))=\oH^0(\bP^1,\cO _{\mb{P}^1}(4)\otimes g_{*}\cO_X(S))=\oH^0(\bP^1, \cO _{\mb{P}^1}(4)).$$ Note that the map $\alpha$ takes a section of the line bundle to its restriction to each of the $36$ nodes $p_i$. Under the identification $\oH^0(X, \omega_g \otimes \omega_X(S)) = \oH^0(\bP^1, \cO _{\mb{P}^1}(4))$, the restriction to the node $p_i$ corresponds to the restriction of the corresponding section of $\cO_{\mathbb{P}^1}(4)$ to the image $g(p_i)$. Hence we can identify the map $\alpha$ as the evaluation of a quartic polynomial on $\bP^1$ on the 36 points on $\bP^1$ whose fibers via $g$ are singular. Then $\alpha$ is injective, as a quartic polynomial on $\mb{P}^1$ that vanishes at $36$ points has to be zero. Hence, we have the desired vanishing.\\

\noindent \textbf{Proof of $h^1(X,T_X)=30$}.
Observe first that $h^2(X,T_X)=0$.
Indeed, we have an exact sequence $$0\to T_X(-S)\to T_X \to (T_X)|_S\to 0.$$
As $S$ has dimension 1, we have that $h^2(S,(T_X)|_S)=0$. The desired vanishing follows from the previous point, and the long exact sequence in cohomology.
Moreover, we have
 \begin{itemize}
     \item $\Ext^1(\oplus_{i=1}^{36} k_{p_i},\mathcal{O}_X)$=0,
     \item $\Ext^2(\Omega_{X/\bP^1}^1,\mathcal{O}_X)=0$, and
     \item $\oH^0(X,\Omega_X^1)=0$.
 \end{itemize}
 Indeed, the first bullet point follows from the local-to-global spectral sequence for $\Ext$, and since the points $p_i$ are smooth points of the surface $X$. The second one follows by applying $\Hom(\bullet,\mathcal{O}_X)$ to the short exact sequence (\ref{eq_5}), Lemma \ref{lemma_cohomol_facts_using_LSS}(7), and the fact that $\Ext^2(\Omega_{X}^1,\mathcal{O}_X)=\oH^2(X,T_X)=0$ we just proved. The third bullet point follows since there is an injection $\oH^0(X,\Omega_X^1)\to \oH^0(X,\Omega_X^1\otimes \omega_X(S))$ and above we prove that the latter is 0.

Then, if we apply $\Hom(\bullet,\mathcal{O}_X)$ to the short exact sequence (\ref{eq_4}), using the previous two vanishings we get $$0\longrightarrow \Ext^1(\omega_g,\mathcal{O}_X)\longrightarrow \Ext^1(\Omega_{X/\bP^1}^1,\mathcal{O}_X)\longrightarrow \Ext^2(\oplus_{i=1}^{36} k_{p_i},\mathcal{O}_X)\longrightarrow \Ext^2(\omega_g,\mathcal{O}_X)\longrightarrow0.$$

 From Lemma \ref{lemma_cohomol_facts_using_LSS}(8), we have $\ext^1(\omega_g,\mathcal{O}_X)= 2$ and $\ext^2(\omega_g,\mathcal{O}_X)=5$, where we denote by $\ext$ the dimension of the corresponding $\Ext$ group. By Serre duality, $\ext^2(\oplus k_{p_i},\mathcal{O}_X)=h^0(X,\oplus k_{p_i})=36$, therefore $$\ext^1(\Omega^1_{X/\bP^1},\mathcal{O}_X)=33.$$
    Applying $\Hom(\bullet,\mathcal{O}_X)$ to the sequence (\ref{eq_5}), using $\oH^0(X,\Omega_X)=0$ and Lemma \ref{lemma_cohomol_facts_using_LSS}(7), we obtain $$0\longrightarrow \Hom(g^{*}\omega_{\bP^1},\mathcal{O}_X)\longrightarrow \Ext^1(\Omega^1_{X/\bP^1},\mathcal{O}_X)\longrightarrow\Ext^1(\Omega^1_{X},\mathcal{O}_X)\longrightarrow0.$$ As we have $$\Hom(g^{*}\omega_{\bP^1},\mathcal{O}_X)=\oH^0(\bP^1,\mathcal{O}_{\bP^1}(2)\otimes g_{*}\mathcal{O}_X)=\oH^0(\bP^1,\mathcal{O}_{\bP^1}(2))\simeq k^{\oplus 3},$$ then $h^1(X,T_X)=33-3=30.$  
\end{proof}

\begin{Cor}\label{cor_vanishing_tg_Y}
    In the situation of \Cref{notation_section_n=3_are_generic}, we have $$h^1(Y,T_Y)=28,\quad \textup{and} \quad h^2(Y,T_Y)=0,$$
    where we denote $T_Y:=\pi_*T_X$.
\end{Cor}
\begin{proof}
First, we prove that $R^1\pi_{*}T_X\cong k_p^{\oplus 2}$, where $p\in Y$ is the point to which $S$ is contracted. As $\pi$ is an isomorphism away from $p$, the sheaf $R^1\pi_{*}T_X$ is a skyscraper sheaf supported at $p$. We now compute its length.
     Let $S_m$ be the $m$-th thickened neighborhood of $S$, and let $\mathcal{I}$ be the ideal sheaf of $S=S_1$ in $X$. Notice that we have $$\mathcal{I}/\mathcal{I}^2\simeq \mathcal{N}_{S/X}^{*}\simeq \mathcal{O}_{\mb{P}^1}(3),\quad \textup{and}\quad \mathcal{I}^m/\mathcal{I}^{m+1}\simeq (\mathcal{I}/\mathcal{I}^2)^m\simeq \mathcal{O}_{\mb{P}^1}(3m).$$ Also, notice that $T_S\simeq \cO _{\mb{P}^1}(2)$ and we have the exact sequence $$0\longrightarrow T_S\longrightarrow T_X|_S\longrightarrow \mtc{N}_{S/X}\longrightarrow0,$$ then $T_X|_S\simeq \cO _{\mb{P}^1}(2)\oplus \cO _{\mb{P}^1}(-3)$ since the extension has to be trivial (as $\Ext^1(\cO_{\bP^1}(-3), \cO_{\bP^1}(2))=0$). Taking cohomology of the exact sequence $$0\longrightarrow \left(\mathcal{I}^m/\mathcal{I}^{m+1}\right)\otimes 
    T_X\longrightarrow T_X|_{S_{m+1}}\longrightarrow T_X|_{S_m}\longrightarrow0,$$ we see that $\oH^1(T_X|_{S_{m+1}})$ is canonically isomorphic to $\oH^1(T_X|_{S_m})$, which is in turn isomorphic to $\oH^1(T_X|_{S})=k^{\oplus 2}$. It follows from the theorem on formal functions that $$\widehat{\left(R^1\pi_{*}T_X\right)}_p\ =\ \varprojlim_m \oH^1(S_m,T_X|_{S_m})\ \simeq\ k^{\oplus2}.$$ 
    
Taking the five-term exact sequence associated to the Leray spectral sequence $$E_2^{p,q}:=\oH^q(Y,R^p\pi_{*}T_X)\ \Rightarrow\  \oH^{p+q}(X,T_X),$$ we get an exact sequence $$0\rightarrow \oH^1(Y,T_Y)\rightarrow \oH^1(X,T_X)\rightarrow \oH^0(Y,R^1\pi_{*}T_X)\rightarrow \oH^2(Y,T_Y)\rightarrow \oH^2(X,T_X).$$ Since $\oH^2(X,T_X(-S))=0$ by Lemma \ref{3}, the map $\oH^1(X,T_X)\rightarrow \oH^1(X,T_X|_S)$ is surjective. As a consequence, the morphism $$\oH^1(X,T_X)\longrightarrow \oH^0(Y,R^1\pi_{*}T_X)$$ is surjective. The desired statement now follows from Lemma \ref{3}.
\end{proof}

\begin{Prop}\label{prop_tg_Y_is_reflexive} In the situation of \Cref{notation_section_n=3_are_generic}, set $\cU:=X\setminus S$ and denote by $i:\cU\to X$ the corresponding inclusion. Then we have $\pi_*T_X = \pi_*i_*T_{\cU}$.
\end{Prop}

\begin{proof}
    Since we have $\omega_X=g^*\cO_{\bP^1}(1)$, then there is an exact sequence
    $$0\longrightarrow \omega_X^{-1}\longrightarrow \cO_X\longrightarrow \cO_F\longrightarrow 0$$
    where $F$ is a (general) fiber of $g$. Twisting the previous sequence by $\Omega_X^1$ and applying $T_X\simeq \Omega^1_X\otimes \omega_X^{-1}$, we get that
    $$0\longrightarrow T_X\longrightarrow \Omega_X^1\longrightarrow\cO_F\otimes \Omega_X^1\longrightarrow 0.$$
    For any given Zariski open subset $V \subset X$ containing $S$, consider the following diagram where the vertical arrows are restrictions to $U:=V\setminus S$
    $$\xymatrix{0 \ar[r] & \Gamma(V,T_X)\ar[d] \ar[r] & \Gamma(V,\Omega_X^1)\ar[r] \ar[d]_\alpha & \Gamma(V\cap F,\Omega_X^1|_F)\ar[d]_\beta  \\0\ar[r] & \Gamma(U, T_U)\ar[r] & \Gamma(U, \Omega_U^1)\ar[r] & \Gamma(U\cap F, \Omega_X^1|_{F\cap U})}$$
    The map $\beta$ is injective, as $\Omega_X^1$ is a vector bundle on an integral scheme, so the restriction to the generic point is an injective morphism. The morphism $\alpha$ is an isomorphism by \cite{GKKP}*{Obs. 1.3, Thm. 1.4}, so from diagram chasing the first map is an isomorphism. We have thus verified $\pi_*T_X = \pi_*i_*T_{\cU}$. 
\end{proof}

\begin{Prop}\label{prop_K3_smooth_and_of_dim_28}
    The moduli stack $\Psell[3]$ is smooth of dimension 28 at the points $Y$ arising from \Cref{notation_section_n=3_are_generic} as above.
\end{Prop}
\begin{proof}
From Theorem \ref{thm_to_check_defs_we_look_at_covering_stack}, if we denote by $\cY$ the covering stack of $Y$, it suffices to check that
\begin{enumerate}
    \item $\cY$ is smooth,
    \item $h^2(\cY,T_\cY)=0$, and
    \item $h^1(\cY,T_\cY)=28$.
\end{enumerate}
Indeed, (1) follows from the proof of \Cref{thm_proof_psi_n_isom}. The second one holds true since, if we denote by $q:\cY\to Y$ the coarse space map, then $$\oH^2(\cY,T_\cY)=\oH^2(Y,q_*T_\cY), \ \ \ \textup{and} \ \ \ \oH^1(\cY,T_\cY)=\oH^1(Y,q_*T_\cY)$$ as $q_*$ is exact (cf. \cite{AV_compactifying}*{Lem. 2.3.4}). But as $\cY$ is smooth, if we denote by $j \colon \cU\to \cY$ the inclusion of the schematic locus of $\cY$ (namely, $\cY$ without a single point), then $$T_\cY=j_*T_{\cU} \quad \text{ and }\quad q_*T_\cY=q_*j_*T_{\cU}.$$
From Proposition \ref{prop_tg_Y_is_reflexive}, $q_*T_\cY=\pi_*T_X$, and the desired statements now follow from Corollary \ref{cor_vanishing_tg_Y}.
\end{proof}
\begin{Teo}\label{thm:embedding_3}
    The morphism $\Psi_3:\cE^{\lc}_3\rightarrow \cM^{\KSB}_{\frac{1}{3}}$ is an open embedding at the points $(X,S)$ as in \Cref{notation_section_n=3_are_generic}.
\end{Teo}
\begin{proof}We use
Zariski's main theorem, which requires that $\Psi_3$ is representable, injective, and birational.

To check that $\Psi_3$ representable we show it is injective on automorphisms. We need to check that if $\sigma$ is an automorphism of $(X,S)$ which induces the identity on $Y$, then it is the identity. This is clear, as if $\sigma$ is the identity on a dense open subset (namely, the complement of $S$), then it has to be the identity.

To check that $\Psi_3$ is injective: we can construct $X$ as the minimal resolution of $Y$, and $S$ is the exceptional divisor. Since the minimal resolution is unique, $\Psi_3$ is injective.

To check that $\Psi_3$ is birational: from Proposition \ref{prop_K3_smooth_and_of_dim_28} the dimension of $\Psell[3]$ is 28. From the isomorphism between $\cE_3$ and $\cW_3$ (by the definition of $\cE_3$ via the open immersion in \Cref{proposition: Phi defines iso of W with KSBA}), we have
$\dim(\cE_3)=\dim(\cW_3)$, and the latter is 28 from the explicit description of $\cW_3$ as a quotient stack given in \cite{canning2022integral}*{Section 2}. The equality of dimensions jointly with injectivity on geometric points shows that $\Psi_3$ is birational at the level of coarse moduli spaces.

To show it is birational at the level of stacks, we use \cite[Theorem A.5]{asgarli2019picard}, which requires that $\Psi_3$ is an isomorphism onto its image on the groupoid of $k$-points, up to shrinking the domain of $\Psi_3$ where it is birational at the level of coarse moduli spaces. For doing so, as we know that $\Psi_3$ is bijective on geometric points and injective on automorphisms, it suffices to prove it is surjective on automorphisms. It suffices to observe that any automorphism of $Y$ will send the singular point to itself, and so it will induce an automorphism of the blow-up of the singular point (namely $X$).

As both $\cE_3$ and $\Psell[3]$ are smooth (hence normal), from Zariski's main theorem the map $\Psi_3$ is an open embedding at the points $(X,S)$ as above (which is an open locus in $\cE_3$ from \Cref{lemma_notation_gives_an_open_condition}).
\end{proof}
\subsection{Case \texorpdfstring{$n>3$}{n>3}}
It is natural to wonder if the (naive) purely cohomological methods of this section can be applied even when $n>3$. Unfortunately, the answer is no.

For this subsection $X$ will be a smooth Weierstrass fibration with $12n$ singular fibers, and $n\ge 4$.
As before, we have the exact sequence 
\begin{equation}\label{1}
    0\longrightarrow g^{*}\omega_{\bP^1}\longrightarrow \Omega^1_X\longrightarrow \Omega^1_{X/\bP^1}\longrightarrow0,
\end{equation}
where the points $p_i$ are the nodes in the singular fibers of $g$.

\begin{Lemma}\label{lemma1}
    The cohomology group $H^2(X,T_X)$ is non-zero when $n \geq 4$. 
\end{Lemma}
    
\begin{proof}
    Using Serre duality, it suffices to check that $H^0(X,\Omega^1_X\otimes \omega_X)\neq 0$. Twisting (\ref{1}) by $\omega_X$ and taking the induced long exact sequence on cohomology, one sees that $H^0(X,\Omega^1_X\otimes \omega_X)$ contains a subgroup 
    $H^0(X,\omega_X\otimes g^{*}\cO_{\bP^1}(-2))$. Using that $\omega_X \simeq g^* \cO_{\bP^1}(n-2)$ and $g_* \cO_X \simeq \cO_{\bP^1}$, we conclude that:
    \[
        H^0(X,\omega_X\otimes g^{*}\cO_{\bP^1}(-2)) \simeq
        H^0(\bP^1,\cO_{\bP^1}(n-4)).
    \]
    and hence this group is nonzero for $n\geq 4.$
\end{proof}

\begin{Lemma}\label{lemma2}
    The cohomology group $\Ext^2(\Omega^1_X(\log S),\mathcal{O}_X)$ is non-zero when $n\geq 4$.
\end{Lemma}
Recall that the obstruction to a deformation of $(X,S)$ lies in $\Ext^2(\Omega^1_X(\log S),\mathcal{O}_X)$.

\begin{proof}
    We utilize the long exact sequence obtained by applying $R\Hom(-, \mathcal{O}_X)$ to the short exact sequence 
    $$
    0\longrightarrow \Omega^1_X\longrightarrow \Omega^1_X(\log S)\longrightarrow \mathcal{O}_S\longrightarrow 0,
    $$
    the last terms of which are 
    $$
    \cdots \longrightarrow  \Ext^2(\mathcal{O}_S, \mathcal{O}_X)\longrightarrow \Ext^2(\Omega^1_X(\log S),\mathcal{O}_X)\longrightarrow \Ext^2(\Omega_X^1, \mathcal{O}_X)\longrightarrow 0.
    $$
    The final term $\Ext^2(\Omega_X^1, \mathcal{O}_X)\cong H^2(X, T_X)$ is non-zero when $n\geq 4$ by Lemma \ref{lemma1} and the result follows.
\end{proof}

The non-vanishing of the cohomology groups $ H^2(X, T_X) $ and $\Ext^2\!\bigl(\Omega^1_X(\log S), \mathcal{O}_X\bigr) $
indicates the possible presence of obstructions to the deformations of \(X\) and \((X,S)\). In particular, we cannot prove that the stacks $\cP_n$ for $n\geq4$ are smooth using the cohomology groups.

\section{KSBA-compactification via twisted stable maps}\label{section:tsmcompactification}

In this section we discuss the natural compactifications of $\cE_n$ (resp. $\Psell$) given by taking the closure in the proper moduli stack $\cM^{\KSBA}_{c(\epsilon),v(\epsilon)}$ (resp. $\cM^{\KSB}_v$). By the main theorem, these compactifications are irreducible components of the KSBA (resp. KSB) moduli spaces for $n \neq 4$. 

We use two tools to understand these compactifications: twisted stable maps and wall-crossing as developed in \cites{av,AV_compactifying, tsm, tsm2} and \cites{AB, inchiostro2020moduli, ascher2021wall, meng2023mmp} respectively. 

The starting point is the observation that an elliptic surface $(X \to C, S)$ with at worst $I_k$ fibers is equivalent to a morphism 
$$
g \colon C \to \overline{\cM}_{1,1}
$$
to the moduli stack of pointed elliptic curves. The composition to the coarse moduli space $j \colon C \to \overline{M}_{1,1}$ is the $j$-invariant of the elliptic surface. The $I_k$ fibers lie over the preimages of $\infty \in \overline{\cM}_{1,1}$ and $k$ is the ramification of $g$ at a given preimage \cite[page 41, Table 4.3.1]{Miranda}. Here we have used that the coarse moduli space $\overline{\cM}_{1,1}\to \overline{M}_{1,1}$ is unramified at $\infty$ so the ramification of $g$ and the $j$-map agree.
The degree of the $j$-map satisfies $\deg(j) = 12n$, where $n$ is the height. The space of maps admits a compactification by a proper Deligne-Mumford stack
$$
\cK_n := \overline{\cK}_{0,0}(\overline{\cM}_{1,1}, n)
$$
parametrizing $0$-pointed, genus $0$ and degree $n$ twisted stable maps (cf. \cite{AV_compactifying}). 

\begin{Def}\label{def_tsm_stability}
    A \emph{$0$-pointed twisted stable map of genus $g$ and degree $n$} is a commutative diagram 
    $$
    \xymatrix{\cC \ar[r]^{j'} \ar[d]_\pi & \overline{\cM}_{1,1} \ar[d] \\ C \ar[r]_j & \overline{M}_{1,1}}
    $$
    such that
    \begin{enumerate}
      \item $\cC$ is a stacky curve with at worst nodal singularities, and $\pi \colon \cC \to C$ is the coarse moduli space; 
      \item $\pi$ is an isomorphism over the non-singular locus of $C$;
      \item $\cC$ is formally locally isomorphic around each node to 
      $$
      \big[\big(\spec k[\![x,y]\!]/(xy)\big)/\bmu_r\big] \quad \textup{by} \quad (x,y) \mapsto (\xi \cdot x, \xi^{-1} \cdot y);
      $$
      \item $j'$ is a representable morphism which induces the map $j$; and 
      \item $j$ is a stable map of genus $g$ and degree $12n$.     
    \end{enumerate}
Since the bottom half of the diagram is determined by the top half, we often just write $\big(j' \colon \cC \to \overline{\cM}_{1,1}\big)$. 
\end{Def}

\begin{Remark}
    Observe that, from (3) in \Cref{def_tsm_stability}, the twisted curve $\cC$ has no nodes if and only if $C$ has no nodes.
\end{Remark}

We start proving part (\ref{item: main: birationality of Kn and En}) of \cref{thm:main} in the following.

\begin{Teo} \label{thm: irreducibility of K_n}
    The stack $\cK_n$ is irreducible and proper. 
\end{Teo}

\begin{proof}
    Properness is proved in \cite[Theorem 1.4.1(1)]{AV_compactifying}, so we focus on proving irreducibility.

    By \cite{tsm2}*{Thm. 5.6}, any genus $0$ twisted stable map $g_0 \colon \cC_0 \to \overline{\cM}_{1,1}$ can be deformed to a family of maps
    $$
    \xymatrix{\cC \ar[r]^g \ar[d] & \overline{\cM}_{1,1} \\ \spec R & }
    $$
    over the spectrum of a DVR $R$ with closed point $0$ and generic point $\eta \in \spec R$ such that $\cC_\eta$ is a smooth genus $0$ curve. Thus, the locus of maps with smooth source is dense in $\cK_n$. On the other hand, by definition, if $\cC$ is smooth, then $\cC \to C$ is an isomorphism and so $\cC \cong \bP^1$. We conclude that the space of twisted stable maps to $\overline{\cM}_{1,1}$ of degree $n$ and genus $0$ with \emph{smooth} source curve is simply the space $\cW^{\nod}_n$ of Weierstrass elliptic fibrations of height $n$ with at worst nodal fibers, which is irreducible by Theorem \ref{theorem: construction of Wn}(2). 
\end{proof}

We recall the notation from Remark \ref{rem:walls}. The rest of part (\ref{item: main: birationality of Kn and En}) of \cref{thm:main} follows from \cite{tsm} and the general wall-crossing formalism of \cite{AB, inchiostro2020moduli, ascher2021wall, meng2023mmp} that we cite below. 

\begin{Prop}
\label{proposition: maps between compactifications}
    There exists a commutative diagram of wall-crossing morphisms 

    \[
        \xymatrix{\cK_n^{\nu} \ar[dr]^{\bar{\Phi}_{n,0}} \ar[d]_{\bar{\Phi}_n} & & \\ \overline{\cE}_{n} &  \overline{\cE}_{n,0} \ar[r]^{\bar{\Psi}_n} \ar[l]_{\rho} & \overline{\cP}_n}
    \]
    
    where $\cK_n^{\nu}$ is the normalization of $\cK_n$ and $\bar{\Phi}_n$ and $\bar{\Phi}_{n,0}$ (resp. $\bar{\Psi}_n$) extend $\Phi_n|_{\cW_n^{\nod}}$ (resp. $\Psi_n$). 
\end{Prop}

\subsection{From twisted stable maps limits to KSBA limits}\label{subsection_from_tsm_to_KSBA}

There are two stability conditions that we will consider for elliptic surfaces.
The first one comes from \textit{twisted stable maps}, as in \Cref{def_tsm_stability}, and the second one from KSB(A), as in \Cref{subsection_KSB_stab}. 
\begin{Def}\label{def_tsm}
    An elliptic surface $(X,S)\to C$ is \textit{twisted stable maps-stable} (abbv. tsm-stable) if there is 0-pointed twisted stable map $(j' \colon \cC\to \overline{\cM}_{1,1})$ of a certain genus such that:
    \begin{enumerate}
        \item $C$ is the coarse moduli space of $\cC$;
        \item if we denote by $(\sX,\sS)\to \cC$ the pull-back of the universal curve and the universal section to $\cC$, then $(X,S)$ is the coarse moduli space of $(\sX,\sS)$; and
        \item the projection morphism $X\to C$ coincides with the induced morphism between coarse moduli spaces of $\sX\to \cC$.
    \end{enumerate}
\end{Def}

\begin{Remark}
    One can check that the locus where $\cC$ is a stack (but not a scheme) can be determined from the data of $X\to C$. Indeed, $\cC$ is a stack along the nodes $\mathfrak{n}\in C$ such that the fiber of $X\to C$ is non-reduced. Indeed, the fibers of $\cX\to \cC$ are parametrized by $\overline{\cM}_{1,1}$, so if we take the reduced structure on the geometric fiber of $X\to C$ over $\mathfrak{n}$, the resulting curve is a quotient of a Deligne-Mumford stable 1-pointed genus 1 curve $(E,x)$, by a subgroup $\Gamma$ of $\Aut(E,x)$, and $\Gamma = \{1\}$ if and only if $\cC$ is a scheme at the node $\mathfrak{n}$; see \cite{tsm}.
\end{Remark}
It follows from the wall-crossing results of \cite{ascher2021wall,meng2023mmp} that there is a morphism $\rho\colon \overline{\cE}_{n, 0}\to \overline{\cE}_{n}$, which reduces the weight of the divisor from $\frac{n-2}{n}-\epsilon$ (in $\overline{\cE}_{n,0}$) to $\epsilon$ (in $\overline{\cE}_{n}$) where $0 < \epsilon \ll 1$. By Proposition \ref{proposition: maps between compactifications}, have $\rho\circ \bar{\Phi}_{n,0} = \bar{\Phi}_n$. It turns out that it is slightly easier to understand $\bar{\Phi}_n \colon \cK_n^\nu\to \overline{\cE}_{n}$ rather than $\bar{\Phi}_{n,0}$, so we will focus on $\bar{\Phi}_n$ in this section. We will later prove that $\rho$ is an isomorphism, so that we have \emph{a posteriori} also described $\bar{\Phi}_{n,0}$.

We now describe how to understand $\bar{\Phi}_n$. For each point $p\in \cK_n^{\nu}$, consider a one parameter family $\spec(R)\to \cK_n^{\nu}$ over the spectrum of a DVR $R$, which maps the generic point $\eta$ to the locus parametrizing minimal Weierstrass fibrations in $\cK_n$, and the special point, which we denote by $0$, to $p$. One has the resulting family of tsm-stable elliptic surfaces:
\[ (\sX,\sS)\ \longrightarrow\ \sC\ \longrightarrow\ \spec(R).\] 
\begin{Remark}
    In this subsection we will denote by $\sC$ a family of nodal curves over a DVR as above, whereas by $\cC$ a twisted curve over the spectrum of a field as in \Cref{def_tsm}.\end{Remark}
Choose $0<\epsilon\ll 1$, and consider the canonical model  \[ (\sX^c,\epsilon\sS^c)\ \longrightarrow\ \spec(R)\ \ \ \text{ of } \ \ \ (\sX,\epsilon\sS) \ \longrightarrow\  \spec(R).\] From \cite{inchiostro2020moduli}*{Thm. 1.2}, for $\epsilon$ small enough, the special fiber of such a canonical model will be an elliptic surface $\sX^c_0\to \sC^c_0$ with a section $\sS^c_0$ and irreducible fibers. This stable pair $(\sX_0^c, \sC_0^c)$ yields the point $\bar{\Phi}_n(p)$. 

 To obtain the canonical model $(\sX^c,\epsilon \sS^c)$, one has to run an MMP and use the abundance theorem; this is worked out in \cites{AB, inchiostro2020moduli}, we report the salient steps.

In \textit{loc. cit.} it is proven that there is a \textit{specific} MMP with scaling that one can run, such that only a specific type of flip is needed, the so-called \textit{flip of La Nave} (see \cite{AB}*{Appendix A}; see also \cite{lanave2002explicit}*{Thm. 7.1.2}, \cite{AB}*{Sec. 6.2.2} or \cite{inchiostro2020moduli}*{3.2} for a description of the flip of La Nave). More specifically, it is proven in \cite{inchiostro2020moduli}*{Thm. 6.5} that the specific MMP mentioned above can be factored as 
$$(\sX,\sS)\ =\ (\sX^{(1)},\epsilon\sS^{(1)})\dashrightarrow(\sX^{(2)},\epsilon\sS^{(2)}) \dashrightarrow ... \dashrightarrow (\sX^{(m)},\epsilon\sS^{(m)})\ =\ (\sX^c,\epsilon\sS^c)$$
where $f_i: \sX^{(i)}\dashrightarrow \sX^{(i+1)}$ is either a flip of La Nave, or a divisorial contraction of some irreducible component of $\sX^{(i)}_p$. Via these steps, it is proven in \cite{AB}*{Section 6} and \cite{inchiostro2020moduli}*{Cor. 6.7} that the central fiber of $\sX^{(i)}$ is a nodal union of irreducible components which are either:
\begin{enumerate}
    \item pseudo-elliptic surfaces, or
    \item elliptic surfaces.
\end{enumerate}
 The flip of La Nave contracts an irreducible component of the special fiber of the section $\sS^{(i)}$; this will result in a pseudo-elliptic component attached to a so-called \textit{intermediate fiber} (cf. \cite{ascher2017log}*{Def. 4.9}).

 Now, to control how the special fiber of $(\sX^{(i)},\sS^{(i)})$ changes after each step of the MMP, we have to control the intersection pairing on some elliptic surfaces; this is the goal of the remaining part of this subsection.

  \begin{Lemma}\label{lemma_which_intermediate_fiber_after_a_flip}
     Assume that $\sX^{(i)}\dashrightarrow \sX^{(i+1)}$ is a flip of La Nave, which contracts a section $S_P\subseteq \sX^{(i)}$ and extracts the curve $A\subseteq \sX^{(i+1)}$. Let $S^+$ be the section of the elliptic surface given by the irreducible component $E^+\subseteq \sX^{(i+1)}_0$ containing $A$, and let $S^-$ be the section in the elliptic surface $E^-$, given by the proper transform of $E^+$ in $\sX^{(i)}_{0}$.
     Then one has
     \[
     (S^-)^2 \ =\ (S^+)^2 -\frac{1}{A^2}\ \ \text{ and } \ \ S_P^2\ =\ -\frac{1}{A^2}. 
     \]
 \end{Lemma} 
 \begin{figure}[H]
\centering
\includegraphics[width=13cm]{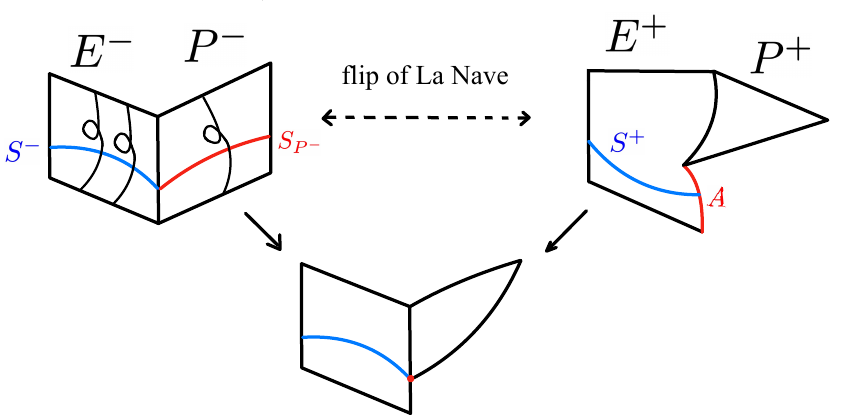}
\caption{Special fibers $\sX^{(i)}_{0}\dashrightarrow \sX^{(i+1)}_{0}$ of a flip of La Nave}
\end{figure}
  \begin{proof}
     Let $\sS^{(j)}$ be the proper transform of the section on $\sX^{(j)}$ for every $j$. Then, if $S_1,...,S_k$ are the irreducible components of $\sS^{(j)}_{0}$, we have that \[\big(\sS^{(j)}_\eta\big)^2  \ =\ \sum_{\ell=1}^k \big(\sS^{(j)}.S_\ell\big) \ =\ \sum_{\ell =1}^k \big(S_\ell^2\big)_{\sX^{(j)}_l}, \]
     where $\sS^{(j)}_\eta$ is the restriction of $\sS^{(j)}$ to the generic fiber of $\sX^{(j)}\to \spec(R)$, and the last intersection pairing is computed on each irreducible component $\sX^{(j)}_l$ of $\sX^{(j)}$. The first equality follows from the flatness of $\sS^{(j)}\to \spec(R)$, and so its self-intersection is constant along the fibers of $\sX^{(i)}\to \spec(R)$. 

     Now, from how the flip of La Nave is constructed, there are exactly two irreducible components of $\sX^{(i)}$ where $f_i$ is not an isomorphism. One is $E^-$, the proper transform of $E^+$, and the other is $P^-$, the irreducible component whose section $S_P$ will be contracted after the flip. Denote by $S^{+}$ (resp. $S^{-}$) the section on $E^{+}$ (resp. $E^{-}$). Then we have that 
     \[ (S^-)^2 + (S_P)^2 \ =\ (S^+)^2.\]
     Moreover, there is a morphism $p: E^+\to E^-$ which contracts $A$, and $p_*S^+ = S^-$. So
     \[(S^-)^2 \ = \ (S^-.\ p_*S^+)\ =\ (p^*S^-.\ S^+) \ =\ \big(S^++\alpha A.\ S^+\big). \]
     Here, we can compute \[\big(S^++\alpha A.\ A\big)=0 \ \ \ \textup{so}\ \ \ \alpha = \frac{-1}{A^2}.\]
     Therefore one has $
     \big(S^-\big)^2 = \big(S^+\big)^2 -\frac{1}{A^2}
     $ as desired.
 \end{proof}
It follows from \cite{inchiostro2020moduli}*{Thm. 9.9} that when $0<\epsilon\ll 1$, there are no pseudo-elliptic surfaces among the irreducible components of $\sX^c$. Therefore, all the pseudo-elliptic surfaces which appear after a flip of La Nave will eventually get contracted, and the corresponding intermediate fiber along which the pseudo-elliptic component is attached will become a cuspidal fiber (cf. \cite{ascher2017log}). Moreover, from \cite{inchiostro2020moduli}*{Prop. 4.14}, the type of intermediate fiber that appears after a flip of La Nave is uniquely determined by $A^2$. Combining this with \cite{ascher2017log}*{Thm. 1.1} and \cite{AB}*{Table 2}, where the authors compute the log-canonical threshold of a cusp in a minimal Weierstrass fibration, we obtain the following result.

\begin{Lemma}\label{lemma_lct_Asquare}
    Suppose that $f_i:\sX^{(i)}\to \sX^{(i+1)}$ contracts a pseudo-elliptic surface. Let $A\subseteq \sX^{(i)}$ be the intermediate component of the intermediate fiber along which is attached the pseudo-elliptic surface contracted by $f_i$. Let $A^c$ be the cuspidal fiber given by the proper transform of $A$ in $\sX^{(i+1)}$, and $X$ be the irreducible component of $\sX^{(i+1)}_{0}$ containing $A^c$. Then we have that $$-\frac{1}{A^2} \ =\ 1-\lct(X; A^c).$$
\end{Lemma}
\begin{Remark}
    We justify why, in the previous setting, $A^c$ is a cuspidal fiber. Note first that this is proved in \cite{AB}, see especially \cite[\S 6.2.3]{AB} and \cite[Theorem 6.3 (c)]{AB}; however we include a more direct argument.

    We observe that each step of our MMP for $(\sX,\epsilon \sS)$ maps to a family of nodal curves, i.e.
    we have maps $\sX^{(j)}\to \sC^{(j)}$ for every $j$. This follows from how the MMP is constructed in \cite{AB,inchiostro2020moduli}. Observe also that $A^c$ is irreducible. This follows as it is the flipped curve of a flip of La Nave; so in particular $\sS^{(i+1)}\subseteq \sX^{(i+1)}$ is ample over $\sC^{(i+1)}$, as in our setting we contracted the pseudoelliptic components.
    
    Now, the image of the pseudoelliptic component which is contracted via $\sX^{(i)}\dashrightarrow \sX^{(i+1)}$ lies over a \textit{smooth}
    point $p$ of the closed fiber of $\sC^{(i+1)}$. There is an open $p\in U\subseteq \sC^{(i+1)}$ such that the fibers of \[(\sX^{(i+1)}|_{U\smallsetminus \{p\}},\sS^{(i+1)}|_{U\smallsetminus \{p\}})\to {U\smallsetminus \{p\}}\] away from $p$ are parametrized by $\overline{\cM}_{1,1}$, i.e. come from a map $U\smallsetminus \{p\}\to \overline{\cM}_{1,1}$. Recall that $\overline{\cM}_{1,1}$
    is a weighted projective stack, so there is an inclusion $\overline{\cM}_{1,1}\subseteq [\bA^2_{A,B}/\bG_m]$ where the action has weight 4 and 6. More explicitly, the stack $[\bA^2_{A,B}/\bG_m]$ parametrizes cubics of the form $y^2z=x^3+Axz^2+Bz^3$ with a marked smooth point,
    which are either smooth, nodal, or cuspidal, and the good moduli space for $[\bA^2_{A,B}/\bG_m]$ is the
    ring of invariants of $k[A,B]$ with the action of weight 4 on $A$ and 6 on $B$ (namely, it is $\spec(k)$). Now it follows from \cite[Lem. 2.1]{di2024stable}
    that the composition
    \[U\smallsetminus \{p\}\to \overline{\cM}_{1,1}\to [\bA^2/\bG_m]
    \]extends to $U\to [\bA^2/\bG_m]$, namely, the family of cubics with a marked point over $U\smallsetminus \{p\}$ extends to a family of cubics with a marked point over $U$. More specifically, the content of \cite[Lem. 2.1]{di2024stable} in our case specializes to the fact that the family of cubics over $U\smallsetminus\{p\}$ is the data of its Weierstrass form, which is the data of a line bundle $L$ and two sections, one of $L^{\otimes 4}$ and one of $L^{\otimes 6}$. As $p$ is smooth on $U$ which has dimension 2, the line bundle $L$ with the two sections $A$ and $B$ extend to $U$, so we can extend the Weierstrass form.

    Now note that the family of pointed cubics we constructed extending $L$ and the two sections as above, which we denote by $(\sX',\sS')\to U$, agrees with $(\sX^{(i+1)}|_U,\sS^{(i+1)}|_U)\to U$ as both $\sS'$ and $\sS^{(i+1)}|_U$ are ample over $U$, both $\sX'$ and $\sX^{(i+1)}$ are $S_2$ so $\sX^{(i+1)}|_U$ and $\sX'$ are Proj of the same algebra.
\end{Remark}
\begin{Prop}\label{prop_can_bundle_formula}
    Let $(X',S')\to C'$ be a tsm-stable elliptic surface, and let $(X,S)$ be an irreducible component of $X'$. Let
    \begin{enumerate}
    \item $C$ be the irreducible component of $C'$ such that $X$ maps to $C$;
    \item $f\colon X\to C$ be the corresponding map;
    \item $p_1,\dots,p_k$ be the nodes of $C'$ along $C$;
    \item $F_i:=(f^{*}p_i)^{\red}$, and $m_i$ be the multiplicity of $F_i$ in $f^{*}p_i$.
    \end{enumerate} If $j\colon C\to \overline{M}_{1,1}$ is the corresponding $j$-map, then one has \[
    (S^2)\ = \ -\frac{\deg(j)}{12} \ \ \text{ and }\ \ \ \bigg(K_X+\sum_{i=1}^kF_i\bigg). \ S \ =\ 2g(C)-2+ k +\frac{\deg(j)}{12}.
    \]
    Moreover, if we denote by $Y$ the surface obtained by replacing $F_1,\dots,F_n$ by cuspidal fibers $F^c_1,\dots ,F^c_n$, and denote by $S_Y$ the proper transform of $S$ on $Y$, then one has
    \[
    \bigg(K_Y+\sum_{i=n+1}^kF_i\bigg).\ S_Y\ = \ \bigg(K_X+\sum_{i=1}^kF_i\bigg).\ S -\sum_{i=1}^n\lct(Y; F^c_i)
    \] and
    \[
    \big(S^2_Y\big) \ =\  -\frac{\deg(j)}{12} - \sum_{i=1}^n\big(1 - \lct(Y;F_i^c)\big).
    \]
    
\end{Prop}
\noindent{}We now explain the process of replacing $F_1,\dots,F_n$ with cuspidal fibers to obtain \(Y\).
One can take a resolution of $X$ around the intersection point of $S$ and $F_i$.
The resulting surface will admit a morphism to $C$ and if one contracts all the irreducible fiber components not meeting $S$,
the fiber $F$ gets replaced by a singular fiber,
as shown in \cite{ascher2017log}. It is shown in \textit{loc. cit.} that the resulting fibers are cusps in our case.

\begin{Remark}
    \Cref{prop_can_bundle_formula} is the reason it is simpler to work with $\bar{\Phi}_n$ and the limit in $\overline{\cE}_{n}$ rather than with the moduli space $\overline{\cE}_{n,0}$: we can use the canonical bundle formula on $(\sX^{(m)}, \sS^{(m)}) = (\sX^c, \epsilon \sS^c)$, as $\sX^{(m)}$ admits a fibration with irreducible fibers which generically are elliptic curves.
\end{Remark}

\begin{Remark}\label{rmk_S_Y_square_is_an_integer}
    Observe that in the formula computing $(S_Y^2)$, if all the fibers of $Y$ are reduced (e.g. when $n=k$), then $(S_Y^2)\in \bZ$, since in this case the section of $Y$ is contained in the smooth locus of $Y$.
\end{Remark}
\begin{proof}[Proof of \Cref{prop_can_bundle_formula}]
    Recall that the fibers of $f:X\to C$ are either one-pointed genus one stable curves, or the twisted fibers of \Cref{subsection_twisted_fiber}. In particular, they are irreducible, so from the canonical bundle formula \cite{floris2020travel}*{Ex. 2.7} one has
    \[K_X = f^*\left(K_{C}+\frac{\deg(j)}{12}\cO(1) + \sum_{i=1}^k(1-m_i')p_i\right)\]
    where $\{p_1,...,p_k\}$ are points in $C$ where the fibers of $X\to C$ could be worse than nodal singularities (which by assumption are supported on the nodes of $C'$).
    From canonical bundle formula $m_i'=\lct(X; m_iF_i)$. In \cite{ascher2017log}, the authors compute it explicitly, and it is shown that $m_i'=\frac{1}{m_i}$, or in other terms, $(X,F_i)$ is log-canonical.

    Similarly, it is proven in \cite{tsm}*{Prop. 5.3} that \[
    \bigg(K_X+S+\sum_{i=1}^k F_i\bigg). \ S\ =\ 2g(S)-2+k.
    \]
    Since \[
    F_i.S = \frac{1}{m_i}f^{-1}(p_i).S = \frac{1}{m_i},
    \]
    putting this together leads to \[
   2g(S)-2+\frac{\deg(j)}{12} + \sum_{i=1}^k(1-m_i) + (S^2) + \sum_{i=1}^km_i\ =\ 2g(S)-2+k \] and thus \[(S^2)\ =\  - \frac{\deg(j)}{12}.\]
    
    The other equality now follows similarly, and the ``moreover" part follows again from the canonical bundle formula.
\end{proof}


\begin{Algorithm}[explicit construction of $(\sX^c,\epsilon\sS^c)$]
\label{algorith: stable reduction}
We are ready to explain how to compute the canonical model $(\sX^c,\epsilon \sS^c)$, starting from $(\sX,\sS)$. From the explicit descriptions of the steps of the MMP, one can proceed as follows. Let $S$ be an irreducible component of $\sS^{(i)}$ which is a leaf\footnote{Here, a \emph{leaf} is a vertex of degree one in the dual graph, i.e. it is incident to exactly one edge.} on the dual graph of $\sS^{(i)}$, and which belongs to the irreducible component $X\subseteq \sX^{(i)}_{0}$. Then $X$ is glued to the other components of $\sX_0^{(i)}$ by a single fiber $F$ and $S^2 < 0$ by \Cref{prop_can_bundle_formula}. We can use \Cref{prop_can_bundle_formula} to compute $\big(K_X+F\big).S$. There are two cases:
\begin{enumerate}
    \item  If $(K_X+F).S\le0$, then $(K_{\sX^{(i)}} + \epsilon \sS).S = (K_X + \epsilon S + F).S < 0$ and we flip $S$ via a flip of La Nave. Since $S$ is a leaf, there is a unique irreducible component $X'\subseteq \sX^{(i)}_{0}$ to which $X$ is attached along the fiber $F$. The flip of La Nave will replace $F$ with an intermediate fiber. We know that the resulting pseudo-elliptic component $P$ (namely, the proper transform of $X$) has to be contracted by taking the canonical model $(\sX^c,\sS^c)$, either by a step of our special MMP (if there is a $(K_{\sX^{(i)}}+\epsilon \sS^{(i)})$-extremal ray, whose contraction will contract $P$) or by taking the canonical model of our minimal model (i.e. by $f^{(m-1)}$). We choose to contract it right away, so that the resulting contraction will replace $F$ with a cusp, whose log-canonical threshold can be computed using \Cref{lemma_which_intermediate_fiber_after_a_flip}, and so that each irreducible component of the resulting special fiber will be an elliptic surface with a section and with irreducible fibers (which will allow us to use \Cref{prop_can_bundle_formula} once again). One can perform the contraction of a pseudoelliptic component explicitly by adding several marked fibers of $\sX^{(i)}\to \sC^{(i)}$, away from $P$, and taking the canonical model of the resulting threefold pair.
    \item  If $(K_X+F).S> 0$ we do nothing and move to the next leaf.
\end{enumerate}
The algorithm terminates precisely when we are in case $(2)$ for every leaf, and in this case the resulting central fiber is stable (here we have used that $0 < \epsilon \ll 1$ is very small). Since the KSBA moduli space is separated, this must be the unique KSBA-stable limit.
\end{Algorithm}

The previous propositions explain how the intersection pairing change after point (1) above, so we can iterate the step above until for each $S$ we have $(K_X+\sum F_i).S> 0$. At this point, the resulting elliptic surface will be the special fiber of $(\sX^c,\epsilon \sS^c)$.
\begin{Remark}
    One might wonder how the algorithm would change if we were to consider other coefficients for the section $\sS$. It turns out that, for that case, the control on the steps of our special MMP will provide other (slightly more complicated) combinatorial invariants that one has to consider; see \cite{inchiostro2020moduli}*{Thm. 1.4 and Def. 7.3}.
\end{Remark}

\subsection{Combinatorial description}\label{subsection_comb}
The goal of this subsection is to introduce a combinatorial data which will package all the information in \Cref{algorith: stable reduction}. We will use this combinatorial description as a gadget to control the steps of \Cref{algorith: stable reduction}, which are needed to study the boundary of $\overline{\cE}_n$ and $\overline{\cP}_n$; see \Cref{thm:stratificationE_n} and \Cref{thm:boundary_P_n}. 
We will first introduce objects that correspond to the locally closed strata of $\cK_n$ -- the compactification of the Weierstrass locus $\Wei^{\nod}$ by twisted stable maps.

\begin{Def}
\label{definition: sliced tree}
    A \emph{sliced tree} $\Gamma= (V,E,E_0,\jdeg)$ is a graph with vertices $V$ and edges $E$, which is a tree, together with the following structure.
    We choose a $\jdeg$ function $$\jdeg \colon V \longrightarrow \frac{1}{12}\bZ_{\geq0},$$ a subset $E_0 \subset E$ of \emph{sliced edges}, and for each $e \in E_0$ connecting $v$ and $w$, we assign a pair $(e_v, e_w)$ of fractions, called \emph{slicings}, from the following list:
    \[
    \left( \frac{1}{2}, \frac{1}{2} \right), \ \ \ 
    \left( \frac{1}{3}, \frac{2}{3} \right), \ \ \
    \left( \frac{1}{4}, \frac{3}{4} \right), \ \ \
    \left( \frac{1}{6}, \frac{5}{6} \right).
    \]
    
\noindent As usual, a vertex $v \in V$ is called a \emph{leaf} if it is adjacent to only one other vertex.
We define
$$\jdeg(\Gamma)\ :=\ \sum_{v\in V}\jdeg(v).$$
Moreover, we define $E(v)$ (resp. $E_0(v)$) to be the set of edges (resp. sliced edges) adjacent to $v \in V$, and
we require that for each vertex $v$ we have that $$ \jdeg(v)+\sum_{e \in E_0(v)} e_v \ \in\ \bZ.$$
\end{Def}

\begin{Def}
    We say that a sliced tree $\Gamma = (V,E,E_0,\jdeg)$ is \emph{tsm-stable} if
    \begin{enumerate}
        \item $\jdeg(v) \geq 0$ for all $v \in V$; and
        \item $|E(v)| \geq 3$ if $\jdeg(v) = 0$.
    \end{enumerate}
    
\end{Def}

\begin{Def}\label{def_assignment}
    Let $(X,S)\to C$ be a tsm-stable elliptic surface. The \emph{sliced tree $\Gamma$ associated to $(X,S)\to C$} is given as follows.
    \begin{enumerate}
        \item Each irreducible component of $X$ corresponds to a vertex of $\Gamma$.
        \item The $\jdeg$ of each vertex is the degree of the corresponding $j$-map divided by $12$.
        \item There is an edge between vertices $v$ and $w$ if the corresponding irreducible components intersect.
        \item
        \label{item: edge is sliced iff glued along non-reduced fiber}
        The edge connecting $v$ and $w$ is sliced if the corresponding irreducible components are glued along a non-reduced fiber.
        \item
        \label{item: associated slicing}
        The slicing is defined as follows.
        If two irreducible components $X_v$ and $X_w$ of $X$ intersect along a non-reduced fiber $F$, let $Y_v \to X_v$ (resp. $Y_w\to X_w$)
        be the minimal resolution of $X_v$ (resp. $X_w$) along $S\cap F$, and let $Y_v\to Z_v$
        (resp. $Y_w \to Z_w$) be the surface obtained by contracting all the fiber components not meeting $S$.
        Then the fiber $F$ is replaced with a different fiber
        $F_v \subseteq Z_v$ (resp. $F_w \subseteq Z_w$).
        The slicing adjacent to $v$ (resp. $w$) is $1-\lct(Z_v; F_v)$ (resp. $1-\lct(Z_w;F_w)$).
    \end{enumerate}

\end{Def}
~\
\begin{Remark}~\

\begin{enumerate}[(a)]
    \item  The surfaces $Z_v$ and $Z_w$ of \Cref{def_assignment} (\ref{item: associated slicing}) can also be constructed by taking the minimal Weierstrass fibration birational to $X_v$ and $X_w$ respectively, if $X_v$ and $X_w$ are normal.
    The fibers $F_v$ and $F_w$ would be the cuspidal fibers which replace $F$.
    \item The condition $\jdeg(v)+\sum_{e \in E_0(v)} e_v \in \bZ$ follows from \Cref{rmk_S_Y_square_is_an_integer}.
    \item One might wonder why the possible markings are only those listed in \Cref{definition: sliced tree}, for example, why $\left(\frac{1}{2},\frac{1}{6}\right)$ is not allowed. This follows from \Cref{def_assignment} (\ref{item: edge is sliced iff glued along non-reduced fiber}).
    Indeed, it turns out that if the action is of the type $(x,y)\mapsto (\xi\cdot x, \xi^{-1}\cdot y)$, then the only slicing allowed are those of \Cref{definition: sliced tree}. We now explain this point more carefully.

    Consider a family of twisted curves $\cC\to \spec(R)$ over the spectrum of a DVR with uniformizer $\pi$, with smooth generic fiber. Assume we are given a map $\gamma\colon\cC\to \overline{\cM}_{1,1}$, and let $n_\cC\in \cC$ be a node of the central fiber of $\cC\to \spec(R)$. It follows from \cite[Proposition 2.2 (ii)]{olsson2007log} the local equation of the node is of the form
    \[
    [\spec(R[\![x,y]\!]/xy-\pi^k)/\bmu_n]
    \]
    with the action of $\bmu_n$ of the form $\xi*x=\xi x$ and $\xi*y=\xi^{-1} y$. Then, let $(\cX,\cS)\to \cC$ be the family of genus 1 curves with a section corresponding to $\gamma$, and let $n_\cS$ the fiber of $\cS\to \cC$ over $n_\cC$, the nodal point of $\cC$. As $\gamma$ is representable, the action of $\bmu_n$ is faithful on the vertical tangent direction at $n_\cS$. In particular, the local equation of $\cX$ at $n_\cS$ is of the form
    \[
    [\spec(R[\![x,y,f]\!]/xy-\pi^k)/\bmu_n]
    \]
    with the action of $\bmu_n$ of the form $\xi*x=\xi x$, $\xi*f=\xi^\delta f$ for $\delta$ and $n$ coprime, and $\xi*y=\xi^{-1} y$. Here, we denoted by $f$ the vertical tangent direction. Now, over the central fiber $\cX_0\subseteq \cX$ (i.e. when $\pi=0$) the equation of $\cX_0$ is \[
    [\spec(k[\![x,y,f]\!]/xy)/\bmu_n]
    \] with the same action as before. So for example when $n=3$ and $\delta=1$, we see that the two components of the central fiber, once we take the coarse moduli space, are:
    \begin{enumerate}
        \item the quotient of $k[\![x,f]\!]$ by the action of $\bmu_3$ which sends $\xi*x = \xi x$ and $\xi*f=\xi f$, and
        \item the quotient of $k[\![y,f]\!]$ by the action of $\bmu_3$ which sends $\xi*y = \xi^{-1}y$ and $\xi*f=\xi f$.
    \end{enumerate}
    Namely, an $A_2$-singularity and a singularity which is the cone over a rational normal curve in $\bP^3$.

    As the map $\cC\to \overline{\cM}_{1,1}$ is representable, we can only have $\bmu_n=\{1,\bmu_2, \bmu_4,\bmu_3,\bmu_6\}$. A tedious but straightforward case by case analysis leads to the previous table.
\end{enumerate}
\end{Remark}

The following proposition is now straightforward.
\begin{Prop}
    The stack $\cK_n$ admits a locally closed stratification with  strata labeled by the stable sliced trees of $\jdeg(\Gamma)=n$:
    \[
        \cK_n = \bigsqcup_{\Gamma, \, \jdeg(\Gamma)=n} \cK_\Gamma.
    \]
\end{Prop}

\begin{Remark}
    \Cref{table: sliced weights} explains what is the slicing (and so $1-$ the minimal log-canonical threshold) of each type of cusp which appears in a minimal Weierstrass fibration.
    Moreover, if one replaces a cuspidal fiber $f^{-1}(p)$ with a twisted fiber as in \cite{ascher2017log}, the corresponding elliptic surface in a neighborhood of $p$ comes from a twisted stable map $\phi \colon \cC\to \overline{\mathcal{M}}_{1,1}$.
    From the analysis in \cite{ascher2017log} and \cite{bejleri2022height}*{Thm. 1.6 \& Sec. 7}, the denominator $d$ of each slicing corresponds to the order of the stabilizers of $p\in \cC$, and the numerator comes from the character of the irreducible representation of $\bmu_d$ on $\phi^*\cO(1)$. For example, type $\operatorname{III}^*$ cusps get replaced with twisted fibers such that $\cC$ has automorphism group $\bmu_4$, with the action on the fibers of $\phi^*\cO(1)$ being $\zeta\cdot v = \zeta^3v$.
\end{Remark}
 
    \begin{table}
    \caption{Correspondence between Kodaira fiber types and slicings}
    \label{table: sliced weights}
    \renewcommand{\arraystretch}{1.3}
    \begin{tabular}{|c|c|l|c|c|}
      \cline{1-2} \cline{4-5}
    Kodaira type & Slicing, or $1- \lct$ &  & Kodaira type  & Slicing, or $1- \lct$ \\ 
    \cline{1-2} \cline{4-5}
    I       & $0$ &  & 
    I*          &$\frac{1}{2}$ \\ 
     \cline{1-2} \cline{4-5}
    II          &$\frac{1}{6}$ &  & 
    II*     &   $\frac{5}{6}$ \\ 
     \cline{1-2} \cline{4-5}
    III      &$\frac{1}{4}$ &  & 
    III*       &$\frac{3}{4}$ \\ 
     \cline{1-2} \cline{4-5}
    IV         &$\frac{1}{3}$  &  & 
    IV*        &$\frac{2}{3}$ \\ 
      \cline{1-2} \cline{4-5}
    \end{tabular}
    \end{table}

\begin{Remark}\label{rem:empty}
      Some of these strata may be empty. For example, one can consider a tree with a vertex 
     whose $\jdeg$ equal to $\frac{1}{6}$, and with 5 edges exiting from it, each sliced with coefficient $\frac{1}{6}$.
     This should correspond to an elliptic surface, with $j$-map of degree 2, and with five non-reduced fibers. One can replace this elliptic surface with its minimal Weierstrass fibration, and one would have a Weierstrass fibration with five type $\mathrm{II}$ cusps and $j$-map of degree $2$.
     But such a Weierstrass fibration would be given by $$A\in \oH^0(\cO_{\bP^1}(4)) \ \ \ \textup{and} \ \ \ B\in \oH^0(\cO_{\bP^1}(6)),$$ where the polynomial $A$ vanishes along the five points corresponding to cuspidal fibers. This forces $A$ to be identically 0. By \cite{bejleri2022height}*{Thm. 1.6}, if the edges are sliced with coefficient $\frac{1}{6}$, the corresponding cusp will be such that $B$ cannot have a double root. So $B$ vanishes at six distinct points, and therefore there are six cusps rather than five, which is the desired contradiction.
\end{Remark}

\begin{EG}
\label{example: explicit sliced tree}
    Consider a tsm-stable elliptic surface $(\pi \colon X \to C, S)$, whose associated sliced tree $\Gamma$ is as on the right.
    Then it should be interpreted as providing the following information.
    
    \noindent\begin{minipage}{0.8\textwidth}
    \begin{itemize}
        \item $C$ has $6$ irreducible components $C_v$, $v \in \{ a,b,c,d,f,h \}$, each isomorphic to $\bP^1$.
        \item $\jdeg(\Gamma) = 6$, hence $g \in \cK_6$, i.e. it is a degeneration of a Weierstrass fibration whose $j$-map has degree 72.        
        \item The degree of the $j$-map restricted to $C_v$ and divided by 12 is the label $\jdeg (v)$ of the vertex $v$.
    \end{itemize}
    \vspace{1 ex}
    \end{minipage}
    \begin{minipage}{0.19\textwidth}
    \flushright
    \begin{tikzpicture}[radius=2pt]
    \fill (0,0) circle node[above]{\small{1}} node[below]{\small{a}};
    \fill (1,0) circle node[above]{\small{4/3}} node[shift={(1.5mm,-2mm)}]{\small{b}};
    \fill (2,0) circle node[above]{\small{1/6}} node[below]{\small{c}};
    \fill (1,-1) circle node[shift={(-1.5mm,2mm)}]{\small{0}} node[below]{\small{d}};
    \fill (0,-1) circle node[above]{\small{2}} node[below]{\small{f}};
    \fill (2,-1) circle node[above]{\small{3/2}} node[below]{\small{h}};
    \draw (0,0) -- (1,0) -- (2,0);
    \draw[-|] (1,0) -- (1.5,0);
    \draw[-|] (1,-1) -- (1.5,-1);
    \draw[-|] (1,0) -- (1,-0.5);
    \draw (1,0) -- (1,-1);
    \draw (0,-1) -- (1,-1) -- (2,-1);
    \end{tikzpicture}
    \end{minipage}
    
    \noindent{}In this example, the sliced edges are $(bc)$, $(bd)$ and $(dh)$, where the slicings are $\left( \frac{1}{6}, \frac{5}{6} \right)$, $\left( \frac{1}{2}, \frac{1}{2} \right)$ and $\left( \frac{1}{2}, \frac{1}{2} \right)$ respectively.
\end{EG}

In order to describe boundary strata in the closure of the KSBA moduli space, we need to contract certain irreducible components of surfaces via \Cref{algorith: stable reduction}. This will reduce the number of vertices, but create singularities on the surfaces, which can be tracked combinatorially by attaching half-edges. Since the contraction process resembles the process of pruning trees, we name the corresponding object accordingly.

\begin{Def}\label{def_pruned_tree}
    A \emph{pruned tree} $\Pi=(V,E,E_0,\jdeg, F,T)$ is a sliced tree $(V,E,E_0,\jdeg)$ together with a set $K$, called \emph{klt-markings}, a set $L$ called \emph{lc-markings}, and the following additional structure. Each klt-marking and lc-marking is attached to one vertex, and similar to the notation in \cref{definition: sliced tree}, we denote the set of klt-markings (resp. lc-markings) attached to a vertex $v \in V$ by $T(v)$ (resp. $F(v)$).
    Each klt-marking $t \in T(v)$ attached to $v \in V$ is assigned a number
    \[ t_v \in
    \left\{ 
    \frac{1}{2}, \, \frac{1}{3}, \, 
    \frac{2}{3}, \, \frac{1}{4}, \,
    \frac{3}{4}, \, \frac{1}{6}, \,
    \frac{5}{6} \right\}.
    \]
\end{Def}
\begin{Remark}
    The numbers in \Cref{def_pruned_tree} are $(1-\lct \textup{of cuspidal fibers})$  in minimal Weierstrass fibrations, see \cite{ascher2017log}.
\end{Remark}

\begin{Def}
    The \emph{weight} of a vertex $v$ of a pruned tree is defined as follows:
   
    \[
    \wt(v)\ :=\ \# \{\text{edges adjacent to } v \} -2 +  \jdeg(v) + \sum_{t_v \in T(v)} t_v+ |F(v)|.
    \]
\end{Def}

\begin{Def}
    A pruned tree $\Pi = (V,E,E_0,\jdeg, F,T)$ is \emph{stable} if each vertex $v \in V$ has $\wt(v) > 0$.
\end{Def}

\begin{Remark}
    We now explain the reason behind the previous definition. We will associate to each KSBA-limit, as in \Cref{subsection_from_tsm_to_KSBA}, a pruned tree $\Pi$. More explicitly, if $(\sX,\sS)\to \spec(R)$ is a one-parameter family of tsm-stable elliptic surfaces, whose generic fiber is a Weierstrass fibration, we explained in \Cref{subsection_from_tsm_to_KSBA} how to take the relative canonical model of $(\sX,\epsilon\sS)\to \spec(R)$. If $\Gamma$ is the sliced tree associated to $(\sX_0, \sS_0)$, each step of the algorithm flips and contracts a component of $(\sX_0, \sS_0)$ and that corresponds to removing a leaf of $\Gamma$ and replacing it with an lc cusp or a klt cusp.
    From \Cref{algorith: stable reduction}, the algorithm terminates when each irreducible component of such a canonical model will satisfy that $(K_X+\sum F_i).S>0$, and from \Cref{prop_can_bundle_formula} this translates into a condition that $\wt(v) > 0$ for every vertex in the resulting graph. Moreover, the lc markings (resp. klt markings) corresponds to strictly lc cusps (resp. klt cusps).
\end{Remark}

The discussion in \Cref{algorith: stable reduction} leads to the following \emph{pruning algorithm} for stable sliced trees.

\begin{Prop}\label{prop_pruning_tree}
    Let $\bar{\Phi}_n \colon \cK_n^{\nu}\rightarrow \overline{\cE}_{n}$ be as at the beginning of this subsection. Let $\Gamma = (V,E,E_0,\jdeg)$ be a stable sliced tree parametrizing a stratum $\mathcal{K}^\nu_\Gamma$ in $\cK_n^{\nu}$. Let $\Pi$ be obtained from $\Gamma$ by the following process of \emph{pruning}:
    \begin{enumerate}
        \item Set $\Pi = \Gamma$ with $F = T = \varnothing$.
        \item \label{item: pruning step}
        If there is a leaf $v \in V$ of weight $\leq 0$, denote by $w$ the vertex adjacent to it, and do the following: replace $V$ by $V \setminus\{v\}$, and let $t := \wt(v) + 1$.
        \begin{enumerate}
        \item If $t=1$ then replace $F(w)$ by $F(w)\cup \{v\}$;
        \item if $t=0$, do nothing; otherwise
        \item replace $T(w)$ by $T(w)\cup \{  \wt(v) + 1\}$.
        \end{enumerate}
        If there are no leaves of weight $\le 0$, stop.
        \item Repeat (\ref{item: pruning step}).
    \end{enumerate}
    The resulting pruned tree, denoted $\bar{\Phi}_n(\Gamma) := \Pi$, is stable as a pruned tree. Moreover, the image of the stratum $\bar{\Phi}_n(\mathcal{K}^\nu_\Gamma)$ parametrizes surface pairs $(X,S)$ with sliced tree given by $\Pi$ and whose components have lc cusps (resp. klt cusps) of given type indexed by $F(w)$ (resp. $T(w))$. 
\end{Prop}

\noindent\begin{minipage}{0.79\textwidth}
\begin{EG}
    Let $\Gamma$ be the sliced tree exhibited in \cref{example: explicit sliced tree}. Then the pruned tree $\bar{\Phi}_6 (\Gamma)$ is the one on the right. Every leaf of the pruned tree on the right has positive weight, hence it is stable and does not need to be pruned any further.
\end{EG}
\end{minipage}
\begin{minipage}{0.2\textwidth}
\flushright
\begin{tikzpicture}[radius=2pt]
\fill (1,0) circle node[above]{\small{4/3}} node[shift={(1.5mm,-2mm)}]{\small{b}};
\fill (1,-1) circle node[shift={(-1.5mm,2mm)}]{\small{0}} node[below]{\small{d}};
\fill (0,-1) circle node[above]{\small{2}} node[below]{\small{f}};
\fill (2,-1) circle node[above]{\small{3/2}} node[below]{\small{g}};
\draw [-{Rays[]}] (1,0) -- (0.5,0);
\draw (1,0) -- (1.5,0);
\draw[-|] (1,-1) -- (1.5,-1);
\draw[-|] (1,0) -- (1,-0.5);
\draw (1,0) -- (1,-1);
\draw (0,-1) -- (1,-1) -- (2,-1);
\end{tikzpicture}
\end{minipage}

\begin{EG}
We now give another example of pruning, where each squiggly arrow corresponds to pruning all outermost leaves once. The sliced edges have coefficients $\left(\frac{1}{2},\frac{1}{2}\right)$, $\left(\frac{1}{2},\frac{1}{2}\right)$, $\left(\frac{1}{3},\frac{2}{3}\right)$, $\left(\frac{1}{6},\frac{5}{6}\right)$ respectively, the edges labelled with $(y)$ represent a klt marking with marking $y$, whereas the edges labelled with an x represent lc markings. 

\vspace{1ex}

{ \begin{center}
\noindent\begin{tikzpicture}[radius=2pt, decoration=snake]
\fill (0,0) circle node[above]{\small{1/2}};
\fill (1,0) circle node[above]{\small{1/2}};
\fill (2,0) circle node[above]{\small{1}};
\fill (3,0) circle node[above]{\small{3/2}};
\fill (4,0) circle node[above]{\small{1/6}};
\fill (5,0) circle node[above]{\small{1/6}};
\fill (6,0) circle node[above]{\small{1/6}};
\draw (0,0)--(6,0);
\draw[-|] (0,0) -- (.5,0);
\draw[-|] (3,0) -- (3.5,0);
\draw[-|] (4,0) -- (4.5,0);
\draw[-|] (5,0) -- (5.5,0);
\draw[->, decorate] (7,0) -- (8,0);
\end{tikzpicture}
\hspace{0.8 cm}
\begin{tikzpicture}[radius=2pt, decoration=snake]
\fill (0,-0.25) node[above]{(\small{1/2})};
\fill (1,0) circle node[above]{\small{1/2}};
\fill (2,0) circle node[above]{\small{1}};
\fill (3,0) circle node[above]{\small{3/2}};
\fill (4,0) circle node[above]{\small{1/6}};
\fill (5,0) circle node[above]{\small{1/6}};
\fill (6,-0.25) node[above]{(\small{1/6})};
\draw (0.5,0)--(5.5,0);
\draw[-|] (3,0) -- (3.5,0);
\draw[-|] (4,0) -- (4.5,0);
\end{tikzpicture}

\vspace{3ex}

\noindent\begin{tikzpicture}[radius=2pt, decoration=snake]
\draw[->, decorate] (-.4,0) -- (.6,0);
\fill (2,0) circle node[above]{\small{1}};
\fill (3,0) circle node[above]{\small{3/2}};
\fill (4,0) circle node[above]{\small{1/6}};
\fill (4.9,-0.25) node[above]{(\small{1/3})};
\draw (2,0) -- (4,0) ;
\draw[-|] (3,0) -- (3.5,0);
\draw[-{Rays[]}] (2,0) -- (1.6,0);
\draw (4,0) -- (4.4,0);
\draw[->, decorate] (5.4,0) -- (6.4,0);
\end{tikzpicture}
\hspace{.8 cm}
\noindent\begin{tikzpicture}[radius=2pt, decoration=snake]
\fill (2,0) circle node[above]{\small{1}};
\fill (3,0) circle node[above]{\small{3/2}};
\fill (4,-0.25) node[above]{(\small{1/2})};
\draw  (2,0) -- (3,0);
\draw[-{Rays[]}] (2,0) -- (1.6,0);
\draw (3,0) -- (3.4,0);
;
\end{tikzpicture}
\end{center} }
\end{EG}
\begin{proof}[Proof of \Cref{prop_pruning_tree}]
As mentioned above, the klt-markings (resp. lc-markings) of a given vertex, correspond to the klt (resp. lc) cuspidal fibers in the corresponding irreducible components.
From \Cref{prop_can_bundle_formula}, the weight of a vertex $v$ on a pruned tree corresponds to the intersection number $(K_{\sX^{(i)}}.\ S_v)$, where $S_v$ is the irreducible component of $\sS^{(i)}_{0}$ contained in the irreducible component of $\sX_{0}$ corresponding to the vertex $v$.
Therefore from \Cref{algorith: stable reduction}, the only thing we need to check is the following.
Let $v$ be a leaf with non-positive weight, attached to a vertex $w$, and let $X_v^-$ be the irreducible component corresponding to $v$ with $X_{w}^-$ the irreducible component attached to it.
Then after a flip of La Nave and a contraction of the pseudo-elliptic component, if we denote by $X_v^+$ the proper transform of $X_j^-$, and by $F$ the cuspidal fiber to which the pseudo-elliptic surface given by the proper transform of $X_v^+$ is contracted, then we claim that
\[
1-\lct\big(X_{w}^{\cusp}; F\big) \ =\ \wt(i) + 1
\]
where $X_{w}^{\cusp}$ is the proper transform of $X_{w}^-$ containing $F$.

\begin{figure}[H]
\centering
\includegraphics[width=12cm]{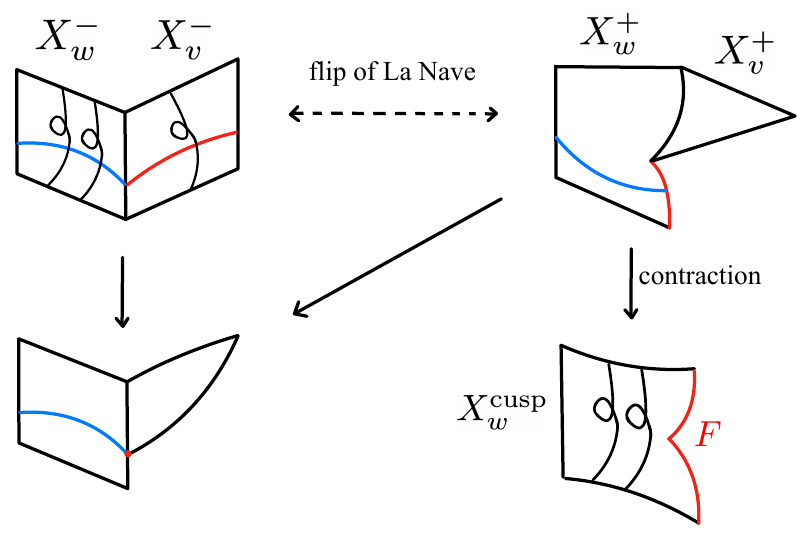}
\caption{Notations as in the proof of \Cref{prop_pruning_tree}.}
\end{figure}

This is because from \Cref{lemma_lct_Asquare}, the log-canonical threshold is given by the reciprocal of the self-intersection of a certain intermediate component, which from \Cref{lemma_which_intermediate_fiber_after_a_flip} is given (up to sign) by $(S_v^2)$, where $S_v \subseteq X_v$ is the section component.
Since $X_v$ is attached only along a single fiber, the desired statement now follows from \Cref{prop_can_bundle_formula}.
\end{proof}

\begin{Cor}\label{cor:weight}
 The value $\wt(\Gamma) := \sum_v \wt(v)$ remains constant during the pruning process.
\end{Cor}

\begin{proof}
       This follows from how the algorithm of \Cref{prop_pruning_tree} works.
\end{proof}

\begin{Def}
    The \emph{height} of a stable pruned tree is defined by 
    $$
    \height(\Pi) = \sum_{i \in V(\Pi)}\left( \jdeg(i) + \sum_{t_j \in T(i)} t_j + |F(i)|\right).
    $$    
\end{Def}

\begin{Cor}\label{cor:height}
    Let $\Pi = \bar{\Phi}_n(\Gamma)$. Then 
    $$\height(\Pi) \ =\  \wt(\Pi) + 2 \ =\  \jdeg(\Gamma).$$
\end{Cor}
\begin{proof}
    By Corollary \ref{cor:weight}, we have $\wt(\Pi) = \wt(\Gamma)$. Then we can compute 
        \begin{align*}
        \wt(\Pi)&\ =\ \wt(\Gamma)\\
        &\ =\ \sum_{i \in V(\Gamma)}\big ( \#\{\text{edges containing }i\}-2+\jdeg(i)\big)\\
        & \ =\ \sum_{i \in V(\Gamma)}\big ( \#\{\text{edges containing }i\}-2\big) + \jdeg(\Gamma) \\
        &\ =\  \jdeg(\Gamma) - 2,
    \end{align*}
    where we have used that the Euler characteristic of $\Gamma$, defined by $\#\{\text{vertices}\} - \#\{\text{edges}\}$, is equal to $1$ since $\Gamma$ is a tree. This proves the second equality. For the first equality, using the definition of $\wt(v)$ and $\height(\Pi)$, we similarly have
    \begin{align*}
        \wt(\Pi)&\ =\ \height(\Pi) + \sum_{v \in \Pi} + \sum_{v \in V(\Pi)}\big ( \#\{\text{edges containing }v\}-2\big)\\
        & \ =\ \height(\Pi) - 2.
    \end{align*}
\end{proof}

\begin{Cor}\label{cor:heightn}
    The stable pruned tree associated to any KSB-stable surface parametrized by $\overline{\cE}_{n}$ has height $n$.
\end{Cor}

\begin{proof} By Corollary \ref{cor:height}, $\height(\Pi) = \jdeg(\Gamma)$ for any stable sliced tree $\Gamma$ which prunes to $\Pi$. On the other hand, $\jdeg(\Gamma)$ is simply the degree of the $j$-map divided by $12$ of any twisted stable map $\mathcal{C} \to \overline{\mathcal{M}}_{1,1}$ with sliced graph $\Gamma$. This degree is constant in families of twisted stable maps and equal to the height $n$ for any stable map in the interior of $\mathcal{K}_n$, i.e. an elliptic surface with at worst nodal fibers.
\end{proof}

Since $\cK^\nu_n\to \overline{\cE}_{n}$ is surjective, every point of $\overline{\cE}_{n}$ is in the image of some stratum. Thus, the combinatorial data of any surface pair $(X,S)$ parametrized by $\overline{\cE}_{n}$ is encoded by a stable pruned graph $\bar{\Phi}_n(\Gamma) = \Pi$ for some stable sliced graph $\Gamma$. This yields the following description of the boundary of $\overline{\cE}_{n}$. 

\begin{theorem}\label{thm:stratificationE_n}
There is a constructible stratification 
 $$
 \overline{\cE}_{n} = \bigsqcup_{\height(\Pi) = n} \mathcal{E}_{\Pi}
 $$
indexed by pruned trees of height $n$, where the constructible stratum $\mathcal{E}_\Pi$ parametrizes smoothable KSBA-stable elliptic surface pairs $(X, \epsilon S)$ for $0 < \epsilon \ll 1$ with combinatorial type $\Pi$.
\end{theorem}

\begin{proof}
Let $\mathcal{E}_\Pi$ be the locus parametrizing surface pairs with combinatorial type $\Pi$. By Corollary \ref{cor:heightn}, $\height(\Pi) = n$. Thus this stratum is simply the union of images under $\bar{\Phi}_n$ of strata in $\mathcal{K}_n$: 
$$
\mathcal{E}_\Pi = \bigcup_{\bar{\Phi}_n(\Gamma) = \Pi} \bar{\Phi}_n(\mathcal{K}_\Gamma).
$$
By Chevalley's Theorem, $\mathcal{E}_\Pi$ are constructible sets, and by surjectivity of $\bar{\Phi}_n$, they cover $\overline{\mathcal{E}}_n$. 
\end{proof}

\section{Wall-crossing morphisms for $\overline{\cE}_{n,t}$ and $\overline{\cP}_n$}\label{sec:wall-crossing-iso}
In this section, we study the wall-crossing morphisms between the moduli spaces $\overline{\cE}_{n,t}$ for different choices of $t$ and $\overline{\cP}_n$ in order to prove that $\rho : \overline{\cE}_{n,0} \to \overline{\cE}_n$ and $\bar{\Psi}_n : \overline{\cE}_{n,0} \to \overline{\cP}_n$ in Proposition \ref{proposition: maps between compactifications} are isomorphisms. Let $c_0 = \frac{n-2}{n}$, and we begin with the following result.
\begin{Lemma}\label{cor_lemma_no_div_contractions}
    Let $(\sX^c,\epsilon \sS^c)\to \spec(R)$ be a family corresponding to a morphism $\spec(R)\to \overline{\cE}_{n}$ from the spectrum of a DVR, with smooth generic fiber. Then
    \begin{enumerate}
        \item the pair $(\sX^c,t\sS^c + \sX_0^c)$ is log canonical, where we marked the divisor $\sS^c$ with coefficient $0 < t \le 1$;
        \item the MMP that yields the canonical model $(\sX^c,(c_0 - \epsilon)\sS^c) \dashrightarrow (\sX', (c_0 - \epsilon)\sS')$ for $0 < \epsilon \ll 1$ contracts no divisors; and
        \item the canonical model $(\sX', c_0 \sS') \to \sY$ contracts only the divisor $\sS'$. 
    \end{enumerate}
    In particular, there is a bijection between the irreducible components of $\sX_0^c$, those of $\sX'_0$, and those of $\sY_0$.
\end{Lemma}
\begin{proof}
    We first prove (1). Recall that the family $(\sX^c,\epsilon \sS^c)\to \spec(R)$ is the canonical model of a threefold pair $(\sX,\epsilon \sS)\to \spec(R)$ which is a family of tsm-stable elliptic surfaces. Moreover, the threefold $\sX^c$ admits a morphism $\xi\colon \sX^c\to \sC$ to a family of nodal curves with pure 1-dimensional fibers \cite[Thm. 1.2]{inchiostro2020moduli}. The proof of point (1) is by induction on the steps of \Cref{algorith: stable reduction}. We argue that at each step of \Cref{algorith: stable reduction}, the pair $(\sX^{(i)},\sS^{(i)}+\sX^{(i)}_0)$ is lc and $\sS^{(i)}$ is a $\mathbb{Q}$-Cartier divisor.

    By \cite[Prop. 4.3]{tsm}, this is true at the beginning of the algorithm, i.e., the tsm-stable pair $(\sX,\sS)$ is lc, and the section is the coarse space of a Cartier divisor on the stack-like elliptic surface. The only steps of the algorithm which replace a neighborhood of $\sS^{(i)}$ are the flips of La Nave. So assume we have a flip of La Nave 
    \[
(\sX^{(i)},\sS^{(i)})\ \dashrightarrow \ (\sX^{(i+1)},\sS^{(i+1)}).
    \] It is proved in the moreover part of \cite[Thm. 7.1.2]{lanave2002explicit} that the pair $(\sX^{(i+1)},\sS^{(i+1)})$ is log smooth around $\sS^{(i+1)}$ in a neighborhood of the flipped curve. So also $(\sX^{(i+1)},\sS^{(i+1)})$ is lc and $\sS^{(i+1)}$ remains $\mathbb{Q}$-Cartier; this finishes (1).

    We now prove (2) and (3). Notice that the generic fiber $(\sX^c_\eta,(c_0-\epsilon)\sS_\eta^c)$ is stable for every $0<\epsilon \ll 1$, and the canonical model of $(\sX^c_\eta,\frac{n-2}{n}\sS_\eta^c)$ contracts only $\sS_\eta^c$. We first check that the canonical model \[\left(\sX^c,\left(c_0-\epsilon\right)\sS^c\right) \dashrightarrow \left(\sX',\left(c_0-\epsilon\right)\sS'\right)\] does not contract divisors. This follows from the fact that the KSBA-moduli space is separated, and the fact that the MMP does not extract divisors. Indeed, the pair $\left(\sX',\left(c_0-\epsilon'\right)\sS'\right)$
    is locally stable and $\sS'$ is $\mathbb{Q}$-Cartier, so if we replace the coefficient of $\sS'$ with $\epsilon$, it remains locally stable. Since the KSBA-moduli stack is separated, the canonical model of $\left(\sX',\epsilon\sS'\right)$ is $(\sX^c,\epsilon\sS^c)$; thus,
 there is a birational contraction $\sX'\dashrightarrow \sX^c$. The composition of the two birational contractions $\sX^c\dashrightarrow \sX'$ and $\sX'\dashrightarrow \sX^c$ constructed above is the identity, so $\sX^c\dashrightarrow \sX'$ does not contract divisors. This proves $(2)$.

    We now check that the canonical model $\left(\sX',c_0\sS'\right)\to \sY$ contracts only $\sS'$. As $\epsilon$ was chosen small so that $c_0 - \epsilon$ is in the chamber below the wall $c_0$, then $(\sX', c_0 \sS')$ is a good minimal model and $K_{\sX'}+c_0\sS'$ is semi-ample. A divisor is contracted by $\sX'\to \sY$ if and only if every curve $A$ on it satisfies that $$\textstyle \big(K_{\sX'}+\frac{n-2}{n}\sS'\big).A=0.$$ Consider a divisor
    $\Gamma'\subseteq \sX'$ different from $\sS'$. As $\sX^c\dashrightarrow\sX'$ does not contract divisors, let $\Gamma^c\subseteq \sX^c$ be the proper transform of $\Gamma'$. Let $p\colon W\to \sX^c$ and $q\colon W\to \sX'$ be a common resolution, and let $D_W$
    be a (not necessarily effective) $\bQ$-divisor such that $q\colon (W,D_W)\to (\sX',\epsilon \sS')$ is crepant birational. Then
    \[
    K_W+D_W\sim_\bQ p^*(K_{\sX^c}+\epsilon \sS^c)+E
    \]
    with $E\ge 0$ and $p$-exceptional, by definition of canonical model. As $\sX'\dashrightarrow \sX^c$ does not contract divisors, $E$ is also $q$-exceptional. In particular, we can choose a curve $A^c\subseteq \Gamma^c$ such that:
    \begin{enumerate}
        \item $A^c$ is not contained in the image of $\operatorname{Supp}(E)$,
        \item $A^c$ is the image of a curve $A_W$ on $\Gamma_W$, the proper transform of $\Gamma^c$ in $W$, and $A':=p(A_W)$ is still a curve in $\Gamma'$ (i.e. $A_W$ is not contracted by $p$ or $q$),
         \item $A'$ is not contained $\sS'$.
    \end{enumerate}
     Then 
     \begin{equation}\nonumber\textstyle 
      \begin{split}
          A'.\left(K_{\sX'}+\frac{n-2}{n} \sS'\right)& \ge A'.(K_{\sX'}+\epsilon\sS')\\
          & = A_W.(K_W+D_W) \\&= A_W.p^*(K_{\sX^c}+\epsilon \sS^c) + A_W.E\\
          & \ge A_W.p^*(K_{\sX^c}+\epsilon \sS^c) \\ &= A^c.(K_{\sX^c}+\epsilon \sS^c)>0.
      \end{split} 
     \end{equation}
    The first inequality follows from the fact that $A'$ is not contained in $\sS'$, the first equality follows from projection formula, the second one by the definition of $E$, the second inequality follows as $A_W \nsubseteq E$, the last equality by the projection formula, and the last inequality is due to the ampleness of $K_{\sX^c}+\epsilon\sS^c$.
    Therefore, the divisor $\Gamma'$ is not contracted by taking the canonical model of $(\sX',\frac{n-2}{n}\sS')$.
\end{proof}
\begin{Cor}\label{cor_comb_description_of_surfaces_in_Pn}
    Let $(\sX,\epsilon \sS)\to \spec(R)$ a KSBA-stable limit of a Weierstrass fibration in $\cE_n$, and let $\sY\to \spec(R)$ be the KSB-stable limit of the surface $\sY_\eta$ obtained by contracting $\sS_\eta\subseteq \sX_\eta$. Then there is a bijection between the irreducible components of $\sX_0$ and those of $\sY_0$. Moreover, for any irreducible components $X_1,X_2$ of $\sX_0$ with proper transforms $Y_1$ and $Y_2$ in $\sY_0$, $X_1$ intersects $X_2$ along a curve if and only if $Y_1$ intersects $Y_2$ along a curve.
\end{Cor}
\begin{proof}
    The first statement follows immediately from \Cref{cor_lemma_no_div_contractions}. The second statement follows from the fact that $\sY$ is the canonical model of $(\sX,\sS)$, from \Cref{cor_lemma_no_div_contractions} and the uniqueness of KSB-stable limits. Moreover, to get this canonical model one has only to perform flips of La Nave from \cite[Appendix B]{ascher2017log} and \Cref{cor_lemma_no_div_contractions}, and possibly contracting some section component of the special fiber. Now the desired statement follows from how flips of La Nave are constructed.
\end{proof}

\begin{Cor}\label{cor_amm_moduli_spaces_are_isom}
    The wall-crossing morphisms $\rho : \overline{\cE}_{n,0} \to \overline{\cE}_n$ and $\bar{\Psi}_n : \overline{\cE}_{n,0} \to \overline{\cP}_n$ as in Proposition \ref{proposition: maps between compactifications} are isomorphisms. In particular, the moduli spaces $\overline{\cE}_{n,t}$ for $0<t<\frac{n-2}{n}$ and $\overline{\cP}_n$ are all isomorphic for $n \neq 4$.
\end{Cor}

\begin{proof}
    Recall that $\overline{\cE}_{n,t}$ (resp. $\overline{\cP}_n$) is the normalization of the closure of the image of $\Phi_{n,t}$ (resp. $\Psi_n$). There are wall-crossing morphisms $\rho_{t',t} \colon \overline{\cE}_{n,t}\to \overline{\cE}_{n,t'}$ for $t<t'$, given by reducing the coefficient on the divisor from $\frac{n-2}{n}-t$ to $\frac{n-2}{n}-t'$ by \cite[Thm. 1.10]{ascher2021wall}. Note that $\rho_{\epsilon, \frac{n-2}{n} - \epsilon} = \rho$ for $0 < \epsilon \ll 1$. We will show that $\rho_{t',t}$ is an isomorphism by constructing its inverse. Since these morphisms are all birational, it suffices to do so for $t=\epsilon$ and $t'=\frac{n-2}{n}-\epsilon$ for $0<\epsilon \ll 1$, i.e. that $\rho : \overline{\cE}_{n,0} \to \overline{\cE}_n$ is an isomorphism. 

    From \Cref{cor_lemma_no_div_contractions} part (1), if we denote by $(\sX,\epsilon\sS)\to \overline{\cE}_{n}$ the universal family, then \[\textstyle \left(\sX,(\frac{n-2}{n}-\epsilon)\sS\right) \ \longrightarrow \  \overline{\cE}_n\]
    is locally stable. Then one can take its canonical model over $\overline{\cE}_n$; see e.g. \cite[Thm. 1.1]{meng2023mmp}. This gives a morphism $\overline{\cE}_n\to \MKSBA$ to the KSBA-moduli stack with coefficients $\frac{n-2}{n}-\epsilon$, and such a morphism factors via $\overline{\cE}_n\to \overline{\cE}_{n,0}\to \MKSBA$ from the universal property of normalizations. This is the desired inverse of $\overline{\cE}_{n,0}\to \overline{\cE}_{n}$. 
    
    The morphism $\bar{\Psi}_n : \overline{\cE}_{n, 0}\to \overline{\cP}_n$  being an isomorphism is proved analogously to \cite[Thm. 1.9]{ascher2021wall}. Note we cannot apply \emph{loc. cit.} directly, but the same argument applies. Note that $\bar{\Psi}_n$ is proper as both source and target are proper, and the fibers of $\bar{\Psi}_n$ are countable by \cite[Lemmas 6.3 and 6.4]{ascher2021wall}, and thus finite. If $(X, (c_0 - \epsilon)S)$ is a pair parametrized by a point $p \in \overline{\cE}_{n,0}$, then $\bar{\Psi}_n(p)$ parametrizes the surface $Y$ obtained by contracting $S$ to a point. Then any automorphism of $(X, S)$ which induces the identity on $Y$ must be the identity on the dense open set $X \setminus S \subset X$ and thus is the identity. We conclude that $\Aut(p) \to \Aut(\bar{\Psi}_n(p))$ is injective so $\bar{\Psi}_n$ is representable. 
    
    For $n \neq 4$,  $\bar{\Psi}_n$ is birational for as it extends the open embedding $\Psi_n$. For $n = 4$ note that $\Psi_4$ induces a set-theoretic bijection between smooth height $4$ elliptic fibrations and height $4$ pseudoelliptic surfaces with exactly one $\frac{1}{4}(1,1)$ singularity. In characteristic $0$, a representable bijection between normal algebraic stacks is birational. Applying Zariski's Main Theorem completes the proof. 
    
\end{proof}

    In the remaining part of this section, we will not distinguish  the moduli spaces $\overline{\cE}_{n,t}$ and $\overline{\cE}_{n}$.

\subsection{Boundary stratification of $\overline{\cP}_n$}

Putting together \Cref{thm:stratificationE_n} with \Cref{cor_amm_moduli_spaces_are_isom}, we obtain a boundary stratification of $\overline{\cP}_n$ from $\overline{\cE}_n$. 

\begin{theorem}\label{thm:boundary_P_n}
    There is a constructible stratification 
    $$
    \overline{\cP}_n = \bigsqcup_{\height(\Pi) = n} \cP_\Pi
    $$
    indexed by stable pruned trees of height $n$. The stratum $\cP_{\Pi}$ parametrizes trees of pseudoelliptic surfaces glued together along pseudofibers with combinatorics described by $\Pi$. 
\end{theorem}

\begin{proof} 
Let $(\sX, \epsilon \sS) \to \overline{\cE}_n$ be the universal family of $\epsilon$-weighted stable elliptic surfaces for $0 < \epsilon \ll 1$. By \Cref{cor_lemma_no_div_contractions}, the rational map $\sX \dashrightarrow  \sX^c$ to the canonical model of $(\sX, \frac{n-2}{n}\sS)$ which induces the isomorphism $\overline{\cE}_n \cong \overline{\cP}_n$ consists of a sequence of flips followed by the single divisorial contraction of the section. By \Cref{algorith: stable reduction}, the only flips are the flips of La Nave which transform elliptic components corresponding to leaves of $\Pi$ into pseudoelliptic components. Then the final contraction of the section contrats any remaining elliptic components into pseudoelliptics. By \Cref{cor_comb_description_of_surfaces_in_Pn}, this doesn't change the combinatorics of how the pseudoelliptic components are glued together. Thus under this isomorphism, the image of the stratum $\cE_\Pi$, which we denote $\cP_\Pi$, parametrizes trees of pseudoelliptic surfaces glued along pseudofibers, with combinatorics described by $\Pi$. 
\end{proof} 

A consequence of \Cref{subsection_comb} is the following.

\begin{Cor}\label{cor_KSB_limits_have_at_most_6_vertices}
    Let $X$ be a KSBA-stable surface parametrized by $\overline{\cE}_3$, and $\Pi$ be the pruned tree associated to $X$. Then $\Pi$ has two leaves (i.e. $\Pi$ is a chain), and has at most six vertices. The same is true $\overline{\cP}_3$. 
\end{Cor}

Observe that one can construct an element with six irreducible components. It suffices to construct a twisted stable map with the following tsm-stable graph $\Gamma$, and with slicings $\left ( \frac{5}{6},\frac{1}{6}\right ),\left ( \frac{2}{3},\frac{1}{3}\right ), \left ( \frac{1}{2},\frac{1}{2}\right ), \left ( \frac{1}{3},\frac{2}{3}\right ), \left ( \frac{1}{6},\frac{5}{6}\right )$ in order from left to right.
\begin{center}
\begin{tikzpicture}[radius=2pt, decoration=snake]
\fill (1,0) circle node[above]{\small{7/6}};
\fill (2,0) circle node[above]{\small{1/6}};
\fill (3,0) circle node[above]{\small{1/6}};
\fill (4,0) circle node[above]{\small{1/6}};
\fill (5,0) circle node[above]{\small{1/6}};
\fill (6,0) circle node[above]{\small{7/6}};
\draw (1,0)--(6,0);
\draw[-|] (1,0) -- (1.5,0);
\draw[-|] (2,0) -- (2.5,0);
\draw[-|] (3,0) -- (3.5,0);
\draw[-|] (4,0) -- (4.5,0);
\draw[-|] (5,0) -- (5.5,0);
\end{tikzpicture}
\end{center}
Such a twisted stable map exists, or in other terms, the stratum in $\cK_{3}$ corresponding to the graph above is not empty; we briefly sketch how to construct it. First, for each vertex $v$ of the diagram above, we consider a map from a root-stack of $\bP^1\to \overline{\cM}_{1,1}$ such that the corresponding map on coarse moduli spaces has degree $12\jdeg(v)$.
Then we glue the corresponding maps along the stacky points, in a way so that the resulting morphism is balanced. 

For example, for the first two vertices on the left, one can proceed as follows. For the leftmost vertex, we start by considering an elliptic K3 which we denote by $Y\to \bP^1$, with a single klt cusp, and from the slicing and \Cref{table: sliced weights}, is of type $\operatorname{II}^*$. From \cite{Miranda}*{Table (IV).3.1}, one can construct it explicitly via its Weierstrass equation, taking two polynomials \[A\in \oH^0(\cO_{\bP^1}(8)), B\in \oH^0(\cO_{\bP^1}(12))\]
which have all distinct and single roots, except for a single point $p\in \bP^1$ where $A$ has a zero of multiplicity 4, and $B$ has multiplicity 5. This is clearly possible, and the two polynomials $A$ and $B$ will induce a morphism $\bP^1\to [\bA^2/\bG_m]$, with the action with weights $4$ and $6$, inducing $f: \bP^1\smallsetminus\{p\}\to \overline{\cM}_{1,1}$. By \cite{bejleri2022height}*{Thm. 1.6 \& Thm. 3.3} (see also \cite{BV}), one can construct a root stack $\sP^1\to \bP^1$ so that $f$ extends to $\sP^1\to \overline{\cM}_{1,1}$.
We can check that $\sP^1$ has $\bmu_6$ as automorphism group on the stacky point. Indeed, the two sections $A$ and $B$ give a map $\bP^1\to [\bA^2/\bG_m]$ which locally is of the form $z\mapsto (z^4,z^5)$, and its $\bG_m$-orbit in $\bA^2$ is of the form \[\bG_m\times \bP^1\smallsetminus \{p\}\to \bA^2, (\lambda, z)\mapsto (\lambda^4 z^4, \lambda^6 z^5).
\] 
Extending this map to $\sP^1$ boils down to replacing $\lambda$ with $z^{\frac{m}{d}}$, so that the two sections $(z^{4 + \frac{4m}{d}},z^{5 + \frac{6m}{d}}) $ do not have a pole at $p$ and do not vanish simultaneously. The smallest positive $d$ that one can take will lead to a representable morphism $\sP^1\to \overline{\cM}_{1,1}$, and it is $d=6$ with $m=-5$; so $\sP^1$ will have a $\bmu_6$ as stabilizer group on the stacky point.
Since the morphism $\sP^1\to \overline{\cM}_{1,1}$ is representable, such a stacky point will go to the only point of $\overline{\cM}_{1,1}$ with $\bmu_6$ as automorphisms.
Similarly, for the second leftmost vertex, one can consider two homogeneous polynomials $A,B$ of degree 4 and 6 respectively, which have distinct and single roots on $\bP^1$, except at two points $p_1$ and $p_2$. From the slicings, \Cref{table: sliced weights} and \cite{Miranda}*{Table (IV).3.1}, we require that at $p_1$ the polynomials $A$ and $B$ must vanish of order 1 (this will correspond to a type $\operatorname{II}$ fiber), and at $p_2$ instead $A$ should vanish with multiplicity 3 and $B$ with multiplicity 4.
This is again possible, and proceeding as before this will lead to a stacky $\bP^1$, which we denote by $\sP^1_2$, with two stacky points. As before, one can check that $\sP^1_2$ has, at $p_1$, automorphism group which is $\bmu_6$, and at $p_2$ which is $\bmu_3$. We can glue $\sP^1$ and
$\sP^1_2$ along the point with automorphism $\bmu_6$, and the corresponding maps to $\overline{\cM}_{1,1}$ will glue. On the glued point, one can check that we will have a twisted vertex. We can proceed with this recipe to get the desired twisted map $\cC\to \overline{\cM}_{1,1}$.
\begin{proof}[Proof of \Cref{cor_KSB_limits_have_at_most_6_vertices}]
 As a consequence of Corollary \ref{cor:heightn}, one has
\begin{equation}\label{eq_weights}
    \sum_{i \in V(\Pi)} \bigg(\jdeg(i) + \sum_{t_j \in T(i)}t_j +|F(i)|\bigg)\ =\ 3. 
\end{equation}
    
    Let us prove that $\Pi$ can have at most two leaves. If not, let $\ell,j,k$ three leaves. 
    The stability condition guarantees that 
    \begin{align*}
        &\jdeg(\ell) + \sum_i |F(\ell)| + \sum_{i\in T(\ell)}t_{i} \ >\ 1, \\& \jdeg(j) + \sum_i |F(j)| + \sum_{i\in T(j)}t_{i}  \ >\  1\\
        & \jdeg(k) + \sum_i |F(k)| + \sum_{i\in T(k)}t_{i} \ >\ 1.
    \end{align*}
    Since the graph is connected, there is no edge between $i,j$ and $\ell$. Then if we add up the previous inequalities, we get a contradiction of \Cref{eq_weights}. We conclude that $\Pi$ has at most two edges, and thus $\Pi$ is a chain. We now prove that there cannot be more than four internal vertices in the chain.

    Let us label the edges on the chain as follows: let $l_1, l_2$ be the leaves and $n_1,\dots,n_k$ the internal vertices. We know that $$\jdeg(l_i) + \sum_i |F(l_i)| + \sum_{i\in T(\ell)}t_{l_i} \ >\ 1$$ for each leaf $i=1,2$, and from how the numbers $t_i$ are defined, we have that
    \[
    \jdeg(l_i) + \sum_i |F(l_i)| + \sum_{i\in T(l_i)}t_{l_i}  \ \ge\ \frac{7}{6}.
    \]
    Similarly for $1\le i\le k$,
    \[
    \jdeg(n_i) + \sum_i |F(n_i)| + \sum_{i\in T(n_i)}t_{n_i}  \ \ge\ \frac{1}{6}.
    \]
    Combining this with \Cref{eq_weights}, we get that $k\le 4$.
\end{proof}

\subsubsection{Non-emptiness of strata}\label{sec:non-empty} In this subsection, we discuss the question of non-emptiness of the strata $\mathcal{P}_{\Pi}$ and $\mathcal{E}_{\Pi}$. The idea is exactly as in the example following \Cref{cor_KSB_limits_have_at_most_6_vertices}. Since the strata are isomorphic we focus on $\mathcal{E}_\Pi$. As shown in Remark \ref{rem:empty}, there exist stable sliced graphs $\Gamma$ for which $\mathcal{K}_{\Gamma} = \emptyset$.
On the other hand, by definition, $\mathcal{E}_\Pi \neq \emptyset$ if and only if there exists some stable sliced graph $\Gamma$ which prunes to $\Pi$ with $\mathcal{K}_\Gamma \neq \emptyset$.
This reduces the question of non-emptiness of strata of \(\cE_{\Pi} \subset \overline{ \mathcal E}_n \) to the same question for strata of $\mathcal{K}_n$. 

Let $\Gamma$ be a stable sliced graph of $\jdeg = n$. Each vertex $v \in V(\Gamma)$ should correspond to a genus $0$ pointed twisted map $f_v : (\mathcal{C}_v, \{p_e\}_{e \in E(v)}) \to \overline{\mathcal{M}}_{1,1}$ of degree $\jdeg(v)$ from a smooth twisted curve $\mathcal{C}$. The stacky points of $\mathcal{C}$ are exactly $\{p_e\}_{e \in E_0(v)}$ corresponding to the sliced edges, and the stabilizer at $p_e$ is $\mu_d$ where $d$ is the denominator of the slicing $e_v$. The numerator of the slicing describes the Kodaira fiber type of the fibers $\{F_e\}_{e \in E_0(v)}$ of the tsm-stable surface $(X_v \to C_v, S + \sum_{e \in E_0(v)} F_e^{\red} + \sum_{e \in E(v) \setminus E_0(v)} G_e)$ obtained by pulling back the universal family along $f_v$ and taking the coarse moduli space. The fibers $G_e$ for $e \notin E_0(v)$ are simply stable fibers over non-stacky points of $\mathcal{C}$. 

Suppose that for each $v \in \Gamma$, there exists a pointed genus $0$ twisted map $f_v$ as above. Then we can glue these maps together along marked stacky points with the same $j$-invariant according to the sliced graph $\Gamma$ to obtain a map $f : \mathcal{C} \to \overline{\mathcal{M}}_{1,1}$ with $\jdeg$ being $\jdeg(\Gamma) = n$. The slicing condition exactly guarantees that the resulting twisted map is balanced and by the smoothability theorem, Theorem \ref{thm: irreducibility of K_n}, the twisted stable map $f$ smooths into the interior of $\mathcal{K}_n$. Thus the associated tsm-stable elliptic surface $(X \to C, S)$ (which is obtained from gluing together the surfaces $(X_v, S + \sum_{e \in E_0(v)} F_e^{\red} + \sum_{e \in E(v) \setminus E_0(v)} G_e)$ along marked fibers according to the sliced graph $\Gamma$) is the limit of an honest height $n$ elliptic surface over $\mathbb{P}^1$. By construction the stable sliced tree of $(X \to C, S)$ is $\Gamma$. Now we apply the minimal model program via the Algorithm \ref{algorith: stable reduction} to a one parameter smoothing of this surface $(X \to C, S)$ to obtain its KSBA-stable model $(X^c, \epsilon S^c)$ for $0 < \epsilon \ll 1$. By Proposition \ref{prop_pruning_tree}, this is a KSBA-stable surface with pruned tree $\overline{\Phi}_n(\Gamma) =\Pi$. 

Now to show that $\mathcal{K}_\Gamma$ is nonempty, it suffices to construct for each $v \in V(\Gamma)$, a twisted map $f_v$ satisfying the conditions above. For this we use the connection with minimal Weierstrass equations as developed in \cite{bejleri2022height}. By \cite[Prop. 5.8, Cor. 7.5, Prop. 7.8 \& Thm 7.12]{bejleri2022height}, a genus $0$ twisted map as in $f_v$ is equivalent to Weierstrass data $(A,B)$ on $\mathbb{P}^1$ of height $n$ with vanishing conditions at points $\{p_e\}_{e \in E_0(v)}$ described by the slicing $e_v$ via the table in Theorem 1.6 of \emph{loc. cit.}. Here the slicing $e = a/r$ for a pair $(r,a)$ in that table. Concretely, $A$ and $B$ are simply homogeneous polynomials of degree $4n$ and $6n$ which simultaneously vanish at a point $p_e \in \mathbb{P}^1$ to an order determined by the slicing $e_v$ and which do not simultaneously vanish elsewhere. The required twisted map exists if and only if we can find such polynomials, which is not always the case is shown in Remark \ref{rem:empty}. This reduces the question of which strata are non-empty to an elementary question about the existence polynomials satisfying the given vanishing conditions.

\bibliographystyle{amsalpha}
\bibliography{bibliography}

\end{document}